\documentclass[12pt]{amsart}
\usepackage{amssymb}
\usepackage{epsfig}
\usepackage{amsfonts}
\usepackage{amscd,color}
\usepackage{tikz}
\usetikzlibrary{cd}
\usepackage{hyperref}
\usepackage[capitalize,nameinlink]{cleveref}
\usepackage[utf8]{inputenc}
\usepackage{youngtab}

\usepackage[all]{xy}
\usepackage{xypic,amscd,epsfig}
\usepackage{relsize}

\newcommand{\Waff}{W_{\textrm{aff}}}

\newcommand{\N}{\mathbb{N}}
\newcommand{\Z}{\mathbb{Z}}
\newcommand{\Q}{\mathbb{Q}}

\newcommand{\C}{\mathbb{C}}
\newcommand{\fS}{\mathfrak{S}}

\newcommand{\bP}{\mathbb{P}}
\newcommand{\cO}{\mathcal{O}}
\newcommand{\cS}{\mathcal{S}}
\newcommand{\cQ}{\mathcal{Q}}

\newcommand{\GL}{\,\mathrm{GL}}
\newcommand{\QH}{\,\mathrm{QH}}
\newcommand{\QK}{\,\mathrm{QK}}
\newcommand{\K}{\,\mathrm{K}}
\newcommand{\Fl}{\,\mathrm{Fl}}
\newcommand{\Gr}{\,\mathrm{Gr}}

\newcommand{\Kpt}{\mathrm{Rep}_{\mathbb{T}}}
\newcommand{\T}{\,\mathbb{T}}
\newcommand{\G}{\,\mathbb{G}}

\newcommand{\End}{\operatorname{End}}
\newcommand{\mb}[1]{\mathbb{#1}}

\newcommand{\mf}[1]{\mathfrak{#1}}

\newcommand{\bs}[1]{\boldsymbol{#1}}
\newcommand{\op}[1]{\operatorname{#1}}

\newcommand{\opp}{\vee}

\newcommand{\pt}{\operatorname{pt}}
\newcommand{\ch}{\operatorname{ch}}
\newcommand{\Tr}{\operatorname{Tr}}

\newcommand{\YB}{\operatorname{YB}}
\newcommand{\Rep}{\operatorname{Rep}}
\newcommand{\sfT}{\mathsf{T}}
\newcommand{\bbV}{\bold{V}}
\newcommand{\ve}{\varepsilon}
\newcommand{\sigmax}{\sigma}
\newcommand{\be}{\mathbf{b}}
\newcommand{\bfe}{\mathbf{e}}
\newcommand{\tg}{\tau}

\newtheorem{defn}{Definition}[section]

\newtheorem{remark}{Remark}[section]
\newtheorem{thm}{Theorem}[section]
\newtheorem{prop}{Proposition}[section]
\newtheorem{cor}[prop]{Corollary}
\newtheorem{lemma}[prop]{Lemma}
\newtheorem{example}[prop]{Example}

\crefname{conjecture}{Conjecture}{Conjectures}
\Crefname{conjecture}{Conjecture}{Conjectures}

\numberwithin{equation}{section}

\newcommand{\connecting}[2]{\tikz[baseline=1ex]{
\draw[black] (0,0) -- (0,0.5);
\node at (0,-0.2) {$#2$};
\node at (0,0.7) {$#2$};
\draw[black] (-0.25, 0.25) -- (-0.05, 0.25);
\draw[black] (0.05, 0.25) -- (0.25, 0.25);
\node at (0.39,0.25) {$#1$};
\node at (-0.39,0.25) {$#1$};}}

\newcommand{\avoiding}[2]{\tikz[baseline=1ex]{
\node at (0.39,0.25) {$#2$};
\node at (-0.39,0.25) {$#1$};
\draw[black] (-0.25, 0.25) .. controls (0, 0.25) .. (0,0);
\draw[black] (0, 0.5) .. controls (0, 0.25) .. (0.25,0.25);
\node at (0,-0.2) {$#1$};
\node at (0,0.7) {$#2$};
} }

\newcommand{\passing}[2]{\tikz[baseline=1ex]{
\draw[black] (-0.25,0.25) -- (0.25,0.25);
\node at (0.39,0.25) {$#1$};
\node at (-0.39,0.25) {$#1$};
\draw[black] (0, 0) -- (0, 0.2);
\draw[black] (0, 0.3) -- (0, 0.5);
\node at (0,-0.2) {$#2$};
\node at (0,0.7) {$#2$};}}

\newcommand{\Connecting}[4]{\tikz[baseline=1ex]{
\draw[black] (0,0) -- (0,0.5);
\node at (0,-0.2) {$#4$};
\node at (0,0.7) {$#2$};
\draw[black] (-0.25, 0.25) -- (-0.05, 0.25);
\draw[black] (0.05, 0.25) -- (0.25, 0.25);
\node at (0.39,0.25) {$#3$};
\node at (-0.39,0.25) {$#1$};}}

\newcommand{\Avoiding}[4]{\tikz[baseline=1ex]{
\node at (0.39,0.25) {$#3$};
\node at (-0.39,0.25) {$#1$};
\draw[black] (-0.25, 0.25) .. controls (0, 0.25) .. (0,0);
\draw[black] (0, 0.5) .. controls (0, 0.25) .. (0.25,0.25);
\node at (0,-0.2) {$#4$};
\node at (0,0.7) {$#2$};
} }

\newcommand{\Passing}[4]{\tikz[baseline=1ex]{
\draw[black] (-0.25,0.25) -- (0.25,0.25);
\node at (0.39,0.25) {$#3$};
\node at (-0.39,0.25) {$#1$};
\draw[black] (0, 0) -- (0, 0.2);
\draw[black] (0, 0.3) -- (0, 0.5);
\node at (0,-0.2) {$#4$};
\node at (0,0.7) {$#2$};}}

\begin{document}
\title[Quantum K--theory from Yang-Baxter algebras]{Quantum K--theory of Grassmannians from a Yang-Baxter algebra}

\author{Vassily Gorbounov}
\address{Faculty of Mathematics, National Research University, Higher School of Economics, Usacheva 6, 119048 Moscow, Russia} 
\email{vgorb10@gmail.com}

\author{Christian Korff}
\address{School of Mathematics and statistics, University of Glasgow, Glasgow G12 8QQ, UK}
\email{Christian.Korff@glasgow.ac.uk}

\author{Leonardo C.~Mihalcea}
\address{
Department of Mathematics, 
Virginia Tech University, 
Blacksburg, VA 24061
USA
}
\email{lmihalce@vt.edu}
\subjclass[2020]{Primary: 14N35, 82B23; Secondary:14M15, 14N15, 17B80, 55N20, 81R12}
\keywords{Quantum K-theory, Yang Baxter algebra, Grassmannians, Bethe ansatz} 

\date{March 31, 2025}

\begin{abstract} In an earlier paper, two of the authors 
defined a $5$-vertex Yang-Baxter algebra (a Hopf algebra) 
which acts on the sum of the equivariant quantum K-rings of Grassmannians
$\Gr(k,n)$, where $k$ varies from $0$ to $n$.
We construct geometrically defined operators
on quantum K-rings describing this action. In particular, the $R$-matrix defining the Yang-Baxter algebra corresponds to the left Weyl group action. Most importantly, we use the `quantum=classical'
statement for the quantum K-theory of Grassmannians 
to prove an explicit geometric interpretation of the action of generators
of the Yang-Baxter algebra. The diagonal entries of the monodromy matrix are given
by quantum K-multiplications by explicitly defined classes, and
the off-diagonal entries by certain push-pull convolutions. We use this to find a 
quantization of the classes of fixed points in the quantum K-rings, corresponding to the Bethe vectors of the Yang-Baxter algebra. On each of the quantum K-rings, we prove that the 
two Frobenius structures (one from geometry, and the other from the integrable system construction) coincide. We discuss several applications, 
including an action of the extended affine Weyl group
on the quantum K-theory ring (extending the Seidel action), 
a quantum version of the localization map (which is a ring
homomorphism with respect to the quantum K-product), and a graphical calculus to multiply by
Hirzebruch $\lambda_y$ classes of the dual of the tautological quotient bundle. In an Appendix we illustrate our results in the case when $n=2$. 
\end{abstract}
\maketitle

\setcounter{tocdepth}{2}
\tableofcontents


\section{Introduction} The study of quantum cohomology, and quantum K-theory 
rings of flag manifolds
is closely related to that of (quantum) integrable systems and lattice 
models in mathematical physics; see for example the pioneering works of 
Givental, Kim \cite{givental-kim:QH,kim:QH} and Givental, Lee \cite{givental.lee:quantum}.
More recently, a different perspective has emerged, which links
the quantum cohomology and quantum K-theory
to solutions of the quantum Yang-Baxter equation.
The latter
describe exactly solvable lattice models in statistical mechanics \cite{Baxter:book} 
and, algebraically, leads to the definition of quantum groups via the quantum 
inverse scattering method; see e.g.~\cite{faddeev1990lectures} and references 
therein. 

The new perspective unifies constructions in the representation theory of quantum groups to those in enumerative geometry. Many of the foundational ideas are presented in the influential works \cite{okounkov:lectures,MO:book} by Maulik and Okounkov.
There is a flurry of activity in this area, and for a sampler biased
towards the specific case of (cotangent bundles of) flag manifolds, see
\cite{BMO:springer,GRTV:quantum,gorbounov2014equivariant,knutson.zinn:SchubertI,
schiffmann.vasserot:cohomological,gorbunov2020yang,KPSZ:quivers,koroteev.zetlin:qKZ-tRS,li.su.xiong;springer,tarasov.varchenko:monodromy,smirnov:elliptic-lectures,sinha.zhang:QK1} and references therein. In parallel, inspired by the Bethe/Gauge correspondence by Nekrasov and Shatashvili \cite{nekrasov.shatashvili:supersymmetric} there have been several papers by different groups approaching quantum cohomology and K-theory from a string theoretic perspective; for example see \cite{jockers.and.al:Wilson,ueda.yoshida:3d,kim.ueda.yoshida:residue,jockers.and.al:BPS,GMSZ:symplectic,closset.khlaif:grothendieck,closset.khlaif:twisted,gu2024quantum,GGMWY:correspondence} and references therein.

The results on the quantum K-theory of cotangent 
bundles are formulated in terms of certain quantizations
of classical operators; see, e.g., \cite{KPSZ:quivers}. 
After taking an appropriate limit, these imply statements 
in the quantum K-theory of the base manifold. However, explicit 
calculations of K-theoretic Gromov-Witten invariants are notoriously 
difficult, therefore the geometric interpretations
of the quantizations is unclear. 

In this paper we take 
a different route by relying on the `quantum=classical' statement
\cite{buch.m:qk}. This allows us to perform explicit 
calculations of quantum K-products with respect to a 
distinguished basis - the {\em Schubert basis} in geometry, respectively the 
{\em spin basis} in physics - and 
one aims to find geometric and combinatorial descriptions of the action of the elements in the 
Yang-Baxter algebra on this basis. 

In physics, another distinguished basis is given by the {\em Bethe vectors}, 
corresponding geometrically to
quantizations of the classes of torus fixed points. 
This basis arises 
as a common eigenbasis of a certain family of commuting elements in the Yang-Baxter algebra.
A foundational problem is to describe as explicitly as possible the transition 
matrix between the spin and Bethe bases. The main obstacle to solving this problem is that the Bethe ansatz equations (which govern the integrable system) cannot be explicitly solved in general. However, in their symmetrized form their solution is given in terms of the spectrum of the quantum cohomology rings (or their generalized versions) and one is therefore interested in exhibiting the underlying geometry of the Bethe ansatz or quantum inverse scattering method by giving a purely geometric construction of the Yang-Baxter algebra.

In this paper we address in detail these problems for the case of $\QK_{\T}(\Gr(k;n))$, the equivariant 
quantum K-theory ring of the Grassmannian $\Gr(k;n)$ parametrizing linear subspaces of dimension $k$ in $\C^n$. The quantum K-ring was defined by Givental and Lee \cite{givental:onwdvv,lee:QK} and is equipped with the distinguished Schubert basis.
An algebraic construction of the action of the Yang-Baxter algebra on $\QK_{\T}(\Gr(k;n))$ was given by two of the authors in \cite{gorbounov2017quantum}. 

Our goal is to find a geometric interpretation of the 
action of the Yang-Baxter algebra (in terms of the $R$-matrix, and 
all the entries of the monodromy matrix), 
and use this interpretation
to transport structures from geometry to integrable systems, and vice versa. 
This leads to new results, and new perspectives, in both 
quantum Schubert Calculus, and quantum integrable systems. For instance, we construct 
an action of the {\em extended affine} Weyl group on the ring $\QK_{\T}(\Gr(k;n))$, 
generalizing to the equivariant case the more familiar Seidel representation from \cite{postnikov:affine,chaput.perrin:affine} in
quantum cohomology, and \cite{li.liu.song.yang:seidel,buch.chaput.perrin:seidel} in quantum K-theory; see also \cite{chow:peterson}
for a different generalization.

Furthermore, we identify the geometric Frobenius structure of the quantum K-theory to the one defined in \cite{gorbounov2017quantum} for the integrable system
(in terms of Bethe vectors). This foundational question only arises for quantum K-rings - 
in quantum cohomology the two pairings giving the Frobenius structures are the 
same as the classical intersection pairing. The identification
leads to a quantum generalization of the localization map, and a quantum version of the
Atiyah-Bott theorem,
both of which might be of interest in their own right.

While this paper focuses on the quantum K-ring of the Grassmannian, 
it also provides a roadmap
for the study of the quantum K-rings of generalized flag manifolds. 
We plan to do this in subsequent work.
We provide next a table summarizing the `dictionary' we prove between the geometric and integrable systems perspectives; this
table should generalize to other flag manifolds, and beyond. 
\begin{table}[h!]
\centering
\begin{tabular}{||c | c||} 
 \hline
 {\bf Quantum Schubert Calculus} & {\bf Quantum Integrable System} \\ [0.5ex] 
 \hline\hline
 left Weyl group action & $R$ matrix \\ 
 & (solution of the quantum Yang-Baxter equation) \\ [0.5ex]\hline
 multiplication/convolution operators & entries of the monodromy matrix \\
 & $T(y)=\left(\begin{smallmatrix}t_{00}(y)& t_{01}(y)\\t_{10}(y)& t_{11}(y)\end{smallmatrix}\right)$ \\
 quantum multiplication by $\lambda_y(\cQ^\vee_{n-k})$ & quantum trace: $t(y)=t_{00}(y)+q t_{11}(y)$\\
 (push-pull) convolution operators & off diagonal operators: $t_{01}(y)$, $t_{10}(y)$\\[0.5ex]\hline
 Schubert classes & spin basis\\
 quantization of fixed points & Bethe vectors\\
 [0.5ex]\hline
 quantum Whitney relation & functional relation/Bethe ansatz equations\\
 quantized Chern roots & Bethe roots\\[1ex] \hline
\end{tabular}
\caption{Dictionary.}
\label{table:1}
\end{table}

\subsection{Statement of results} We present next a more precise account of our results.  
We start by introducing some notation. Let $\T \subset \GL_n(\C)$ be the 
subgroup of diagonal matrices in $\GL_n(\C)$. Denote by 
$\Kpt= \Z[\ve_1^{\pm 1}, \ldots, \ve_n^{\pm 1}]$ the representation ring of $\T$, and by $W \simeq S_n$ the Weyl group.
Consider the tautological sequence of ($\T$-equivariant) vector bundles on $\Gr(k;n)$:
\[ 0 \to \mathcal{S}_k \to \C^n \to \mathcal{Q}_{n-k} \to 0 \/\] 
Set $\K_{\T}(\Gr(k;n))$ to be the $\T$-equivariant $\K$-theory ring of $\Gr(k;n)$.
We use the same notation for an equivariant vector bundle $E$ and the class 
it determines in the equivariant 
$\K$-theory ring. The (Hirzebruch) $\lambda_y$-class of $E$ is the element 
$\lambda_y(E) = 1 + y E + \ldots + y^{e} \wedge^e E \in \K_T(\Gr(k;n)[y]$, where $y$ is a formal variable,
and $e$ is the rank of $E$; we refer to \cref{sec:preliminaries} for more details. 

The equivariant quantum K-theory ring $\QK_{\T}(\Gr(k;n))$, 
defined by Givental and Lee \cite{givental:onwdvv,lee:QK} (in a more general
context), is a free $\Kpt[\![q]\!]$-module equipped with a basis given by the Schubert classes 
$\cO_\lambda$, where $\lambda$
varies in the set of partitions $\lambda=(\lambda_1, \ldots, \lambda_k)$ included in the 
$k \times (n-k)$ rectangle; $q$ is the quantum parameter. The quantum K-multiplication is given by 
\[ \cO_\lambda \star \cO_\mu = \sum N_{\lambda,\mu}^{\nu,d} q^d \cO_\nu \/,\]
 where the sum is over all non-negative degrees $d$, and partitions $\nu$.
The structure constants $N_{\lambda,\mu}^{\nu,d}$ are defined in terms of $2$ and $3$-point K-theoretic
Gromov-Witten invariants of the Grassmannian; see \cref{sec:QK} below. 

We use the Yang-Baxter algebra $\YB$ from \cite{gorbounov2017quantum},
generated by $R=(R_{ij})$ (the $R$-matrix) and $T$ (the monodromy matrix) satisfying the $RTT=TTR$ equation; cf. \cref{sec:YB} below. Here $R$ is associated to a certain five vertex model, see \eqref{5v}. We extend 
scalars by the quantum parameter $q$ and the formal variable $y$ - the 
parameter in the Hirzebruch $\lambda_y$-class, understood here 
as the spectral parameter of the integrable system. 
For each $n\ge 2$ there is a natural $\YB$-module obtained from the $R$-matrix and the coproduct: 
\[ (\C^2)^{\otimes n} \otimes \Kpt[\![q]\!] \simeq \bbV_n = \bigoplus_{k=0}^n V_{k,n}\]
where $V_{k,n}$ is a $\Kpt[\![q]\!]$-module isomorphic to $\QK_{\T}(\Gr(k,n))$, equipped with the
 (spin) $\Kpt$-basis $\{ v_\lambda \}$. Then $\bbV_n$ is a highest weight module 
and each $V_{k,n}$ is a weight subspace. The monodromy matrix 
$T= \begin{pmatrix} t_{0,0} & t_{0,1} \\ t_{1,0} & t_{1,1} \end{pmatrix}$
is defined by 
the diagonal operators 
$t_{ii} \in \End_{\Kpt} \bbV_n$ 
which preserve each $V_{k,n}$, and off diagonal operators
\[ t_{10}: V_{k,n} \to V_{k+1,n}\/; \quad t_{01}: V_{k+1,n} \to V_{k,n} \/.\]
The isomorphism of free $\Kpt[y][\![q]\!]$-modules
\[ \Phi: \bbV_n= \bigoplus_{k=0}^n V_{k,n} \to \bigoplus_{k=0}^n \QK_{\T}(\Gr(k,n))\/; \quad v_\lambda \mapsto \cO_\lambda \/, \]
equips $\bigoplus_{k=0}^n \QK_{\T}(\Gr(k,n))$ with a $\YB$-module structure. Note that the latter module is also equipped
with a structure of $\Kpt[y][\![q]\!]$-{\em algebra} given by the quantum K-product $\star$ on each of the summands.  
Our main result realizes the monodromy and the R-matrix as 
geometric operators in the quantum K-algebra. 

We show that the components of the $R$-matrix correspond to the left Weyl group action
on $\QK_{\T}(\Gr(k;n))$, defined in \cite{knutson:noncomplex,MNS:left}. This is the action induced by 
left multiplication by $\GL_n(\C)$ on $\Gr(k;n)$. 
In various forms, this relation was observed in many other contexts, see e.g. \cite{MNS:left}. 
To be precise, the twisted $R$-matrix operator $\check{R}=P\circ R$, defined in Lemma \ref{lem:Rcheck}, 
is given by the left Weyl group multiplication:
\[ \Phi(\check{R}_{n-i,n-i+1}(\ve_i/\ve_{i+1})v_\lambda) = \bs{s}_i.\cO_\lambda \/,\]
where $\bs{s}_i$ is the action given by the simple reflection $s_i$. 
For $0 \le k \le n-1$ consider the incidence diagram 
\begin{eqnarray}\label{diag:q=cl-intro}
\xymatrix{\Fl(k,k+1;n)\ar[d]_{p_1}\ar[r]^{{p_2}} & \Gr(k+1;n)\\
\Gr(k;n) & } \/
\end{eqnarray}
where $\Fl(k,k+1;n)$ is the two step flag manifold and all maps are the projection maps.
Our main result gives a geometric realization of the operators $t_{i,j}$; cf. \Cref{thm:int-system-ops} and \Cref{cor:main-cor}.
\begin{thm}\label{thm:QK=YB} The following hold:
\begin{enumerate} \item The quantum trace operator $t(y):=t_{00}(y)+q t_{11}(y)$ 
restricted to $V_{k,n} \otimes \C[y,q]$ satisfies:
\[ \Phi(t(y).v_\lambda) = \lambda_y(\mathcal{Q}_{n-k}^\vee) \star \cO_\lambda \/.\]

\item The off-diagonal operators are given by convolutions. More precisely, let
\[ \tg_{10}=\tau_{10}(y): \K_{\T}(\Gr(k,n))[y] \to \K_{\T}(\Gr(k+1,n))[y] \] and
\[ \tg_{01}=\tau_{01}(y): \K_{\T}(\Gr(k+1,n))[y] \to \K_{\T}(\Gr(k,n))[y] \]
be the convolution operators defined by 
\[ \tg_{10}(\kappa) = \lambda_y (\mathcal{Q}_{n-k-1}^\vee) \cdot (p_2)_*(p_1)^*(\kappa) - (p_2)_* p_1^*(\lambda_y (\mathcal{Q}_{n-k}^\vee) \cdot \kappa) \/, \]
and by
\[ \tg_{01}(\kappa) = (p_1)_* (p_2)^*(\lambda_y (\mathcal{Q}_{n-k-1}^\vee) \cdot \kappa) \/. \]
Then $\Phi(t_{10}(y).v_\lambda) = \tg_{10}(\cO_\lambda)$ and $\Phi(t_{01}(y).v_\mu) = \tau_{01}(\cO_\mu)$.
\end{enumerate}
\end{thm}
This generalizes to K-theory the results from \cite{gorbunov2020yang}
for the (equivariant) quantum cohomology of the Grassmannian. 
Unlike quantum cohomology, the structure constants in quantum K-theory are no longer
single Gromov-Witten invariants.
In particular, the two terms in $\tg_{10}$ correspond to the two factors
from \eqref{E:QKstr} below giving the structure constant $N_{\lambda, \mu}^{\nu,d}$. In other
words, $\tau_{10}$ is precisely the convolution arising in the `quantum=classical' phenomenon proved
in \cite{buch.m:qk}; cf.~\Cref{thm:q=cl} below. 

The proof of \Cref{thm:QK=YB} exploits the fact that both sets of operators
$t_{ij}$ and $\tau_{ij}$ commute with the left Weyl group action. Therefore it suffices to check equality on the class of the Schubert 
point class, as this class generates the full ring under the (left) nil-Hecke algebra action. 

In geometry, the action of the `off diagonal' operators $\tau_{ij}$, for $i \neq j$,
on the Schubert point is obtained by using known formulae to push and pull Schubert classes.
For the diagonal operators, we need to  
calculate the quantum K-product by $\lambda_y(\cQ^\vee_{n-k})$. This requires rather detailed 
control on K-theoretic Gromov-Witten invariants involving $\lambda_y(\cQ^\vee_{n-k})$.
Our calculation
relies on the `quantum=classical' statement \cite{buch.m:qk}, which reduces 
the quantum K-multiplication
to calculations of push-pull convolutions. All these are discussed in \cref{sec:QK-by-ly} and 
the final calculations are done in \cref{sec:convo-ops}. 

For integrable systems, the key calculation follows from the
graphical calculus associated to the $5$-vertex model, see \Cref{sec:graphical}. 
This is a direct consequence of the properties of the R-matrix giving the Yang-Baxter algebra in this paper; see \Cref{sec:YB}.

\Cref{thm:QK=YB} allows us to transport structures from one side to the other. 
A remarkable structure on $\bbV_n$ is the action of the extended affine Weyl group 
$\widetilde{W}=W\ltimes\Z^n$, extending the left Weyl group action of $W \simeq S_n$.
Let $t_{\ve_i}$ be the translation corresponding to the $i$th component of $\Z^n$.
This group is also isomorphic to $W_{\mathrm{af}} \ltimes \Z$, where $W_{\mathrm{af}}$
is the affine Weyl group, and $\Z$ is identified with the fundamental group of $\GL_n(\C)$.
We denote by $\rho$ the cyclic generator of $\Z$; in terms of translations, 
$\rho= t_{\ve_n}s_{n-1}\cdots s_1$. The endomorphisms
corresponding to elements in $\widetilde{W}$ 
will be denoted using bold letters.

Under the isomorphism 
$\Phi$ the action of $\widetilde{W}$ may be transported to the quantum $\K$-theory side, giving the following (cf.~\Cref{prop:trans} and \Cref{cor:QK-seidel}):
\begin{cor}
(a) The translations $t_{\varepsilon_i}\in\widetilde W$, for $1 \le i \le n$, act by $t(-\ve_i)$. Explicitly,
\[ {\bs t}_{\varepsilon_i}.\cO_\lambda=\Phi(t(-\ve_i).v_\lambda)= \lambda_{-\ve_i}(\mathcal{Q}^\vee) \star \cO_\lambda\] 
and ${\bs t}_{\varepsilon_i}$ leaves the equivariant parameters in $\Kpt$ invariant.

\noindent (b) The cyclic (or Seidel) generator $\rho$ acts by
\begin{equation}\label{rho}
        \bs{\rho}.(\chi\,\cO_\lambda)=\left\{\begin{array}{ll}
            q\,\chi^\rho\,\cO_{(\lambda_1-1,\ldots,\lambda_k-1)}, &  \ell(\lambda)=k \/;\\
            \chi^\rho\,\cO_{(n-k,\lambda_1,\ldots,\lambda_{k-1})}, & \text{else} \/, 
        \end{array}\right.\;
    \end{equation}
    where $\chi^\rho(\ve_1,\ldots,\ve_n)=\chi(\ve_n,\ve_1,\ldots,\ve_{n-1})$ with $\chi\in\Kpt$.
That is, $\rho$ permutes the equivariant parameters according to the cycle $s_{n-1}\cdots s_1$.
\end{cor}
We note that in integrable systems there is a further 
quantization of this action, with a single parameter, which may be 
thought of as a `loop parameter'; cf. \cref{sec:extended}.

Since the classical multiplication by $\lambda_y(\cQ_{n-k}^\vee)$ 
is diagonalizable with distinct eigenvalues,
so is the quantum multiplication $\lambda_y(\cQ_{n-k}^\vee) \star$. The corresponding 
eigenvectors are called the {\em Bethe vectors}. We denote by $\be_\lambda \in V_{n,k}$ the Bethe vector for $\lambda$,
and by 
\[ \bfe_\lambda^q:= \Phi(\be_\lambda) \in \QK_{\T}(\Gr(k;n)) \]
the corresponding element in the quantum K-theory ring.  
{\em A priori} the Bethe vectors are defined only up to a multiple, 
and a normalization was chosen in \cite{gorbounov2017quantum}, having many desirable properties, 
such as the fact that 
$\bfe_\lambda^q$ modulo $q$ is the class of the fixed point $\bfe_\lambda$ in $\K_{\T}(\Gr(k;n))$.
In particular, the Bethe vectors
$\be_\lambda \in V_{k,n}$ equip each weight space with a structure of a semisimple ring.
The geometric realization of the off-diagonal operators of $T$ gives a new, geometric, algorithm, to find the 
Bethe vectors $\bfe_\lambda^q$ in the quantum integrable system. 
We collect these facts in the next corollary.
\begin{cor} 
(a) The set $\{ \bfe^q_\lambda:=\Phi(\be_\lambda)\}$ of Bethe vectors diagonalize 
the operators $\lambda_y(\mathcal{Q}_{n-k}^\vee) \star$, and the Bethe vectors 
are orthogonal:
\[ \bfe_\lambda^q \star \bfe_\mu^q = 0 \/, \quad \forall \lambda \neq \mu \/. \]
In particular, $\QK_{\T}(\Gr(k,n))$ is a semisimple ring.

(b) The elements $\bfe_\lambda^q \in V_{k,n}$ may be calculated from the convolution 
operators and the roots of the Bethe ansatz equations: 
\[ \Phi(\be_\lambda) =\tg_{10}(-x^\lambda_1)\cdots\tg_{10}(-x^\lambda_{k})\cO_{o}  
= \bfe_\lambda^q = \tilde{\tg}_{01}(-x^{\lambda^t}_1)\cdots \tilde{\tg}_{01}(-x^{\lambda^t}_{n-k})\tilde{\cO}_o\]
where $x^\lambda=(x^\lambda_1, \ldots, x^\lambda_{k})$ with $\lambda\in\Pi_{k,n}$ are the distinct solutions of the Bethe ansatz equations for $\Gr(k,n)$,
\[
\prod_{j=1}^n(1-x_i/\ve_j)\prod_{j\neq i}^k(x_j/x_i)+(-1)^{k}q=0,
\]
and $\cO_o=\Phi(v_o)$ is the unique Schubert class in $\K_T(\Gr(0,n))$. 

Similarly, $x^{\lambda^t}$ denotes the set of roots associated to `dual' 
Bethe ansatz equations, and $\lambda^t$ is the transpose of $\lambda$, 
and $\tilde{\cO}_o$ is the Schubert class in $\K_{\T}(\Gr(n,n))$.~\footnote{Here the solutions $x^\lambda$ of the Bethe Ansatz equations are 
labelled by partitions $\lambda\in\Pi_{k,n}$ as follows: we attach the partition 
$\lambda$ to the solution $x^\lambda$ if at $q=0$ it specializes to $x^\lambda|_{q=0}=\ve^\lambda:=(\ve_{\lambda_k+1},\ldots,\ve_{\lambda_1+k})$; see Lemma 4.6 in \cite{gorbounov2017quantum}.}
Finally, the `dual' operators $\tilde{\tg}_{01}$ are defined in \cref{ss:dual-ops} below.
\end{cor}
Implicit in all of the above is that the left Weyl group action permutes the Bethe vectors, i.e., 
for $w \in W$, 
\[ \bs{w}.\bfe_\lambda^q = \bfe_{w(\lambda)}^q \/.\]

From definition, the quantum K-theory algebra has a structure of a Frobenius algebra, given by a pairing denoted by 
$( \cdot , \cdot )_{\QK}$, and defined by 
\[ (\cO_\lambda,\cO^{\mu})_{\QK} = \frac{q^{d(\lambda,\mu)}}{1-q} \/, \]
where $\cO^\mu$ is the opposite Schubert class, and $d(\lambda,\mu)$ is the smallest power of $q$ in 
the quantum cohomology (or K-theory) product of the opposite Schubert classes for $\lambda$ and $\mu$; cf.~\cite{BCLM:euler,BCMP:qkpos}.
We may use this structure, and the interpretation of the Bethe vectors $\bfe_\lambda^q$, to define a {\em quantum (equivariant) localization map}
\[ \iota: \QK_{\T}(\Gr(k;n)) \to \bigoplus_\lambda \Kpt[\![q]\!] \/; \quad \kappa \mapsto (\kappa, \bfe_\lambda^q)_{\QK} \/. \]
We prove in \Cref{prop:qloc-inj} that 
the quantum localization map $\iota$ is an injective homomorphism of $\Kpt[\![q]\!]$-algebras.

Motivated by calculations in the ordinary equivariant K-theory ring, Gorbounov and Korff 
\cite{gorbounov2017quantum} defined a pairing $\langle \cdot , \cdot \rangle$ giving each 
$V_{k,n}$ a structure 
of a Frobenius algebra. Their pairing is determined by the condition that 
\[ \langle v_\lambda, \be_\mu \rangle = G_\lambda (1-x^\mu| 1-\ve) \/, \]
where $G_\lambda(1-x|1-\ve)$ is the double Grothendieck polynomial
and $x^\mu$ is the $\mu$-Bethe root. (In particular, this implies that
the Bethe vectors are orthogonal, and that $\langle v_\emptyset, \be_\mu \rangle =1$ for any partition $\mu$.)
Since the Bethe vectors are deformations 
of fixed point classes, this formula deforms the usual localization formula for Schubert classes 
in equivariant $\K$-theory. 

Our theorem is that the two pairings coincide (cf.~\Cref{thm:frob-pairings}):
\begin{thm}\label{thm:Frob-intro}
The isomorphism $\Phi:V_{k.n} \to \QK_{\T}(\Gr(k,n))$ is 
an isomorphism of Frobenius algebras, i.e., for any $a,b \in V_{k,n}$,
\[ (\Phi(a), \Phi(b))_{\QK} = \langle a, b \rangle \/. \]
\end{thm}
In the context of equivariant quantum cohomology, variants of
the quantum localization map appeared in \cite{korff.stroppel:W}, and in 
the preprint \cite{gorbounov2014equivariant} (see e.g., section 5.6), which was 
upgraded to quantum K-theory in \cite{gorbounov2017quantum}. We also note that 
in quantum cohomology, the pairings giving the Frobenius structures coincide with the 
classical pairing; see \Cref{rmk:QH-Frobenius}. The question of comparing the `geometric' and `integrable systems' pairings in quantum K-theory does not seem to have been asked before; in particular, \Cref{thm:Frob-intro} is new. We note that the equality of the geometric pairing with one arising from a string theory perspective was investigated in \cite{closset.khlaif:grothendieck,GGMWY:correspondence}.  

The proof of \Cref{thm:Frob-intro} relies on two key results. The first is a presentation
of the quantum K-ring by generators and relations, where one can explicitly determine polynomial
representatives of {\em both} the Schubert classes $\cO_\lambda$, and of the
class $\lambda_y(\cS_k)$. In turn, this is a consequence of the equivalence
of the Desnanot-Jacobi identity
satisfied by the product $t(y) \tilde{t}(y^{-1})$, and
proved in \cite{gorbounov2017quantum}, to the `Whitney relations' satisfied by 
\[\lambda_y(\cS_{k}) \star \lambda_{y}(\cQ_{n-k}) \/, \] 
proved in \cite{GMSZ:QK}; see \cref{sec:functional}. In particular, \cref{cor:QK-pres}, which proves that these are two equivalent presentations of the same ring, answers a question from \cite{GMSZ:QK}. The second fact is an explicit formula for the `off-shell' Bethe vector (cf.~\eqref{E:Bk-GK}), which allows us to expand the Bethe vectors into Schubert classes, and compare the 
two pairings. 

\Cref{thm:Frob-intro} leads to some remarkable identities, such as a 
`quantum Atiyah-Bott' localization theorem (cf.~\Cref{cor:qAB}), which states that:
\[ (\kappa , 1 )_{\QK} = \sum_\lambda \frac{(\kappa, \bfe_\lambda^q)_{\QK}}{(\bfe_\lambda^q,\bfe_\lambda^q)_{\QK}} \/.\]
In the particular case when $\kappa=1$, we know that 
$(1,1)_{\QK} =1/(1-q)$ (see \cite{BCLM:euler}) giving an identity satisfied by the 
the `quantum Euler classes'
$(\bfe_\lambda^q, \bfe_\lambda^q)_{\QK}$:  
\[ \frac{1}{1-q} = \sum_\lambda \frac{1}{(\bfe_\lambda^q, \bfe_\lambda^q)_{\QK}}\/. \]

Another important application of \Cref{thm:QK=YB} is that the graphical calculus 
associated to the $5$-vertex model,
developed in \Cref{sec:graphical}, yields cancellation free 
formulae for the actions of the operators $t(y)=\lambda_y(\cQ_{n-k}^\vee) \star$ 
and $\tilde{t}(y)=\lambda_y(\cS_k)\star$ 
on Schubert classes. This re-interprets formulae from
\cite{gorbounov2017quantum}, and generalizes the Chevalley formulae from \cite{buch.m:qk,BCMP:qkchev}.
Due to length reasons, we decide to explore this and other 
applications of the graphical calculus in future work.

Finally, in \Cref{sec:app-examples} we illustrate \Cref{thm:QK=YB} and the pairings from 
\Cref{thm:Frob-intro} in the case of the module $\bbV_2 = \QK_{\T}(\Gr(0,2)) \oplus \QK_{\T}(\Gr(1,2)) \oplus
\QK_{\T}(\Gr(2,2))$. The projective spaces are the only ones where one can solve the Bethe ansatz equations explicitly, and our results are already non-trivial for $\bP^1$. In \cref{sec:beta-calc} we explain how one can add the homogenizing parameter $\beta$ (of degree $-1$) in our statements, something which is important for potential physics applications, where $\beta$ plays the role of a coupling constant or an interaction parameter.
\medskip

{\em Acknowledgments.}
For V. G. this article is an output of a research project
implemented as part of the Basic Research Program at the National Research
University Higher School of Economics (HSE University).
L.M. was partially supported by NSF grant DMS-2152294, and gratefully acknowledges the support of Charles Simonyi Endowment, which provided funding for the membership at the Institute of Advanced Study during the 2024-25 Special Year in ‘Algebraic and Geometric Combinatorics’. C.K. and L.M. gratefully acknowledge the financial support by an International Exchange award of the Royal Society (IES/R2/232087) which allowed mutual visits to each other's institutions and helped in the completion of this work. L.M. also thanks A. Okounkov and M. Shimozono for useful discussions, and W. Gu, I. Huq-Kuruvilla, E. Sharpe, W. Xu, H. Zhang, and H. Zou for collaborating on related projects.


\section{Preliminaries}

\subsection{Notation and conventions}\label{sec:conventions} 
We work over $\C$, and by a variety we mean a scheme of finite type which is 
reduced and irreducible. For a fixed natural number $n$ denote by $f_1, \ldots, f_n$ 
the standard basis of $\C^n$. Let $\T \simeq (\C^*)^n$ be the $n$-dimensional torus, 
acting on $\C^n$ as usual: 
\[(a_1, \ldots, a_n).(z_1, \ldots, z_n) = (a_1z_1, \ldots, a_n z_n) \/. \]
The $\T$-module $\C^n$ has a weight space decomposition given by 
$\C^n = \bigoplus_{i=1}^n \C_{\ve_i}$, where $\C_{\ve_i}$ denotes 
the one dimensional $\T$-module with action given by the character 
$\ve_i(z_1, \ldots , z_n)= z_i$. Denote 
the representation ring of $\T$ by $\Kpt$, and as customary we identify characters by 
their $1$-dimensional 
modules: $\ve_i = [\C_{\ve_i}]$. Then $\Kpt= \Z[\ve_1^{\pm 1}, \ldots , \ve_n^{\pm 1}]$, 
the Laurent polynomial ring in the indeterminates $\ve_i$.
If $y$ be an indeterminate, the $\lambda_y$ class of 
$\C^n$ is the element in $\Kpt[y]$ defined by 
\[ \lambda_y(\C^n) = (1+y \ve_1)\cdot \ldots \cdot (1+ y\ve_n) \/.\]

Let $W:=S_n$ be the symmetric group in $n$ letters, equipped with 
length function $\ell:W \to \mathbb{N}$. Denote by $w_0$ its longest element. 
For an integer sequence $I= (1 \le i_1 < i_2 < \ldots < i_p < n)$,
define the subgroup $W_I \le W$
generated by simple reflections $s_i = (i,i+1)$ where $i \notin \{ i_1, \ldots, i_p\}$. 
Denote by $W^I$ 
the set of minimal length representatives of $W/W_I$.
This consists
of permutations $w \in W$ which have descents at most at positions 
$i_1, \ldots, i_p$, i.e., $w(i_k+1)< \ldots <w(i_{k+1})$, 
for $k=0, \ldots, p$, with the convention that $i_0=1, i_{p+1} =n$. 
Denote by $w_I$ the longest element in $W_I$. 
Consider two sequences $I=(1 \le i_1 < \ldots, i_p < n)$ and $J=(1 \le j_1 < \ldots < j_s )$
where $J$ is obtained from $I$ by removing some of the indices; we 
denote this situation by 
$J \subset I$. Then $W^J \subset W^I$ (as subsets), and
$W_I \le W_J$ (as subgroups). 

If $I= (1 < k < n)$ we denote by $W^k_n$ the set of minimal length 
coset representatives, or simply $W^k$ if $n$ is understood from the context. 
The elements in $W^k_n$ are called {\em Grassmannian permutations}. 
These are in bijection with $\Pi^k_n$, 
the set of partitions $\lambda=(\lambda_1, \ldots, \lambda_k)$ included in the 
$k \times (n-k)$ rectangle, i.e.,
$n-k \geq \lambda_1 \geq \ldots \ldots \lambda_k \geq 0$. Denote by $\mb{J}_n^k$ 
the set of $01$ words $j_1 j_2 \ldots j_n$
of length $n$ with $k$ $0$'s. We will use the following bijections among the sets $W^k_n, \Pi_n^k$ and $\mb{J}_n^k$.
The permutation $w_\lambda \in W^k_n$ associated to the partition $\lambda \subset k \times(n-k)$ is the unique Grassmannian permutation satisfying
\[ w(i)= \lambda_{k-i+1}+i \/; \quad 1 \le i \le k \/. \]
The $01$ word $J_\lambda = j_1 j_2 \ldots j_n$ is the word with $0$'s in positions 
$\lambda_i+k-i+1$, for $1 \le i \le k$.
Graphically, $J_\lambda$ is obtained by tracing the outline of the Young 
diagram of $\lambda$, starting from the SW corner, and placing $0$'s for 
the vertical 
steps, and $1$'s for the horizontal steps. The permutation $w_\lambda$ 
starts with the $k$ positions given by the $0$'s; see Figure \ref{fig:Young}.

\begin{figure}[h!]\label{fig:Young}
\centering
\includegraphics[width=.2\textwidth]{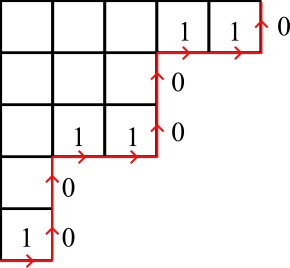} 
\caption{A Young diagram of $\lambda=(5,3,3,1,1)$ and its associated 01-word $I=1001100110$ for $\Gr(5;10)$. The corresponding Grassmannian permutation is $w_\lambda=[2,3,6,7,10,1,4,5,8,9]$.}
\end{figure}

\begin{example}\label{ex:Ydiagram} With these conventions
$J_\emptyset =00\cdots01\cdots1$ and $J_{(n-k)^k}=11\cdots10\cdots0$ (with $k$ $0$'s). 
The corresponding permutations are: 
\[ w_{\emptyset} = id \/; \quad w_{(n-k)^k} = \begin{pmatrix} 1 & 2 & \ldots & k & k+1 & \ldots & n \\ n- k +1 & n-k+2 & \ldots & n & 1 & \ldots & n-k \end{pmatrix} \/. \]
For another example, let $\lambda = (3,3)$, regarded in the $3 \times (7-3)$ rectangle. Then
\[ J_\lambda = 0 1 1 1 0 0 1 \quad \textrm{ and } \quad w_\lambda = \begin{pmatrix} 1 & 2 & 3 & 4 & 5 & 6 & 7 \\ 1 & 5 & 6 & 2 & 3 & 4 & 7 \end{pmatrix} \/. \]
\end{example}
We will use these different combinatorial descriptions interchangeably, noting that the description in terms of partitions and Grassmannian permutations is convenient in the geometric setting while the labelling in terms of binary strings occurs in the discussion of the quantum integrable lattice model and its graphical calculus.

\subsection{Schubert varieties} Let $\Gr(k,n)$ denote the Grassmannian 
parametrizing $k$ dimensional subspaces in $\C^n$. This is a complex projective manifold of 
dimension $k(n-k)$, and it is homogeneous under the action of $\G:=\GL_n(\C)$.
Fix the (opposite) standard flag 
\[ F_\bullet: F_1 = \langle f_n \rangle \subset F_2 = \langle f_n, f_{n-1} \rangle \subset \ldots \subset \C^n \/.\]
Let $\lambda=(\lambda_1, \ldots, \lambda_k)$ be a partition included in the $k \times (n-k)$ rectangle. The {\bf Schubert cell}
associated to $\lambda$ and $F_\bullet$ is defined by 
\[ X_\lambda^\circ(F_\bullet) = \{ V \in \Gr(k,n): \dim V \cap F_{n-k+i-\lambda_i} = i \} \/. \]
This is isomorphic to the affine space $\C^{k(n-k) - |\lambda|}$. The corresponding
{\bf Schubert variety}, denoted by $X_\lambda(F_\bullet)$ is the (Zariski) closure 
of $X_\lambda^\circ(F_\bullet)$. The Schubert cell is $\T$-stable and it 
contains a unique $\T$-fixed point, namely 
\[\bfe_\lambda= \langle f_{k+1-i+\lambda_i}: i=1 \ldots k \rangle \/. \] 
Finally, the Schubert cells give a stratification of the Grassmannian: $\Gr(k,n) = \bigsqcup_\lambda X_\lambda(F_\bullet)$.

More generally, for a sequence $I= (1 \le i_1 < i_2 < \ldots < i_p < n)$,
we consider the partial flag manifold
$\Fl(I) = \Fl(i_1, i_2, \ldots, i_p;n)$
which parametrizes partial flags 
$G_{i_1} \subset \ldots \subset G_{i_p} \subset \C^n$, where $\dim G_i = i$. 
This is homogeneous under the action of $\mathrm{GL}_n:=\mathrm{GL}_n(\C)$. 
The set of $\T$-fixed points of $\Fl(I)$ is in bijection with the set $W^I$, with
$w \in W^I$ corresponding to 
$\bfe_w := w. (F_{i_1} \subset \ldots \subset F_{i_p} \subset \C^n)$,
i.e., the $w$-translate of the appropriate components of the standard flag $F_\bullet$ above. 
To each $w \in W^I$, we may associate two Schubert varieties
$X^w = \overline{B. \bfe_w} \/; \quad X_w = \overline{B^-. \bfe_w}$.
With these conventions,
\[ \dim X^w = \mathrm{codim}~X_w = \ell(w) \/; \quad X_w = w_0 X^{w_0 w w_I} \]
where (recall) $w_I$ is the longest element in $W_I$. 
Inclusion of Schubert varieties defines the (partial) Bruhat order on $W^I$:
\[ v < w \textrm{ in } W^{I} \quad \Leftrightarrow \quad X^v \subset X^w 
\quad \Leftrightarrow \quad X_v \supset X_w \/. \]

\section{Equivariant K-theory}
\subsection{Preliminaries on equivariant K theory}\label{sec:preliminaries} In this section we recall some basic facts about the equivariant 
K-theory of a variety with a group action. For an introduction to equivariant K theory, and more details, see \cite{chriss2009representation}. 

Let $X$ be a smooth projective variety with an 
action of a linear algebraic group $G$. The equivariant 
K theory ring $\K_G(X)$ is the Grothendieck ring 
generated by symbols $[E]$, where $E \to X$ is an 
$G$-equivariant vector bundle, modulo the relations 
$[E]=[E_1]+[E_2]$ for any short exact sequence 
$0 \to E_1 \to E \to E_2 \to 0$ of equivariant vector bundles. 
The additive ring structure is given by direct sum, and the 
multiplication is given by tensor products of vector bundles. 
Since $X$ is smooth, any $G$-linearized coherent sheaf has 
a finite resolution by (equivariant) vector bundles, and the ring 
$\K_G(X)$ coincides with the Grothendieck group of 
$G$-linearized coherent sheaves on $X$. In particular, any 
$G$-linearized coherent sheaf $\mathcal{F}$ on $X$ 
determines a class $[\mathcal{F}] \in \K_G(X)$. 
An important special case is if $\Omega \subset X$ is a 
$G$-stable subscheme; then its structure sheaf $\cO_\Omega$
determines a class $[\cO_\Omega] \in \K_G(X)$.
The ring $\K_G(X)$ is an algebra over $\K_G(pt) = \Rep(G)$, the representation ring of 
$G$. If $G=\T$ is a complex torus, then this is the Laurent polynomial ring 
$\Kpt = \Z[\ve_1^{\pm 1}, \ldots, \ve_n^{\pm 1}]$ from the previous section.

The (Hirzebruch) $\lambda_y$ class is defined by 
\[ \lambda_y(E) := 1 + y [E] + y^2 [\wedge^2 E] + \ldots + y^e  [\wedge^e E] \in \K_G(X)[y] \/. \] 
This class was introduced by Hirzebruch \cite{hirzebruch:topological} in relation to the 
formalism of the Grothendieck-Riemann-Roch theorem. It may be thought as the 
K theoretic analogue of the (cohomological) Chern polynomial 
\[ c_y(E)= 1+ c_1(E) y + \ldots + c_e(E) y^e\] of the bundle $E$. 
The $\lambda_y$ class is multiplicative with respect to short exact sequences, i.e., if 
$0 \to E_1 \to E_2 \to E_3 \to 0$
is such a sequence of vector bundles then 
\[ \lambda_y(E_2) = \lambda_y(E_1) \cdot \lambda_y(E_3) \/; \] 
cf.~\cite{hirzebruch:topological}. 
Since $X$ is proper, the push-forward to a point equals the Euler 
characteristic, or, equivalently, the character of a virtual representation: 
\[\chi(X, \mathcal{F})= \int_X [\mathcal{F}] := \sum_i (-1)^i \ch H^i(X, \mathcal{F}) \/. \] 
(We omit $\T$ from the notation, as all Euler characteristics will be $\T$-equivariant.)  
This defines the pairing (with $E,F$ equivariant vector bundles):
\begin{equation}\label{E:pairing} \langle - , - \rangle :\K_G(X) \otimes \K_G(X) \to \K_G(pt); \quad
\langle [E], [F] \rangle := \int_X E \otimes F = \chi(X, E \otimes F) \/.  \end{equation}

A proper morphism $f: X \to Y$ is {\bf cohomologically trivial} if $f_*\cO_X= \cO_Y$ and 
its higher direct images vanish, i.e., $R^i f_* \cO_X =0$ for $i>0$. This implies that 
the induced morphism $f_*: \K(X) \to \K(Y)$ satisfies $f_*[\cO_X]=[\cO_Y]$. 
Same definition applies for schemes with a $\T$-action and $\T$-equivariant morphisms.
An important class of examples of cohomological trivial morphisms are projections from Schubert varieties, see \Cref{lemma:proj} below. More general situations follow from a theorem of Koll{\'a}r \cite{kollar:higherII}, see also \cite{buch.m:qk,BCMP:projgw}.

\subsection{Equivariant K theory of flag manifolds} Consider the partial flag 
manifolds $\Fl(I) = \Fl(i_1, i_2, \ldots, i_k;n)$. The Schubert varieties 
$X_w, X^w \subset \Fl(I)$
determine classes 
$\cO_w:= [\cO_{X_w}]$ and $\cO^w := [\cO_{X^w}]$
in $\K_{\T}(X)$. The $\T$-fixed points $\bfe_w$ give classes which by abuse of notation we still denote by 
$\bfe_w:= [\mathcal{O}_{\bfe_w}]$. The equivariant K-theory $\K_{\T}(\Fl(I))$
is a free module over 
$\K_{\T}(pt)=\Kpt$, with bases given by Schubert classes $\{ \cO_w \}, \{ \cO^w \}$ as $w$ varies in $W^I$. 
For the Grassmannian $\Gr(k;n)$, we denote by $\cO_\lambda$ the Schubert class
$[\cO_{X_{\lambda}(F_\bullet)}]$.

Consider two sequences $I=(1 \le i_1 < \ldots, i_p < n)$ and $J=(1 \le j_1 < \ldots < j_s )$
such that $J \subset I$, and denote by $\pi_{I,J}: \Fl(I) \to \Fl(J)$ the natural projection. This 
is a $G$-equivariant smooth morphism.
The push forward and pull backs of Schubert varieties are again Schubert varieties,
stable under the same (standard or opposite) Borel subgroup. 
More precisely, if $w \in W^I$, then $\pi(X^w) = X^{w'}$, where 
$w' \in W^J$ is the minimal length 
representative of $w W_J$; if $v \in W^J$, then $\pi^{-1}(X_v) = X_v$.
We will need the following fact:
\begin{lemma}\label{lemma:proj} Let $J \subset I$ and let $\Omega \subset \Fl(I)$ and 
$\Omega' \subset \Fl(J)$
be Schubert varieties. Then the restriction $\pi_{I,J}: \Omega \to \pi_{I,J}(\Omega)$ is 
cohomologically trivial. Furthermore,
\[ (\pi_{I,J})_*[\cO_\Omega] = [\cO_{\pi_{I,J}(\Omega)}] \quad \textrm{ and } (\pi_{I,J})^*[\cO_{\Omega'}] = [\cO_{\pi_{I,J}^{-1}(\Omega')}] \/. \] 
\end{lemma}
\begin{proof} The cohomological triviality, and the first equality, follow from \cite[Thm. 3.3.4(a)]{brion.kumar:frobenius} and the second 
because $\pi_{I,J}$ is a flat morphism.\end{proof}

The (left) action of $\GL_n$ on the flag manifold $\Fl(I)$ induces an action of the Weyl group $W$ 
\[ \bs{w}: \K_{\T}(\Fl(I)) \to  \K_{\T}(\Fl(I)) \/, \quad \kappa  \mapsto \bs{w}.\kappa \/. \]
This was studied in \cite{knutson:noncomplex}
 (in cohomology), and \cite{harada.landweber.sjamaar:divided,MNS:left}
 in K theory. We refer to $\bs{w}$ as the {\bf left Weyl group action}. We recall next some basic facts,
 following \cite[\S 5]{MNS:left}. The action $\bs{w}$ is a ring endomorphism of
 $\K_{\T}(\Fl(I))$ which twists the ground ring $\K_{\T}(\pt)$, i.e., for 
 $\kappa, \zeta \in  \K_{\T}(\Fl(I))$ and $\chi \in \K_{\T}(\pt)$
 \[ \bs{w}.(\chi \otimes \kappa \otimes \zeta) = w(\chi) \otimes \bs{w}(\kappa) \otimes \bs{w}(\zeta)\/. \] 
The left Weyl group determines (left) Demazure operators
defined as follows. For a positive simple root $\alpha_i$ with corresponding reflection
$s_i$, the Demazure operator $\bs{\delta}_i$ is defined by 
\[ \bs{\delta}_i = \frac{1}{1- \alpha_i^{-1}} (1- \alpha_i^{-1} \bs{s}_i) \/. \] 
The operators $\bs{\delta}_i$ satisfy $\bs{\delta}_i^2 = \bs{\delta}_i$ and the usual braid and commutation relations
for the simple reflections in $W$. Therefore for any $w \in W$ with a reduced decomposition
$w= s_{i_1} s_{i_2} \cdot \ldots \cdot s_{i_k}$ there is a well defined operator 
$\bs{\delta}_w = \bs{\delta}_{i_1} \cdot \ldots \cdot \bs{\delta}_{i_k}$. 
The Demazure operators also satisfy a Leibniz rule, see \cite[\S 5]{MNS:left}. 

The action $\bs{w}$ fixes classes in 
$\K_{\GL_n}(\Fl(I))$, and therefore the Demazure
operator is linear with respect to such classes: 
for $\zeta  \in \K_{\GL}(\Fl(I))$ and $\kappa  \in \K_{\T}(\Fl(I))$,
\[ \bs{w}.(\zeta) = \zeta \quad \textrm{ and } \quad \bs{\delta}_w (\zeta \otimes \kappa) = \zeta \otimes \bs{\delta}_w(\kappa) \/. \]
For us, the most important examples of classes in $\K_{\GL}(\Fl(I))$
are the classes of homogeneous vector bundles on $\Fl(I)$ 
(such as the Schur bundles of tautological subbundles and quotient bundles on $\Fl(I)$),
and the classes of $\GL_n$-homogeneous line bundles 
$\mathcal{L}_\chi = \GL_n \times^{P_I} \C_{\chi}$
associated to characters $\chi$ of $P_I$, 
the parabolic subgroup stabilizing $\bfe_{id}$ (i.e., giving the identification
$\Fl(I) = \GL_n/P_I$). We record the following formulae proved in \cite[Prop. 5.5]{MNS:left}.

For $\alpha_i$ a simple root and $w \in W^I$, 
\begin{equation}\label{E:left-schub} \bs{s}_i.\cO_w = \begin{cases} \alpha_i \cO_w + (1-\alpha_i) \cO_{s_iw} & s_i w < w \/; \\
\cO_w & \textrm{otherwise} \/, \end{cases} \textrm{ and }
\bs{\delta}_i \cO_w = \begin{cases} \cO_{s_i w} & s_i w < w \/; \\
\cO_w & \textrm{otherwise} \/. \end{cases} \end{equation}
(Compare with \eqref{leftWaction} below.)
Furthermore, for a fixed point class $\bfe_w \in \K_{\T}(\Fl(I))$, 
\begin{equation}\label{E:left-fixed} \bs{s}_i.\bfe_w = \bfe_{s_i w} \/, \end{equation}
implying also that the pairing from \eqref{E:pairing} is invariant under the action of $W$.
In all cases above, if $s_i w $ is not a minimal length representative in $W^I$ then
it is replaced by the unique element $w'$ which satisfies $w' \in W^I$ and $s_i w W_I = w' W_I$.\begin{footnote}{We warn the reader that the conventions on Schubert classes
 in this paper are opposite from those in \cite{MNS:left}, and to 
relate the two one uses the relation $\cO_w = \bs{w}_0.\cO^{w_0 w w_I}$. Furthermore,
the operator denoted here by $\bs{\delta}_i$ is denoted by $\delta_i^\vee$ in {\em loc.~cit.}}\end{footnote} 
The Demazure operators generate a (degenerate) version of the Hecke algebra of $W$, and equip 
$\K_{\T}(\Fl(I))$ with a structure of a cyclic (Hecke) module.   
 
For further use, we record the following expressions for the pairing in the case of the fixed point 
classes in $\K_{\T}(\Gr(k;n))$. Recall that $\bfe_\lambda$ is 
the fixed point corresponding to $w_\lambda. \langle f_{1}, \ldots, f_k \rangle$.
Then 
\[ \langle \bfe_\lambda, \bfe_\mu \rangle = \delta_{\lambda,\mu} w_\lambda \langle \bfe_\emptyset, \bfe_\emptyset \rangle\] 
and, with $T_{\bfe_\emptyset}(\Gr(k;n))$ denoting the tangent space at 
$\bfe_\emptyset$),
\[ \langle \bfe_\emptyset, \bfe_\emptyset \rangle = \lambda_{-1}T_{\bfe_\emptyset}^*(\Gr(k;n)) = \prod_{1 \le i \le k; k+1 \le j \le n} (1- \ve_i/\ve_j) \/.\]
 
\subsection{Level-Rank duality}\label{ss:level-rank} Fix $\Fl(I) = (1 \le i_1 < i_2 < \ldots < i_k \le n)$ and  
consider the flag manifold 
$\Fl(I)$ equipped with the tautological sequence of bundles:
\[ \mathcal{S}_{i_1} \subset \mathcal{S}_{i_2} \subset \ldots \subset 
\mathcal{S}_{i_{k}} \subset \C^n \/. \]
For $1 \le s \le k$ set $K_{i_s}:= (\C^n/F_{i_s})^* \subset (\C^n)^*$. Define
\[ \Theta: \Fl (i_1, \ldots, i_k; \C^n) \to \Fl(n-i_k, \ldots , n-i_1; (\C^n)^*) \] 
by
\[ \Theta (F_{i_1} \subset F_{i_2} \subset \ldots \subset F_{i_k} \subset \C^n) = 
(K_{n-i_k} \subset K_{n-i_{k-1}} \subset \ldots \subset K_{n-i_1} \subset (\C^n)^*) \/.\]
Then $\Theta$ is an isomorphism of flag manifolds. The action of $\T$ on $\C^n$ induces the
contragredient action on $(\C^n)^*$, giving an action of $\T$ on the `dual'
flag manifold $\Fl(n-i_k, \ldots , n-i_1; (\C^n)^*)$, and making $\Theta$ a $\T$-equivariant morphism.
Note however that $\Theta$ takes the standard flag in $\C^n$ to the {\em opposite} flag in $(\C^n)^*$, 
as follows from 
\begin{equation*}\label{E:theta-flag} \Theta \left( \langle f_n, \ldots , f_{n-k+1} \rangle \subset \C^n \right) = 
\left( \langle f_1^*, \ldots , f_{n-k}^* \rangle \subset (\C^n)^* \right)
\end{equation*}
where $f_1^*, \ldots, f_n^* \in (\C^n)^*$ is the dual basis. 
In particular, the Schubert point $X_{w_0}$ in $\Fl(I)$ is sent to the {\em opposite} Schubert point
in $\Fl(n-i_k, \ldots , n-i_1; (\C^n)^*)$, and, more generally,
a Schubert variety $X_w(F_\bullet)$ is sent to one for the opposite flag
$X_{w^t}(F_\bullet^{opp})$. Here $w^t$ is the `transpose' of $w$, obtained
by taking the Dynkin automorphism exchanging the simple reflections $s_i$ to $s_{n-i+1}$.
For instance, $\Theta:\Gr(k;\C^n) \to \Gr(n-k;(\C^n)^*)$ satisfies
\[ \Theta(X_\lambda(F_\bullet)) = X_{\lambda^t}(\Theta(F_\bullet)) \/,\]
where $\lambda^t$ is the transpose of $\lambda$. 

Denote by $\mathcal{K}_{n-i_s}:= (\C^n/\cS_{i_s})^\vee \subset (\C^n)^*$ the 
tautological subbundle on the dual flag manifold
$\Fl(n-i_k, \ldots , n-i_1; (\C^n)^*)$. Since for any $u,w \in W$, 
$\bs{u}.[\cO_{X_w(F_\bullet)}]= [\cO_{X_w(\bs{u}.F_\bullet)}]$,
it follows that the map $\Theta$ induces an isomorphism of equivariant K-theory rings
\begin{equation}\label{E:level-rank-iso} \Theta^*: \K_{\T}(\Fl(n-i_k, \ldots , n-i_1; (\C^n)^*)) \to \K_{\T}(\Fl(i_1, \ldots, i_k; \C^n)) \end{equation}
which satisfies 
\begin{equation*} \Theta^*(\mathcal{K}_{n-i_s}) = (\C^n/\cS_{i_s})^\vee \textrm{ and } \Theta(\cO_w)= \bs{w}_0.\cO_{w^t} \end{equation*}
for $1 \le s \le k$, $1 \le i \le n$, where the Schubert classes are taken with respect to the standard flag.

Since the $\T$-modules $(\C^n)^*$ and $\C^n$ have distinguished ordered bases, one can further 
identify them by $f_i^* \leftrightarrow f_{n+1-i}$. 
The composition of the two transformations results in an isomorphism 
\[ \Theta': \Fl (i_1, \ldots, i_k; \C^n) \to \Fl(n-i_k, \ldots , n-i_1; \C^n) \]
called the {\em level-rank duality}. Note that $\Theta'$ preserves the standard 
flag, and it is equivariant
with respect to the map $\varphi:\T \to \T$ sending 
 $\ve_i \mapsto \ve_{n+1-i}^{-1}$. 
It satisfies the identities:
\begin{equation}\label{E:theta-schub} \Theta'^*(\mathcal{K}_{n-i_s}) = (\C^n/\cS_{i_s})^\vee \textrm{ and } \Theta'(\ve_i \otimes \cO_w)= \ve_{n+1-i}^{-1} \otimes \cO_{w^t} \end{equation}
In other words, $\Theta'$ sends a Schubert class to its transpose, and it takes equivariant paremeters
to their inverses, and also reverses their order. 
Compare with \eqref{Gamma} below, where the same transformation arises in the 
context of representations of the Yang-Baxter algebra.

\subsection{Push-forward formulae of Schur bundles} Next we recall some results about cohomology of 
Schur bundles on Grassmann bundles. Our main reference is 
\cite[\S 3]{GMSZ:QK}, which in turn follows
Kapranov's paper \cite{kapranov:Gr}. 
For more on Schur bundles see Weyman's book \cite{weyman}.

Recall that if $X$ is a $\T$-variety, $\pi:E \to X$ is any 
$\T$-equivariant vector bundle of rank $e$, and 
$\lambda=(\lambda_1, \ldots, \lambda_k)$ 
is a partition with at most $e$ parts, 
the {\em Schur bundle} $\mathfrak{S}_\lambda(E)$ is a 
$\T$-equivariant vector bundle over $X$. It has the property
that if $x \in X$ is a $\T$-fixed point, the fibre $(\fS_\lambda(E))_x$ 
is the $\T$-module with character the Schur function $s_\lambda$. For example,
if $\lambda = (1^k)$, then $\fS_{(1^k)}(E) = \wedge^k E$, and if $\lambda = (k)$ then
$\fS_{(k)}(E)=\mathrm{Sym}^k(E)$.

In this paper $X=\Gr(k;\C^n)$ with the $\T$-action restricted from
$\mathrm{GL}_n(\C)$. 
To emphasize the $\T = (\C^*)^n$-module structure on
$\C^n$, we will occasionally denote by $V:=\C^n$ and
by $\Gr(k;V)$ the corresponding Grassmannian.
Further, $V$ will also be identified with the trivial, 
but not equivariantly trivial, vector bundle $\Gr(k;V) \times V$.
The following was proved in \cite{kapranov:Gr}.
\begin{prop}[Kapranov]\label{lemma:BWB} Consider the 
Grassmannian $\Gr(k;V)$ with the tautological sequence 
$0 \to \cS \to V \to \cQ \to 0$. 
For any nonempty partition {$\lambda= (\lambda_1\geq \lambda_2 \geq \ldots \geq \lambda_k \geq 0)$ such that 
$\lambda_1 \le n-k$}, 
there are the following isomorphisms of $\T$-modules:

(a) For all $i \ge 0$, $H^i(\Gr(k;V), \fS_\lambda(\cS)) = 0$.

(b) \[ H^i(\Gr(k;V), \fS_\lambda(\mathcal{S}^\vee)) = \begin{cases} \fS_\lambda(V^*) & i=0 \\ 0 & i>0 \/. \end{cases}\]

(c) For all $i \ge 0$, $H^i(\Gr(k;V), \fS_\lambda(\cQ^\vee)) = 0$.

(d) \[ H^i(\Gr(k;V), \fS_\lambda(\cQ)) = \begin{cases} \fS_\lambda(V) & i=0 \\ 0 & i>0 \/. \end{cases}\]

\end{prop}
\begin{proof} Parts (a) and (b) were proved in \cite[Prop. 2.2]{kapranov:Gr}, 
as a consequence of
the Borel-Weil-Bott theorem on the complete flag manifold.
For parts (c) and (d), from level-rank duality, 
there is a $\T$-equivariant isomorphism 
$\Gr(k;V) \simeq \Gr(\dim V - k; V^*)$ 
under which the $\T$-equivariant bundle $\cS$ is sent to $\cQ^\vee$,
and $\cQ$ is sent to $\cS^\vee$. Then parts (c),(d) follow from (a) and (b) respectively.
\end{proof} 

We also need the following immediate consequence,
see e.g. \cite[Cor. 3.3]{GMSZ:QK}. Consider an $\T$-variety $X$ equipped with 
a $\T$-equivariant
vector bundle $\mathcal{V}$ of rank $n$. Denote by 
$\pi: \mathbb{G}(k,\mathcal{V}) \to X$ the Grassmann bundle over $X$. It is equipped with a 
tautological sequence
$0 \to \underline{\cS} \to \pi^* \mathcal{V} \to \underline{\cQ} \to 0$ over $\mathbb{G}(k,\mathcal{V})$. The following corollary follows from \Cref{lemma:BWB},
using that $\pi$ is a $T$-equivariant locally trivial fibration:

\begin{cor}\label{cor:relativeBWB}
For any nonempty partition {$\lambda= (\lambda_1\geq \lambda_2 \geq \ldots \geq \lambda_k \geq 0)$ such that 
$\lambda_1 \le n-k$}, there are the following isomorphisms of $T$-modules:

(a) For all $i \ge 0$, the higher direct images, 
$R^i \pi_* \fS_\lambda(\underline{\cS}) = R^i \pi_* \fS_\lambda(\underline{\cQ^\vee})= 0$. 

(b) \[ R^i \pi_*\fS_\lambda(\underline{\cS}^\vee)) = \begin{cases} \fS_\lambda(\mathcal{V}^\vee) & i=0 \\ 0 & i>0 \/. \end{cases} \/; \quad R^i \pi_*\fS_\lambda(\underline{\cQ}) = \begin{cases} \fS_\lambda(\mathcal{V}) & i=0 \\ 0 & i>0 \/. \end{cases} \]
\end{cor}   
\section{Grothendieck polynomials and expansions of $\lambda_y(\cQ^\vee)$}\label{sec:groth-exp} 
To goal of this section is to prove \Cref{thm:lyQveeGr}, 
giving the Schubert expansions of the class
$\lambda_y(\cQ_{n-k}^\vee) \in \K_{\T}(\Gr(k;n))$. 
To better orient the reader, we will index the tautological
bundles by their ranks; then for $\Gr(k;n)$, the tautological
sequence is $0 \to \cS_k \to \C^n \to \cQ_{n-k} \to 0$.  

To start, we recall a formula about the double Grothendieck polynomials given by column partitions, 
from \cite[Prop. 2.9]{gorbounov2017quantum}. To state it, consider sequences of variables
$x=(x_1, \ldots, x_N)$ and $t=(t_1, \ldots, t_N)$. The (double) Grothendieck polynomial for the column partition
$1^r$ is:
\begin{equation*}\label{E:defgroth} G_{1^r}(x|t) = \sum_{j=1}^{N+1-r} 
\frac{ \prod_{i=1}^N (x_i+t_j-x_i t_j)}{\prod_{i=1,i\neq j}^{N+1-r} \frac{t_j-t_i}{1-t_i}} \/. \end{equation*}
In our case we will need to take $N=n-k$, and, to switch to exponential Chern classes, 
we perform the change of variables
$x_i = 1-Y_i^{-1}, t_i = 1- \ve_{n+1-i}^{-1}$ to obtain:
\begin{equation}\label{E:deglocgroth} G_{1^r}(1-Y^{-1}|1-\ve^{-1}) =
\sum_{j=1}^{n-k+1-r} 
\frac{\prod_{i=1}^{n-k} (1-Y_i^{-1} \ve_{n+1-j}^{-1})}{\prod_{i=1,i\neq j}^{n-k+1-r} (1-\ve_{n+1-i}/\ve_{n+1-j})}\/. \end{equation}
Note that this does not depend on the variables $\ve_{k+1}, \ldots, \ve_{k+r-1}$. 
To relate this polynomial
to geometry, consider the sequence of bundles 
\[ \C \subset \C^2 \subset \ldots \subset \C^n \to \cQ_{n-k} \]
Let $Y_1, \ldots, Y_{n-k}$ be the exponentials of the Chern roots of 
$\cQ_{n-k}$, and recall that $\ve_n, \ldots, \ve_1$ (in this order) are the 
exponentials of the Chern roots of the sequence $\C \subset \ldots \subset \C^n$. 
Consider the (Thom-Porteous) degeneracy locus giving the Schubert variety $X_r$, i.e.,
\[ X_r = \{ V \in \Gr(k;n): \mathrm{rank}(\C^{n-k-r+1} \to \C^n/V) \le n-k-r \} \]
Using the K-theoretic Thom-Porteous formula from \cite[Thm. 2.3]{buch:groth-classes}, 
\begin{equation}\label{E:Kdeglocus} \cO_r = G_{1^{r}}(1 - Y_1^{-1}, \ldots, 1-Y_{n-k}^{-1}|1 - \ve_n, \ldots, 1- \ve_{k+r}) \/. \end{equation}
We now use \cite[Lemma 2.10]{gorbounov2017quantum}
(with $u=1+y$ in {\em loc.cit.}) to obtain:
\begin{equation}\label{E:Groth-exp}
\begin{split} & \prod_{i=1}^{n-k} (1+yY_i^{-1}) =  \prod_{i=k+1}^{n} (1+y \ve_i^{-1}) \\
 & - 
\sum_{r=1}^{n-k} y\ve_{k+r}^{-1} \Bigl(\prod_{i=k+r+1}^{n} (1+y \ve_i^{-1}) \Bigr)
G_{1^r}(1-Y_1^{-1}, \ldots, 1-Y_{n-k}^{-1} | 1-\ve_n, \ldots, 1-\ve_{k+r}) \/. 
\end{split} \end{equation}
Since $\lambda_y(\cQ_{n-k}^\vee)=\prod_{i=1}^{n-k} (1+yY_i^{-1}) $ this proves:
\begin{prop}\label{thm:lyQveeGr} The following holds in $\K_{\T}(\Gr(k;n))$:
\[ \lambda_y(\cQ^\vee_{n-k}) = \prod_{i=k+1}^{n} (1+y \ve_i^{-1}) \cO_\emptyset
 -\sum_{r=1}^{n-k} y\ve_{k+r}^{-1} \Bigl(\prod_{i=k+r+1}^{n} (1+y \ve_i^{-1}) \Bigr)
\cO_{r} \/. \]
\end{prop}
\noindent In particular, for $y=-\ve_n$, this proposition implies that:
\begin{equation}\label{E:y=-en} \lambda_{-\ve_n}(\cQ_{n-k}^\vee) = \cO_{n-k} \/. \end{equation} 
\begin{cor}\label{cor:special-mult} Let $d \ge 1$. Then in $\K_{\T}(\Gr(k+d;n))$,
\[ \lambda_y(\cQ^\vee_{n-k-d}) \cdot \cO_{(n-k-d)^k}= 
\prod_{i=d+1}^{n-k} (1+\frac{y}{\ve_i}) \cO_{(n-k-d)^k}
 -\sum_{r=1}^{n-k-d} \frac{y}{\ve_{d+r}} \Bigl(\prod_{i=d+r+1}^{n-k} (1+\frac{y}{\ve_i}) \Bigr)
\cO_{((n-k-d)^k,r)} \/. \]
\end{cor}
\begin{proof} Note that $X_{(n-k-d)^k} \subset \Gr(k+d;n)$ is the Grassmannian 
$\Gr(d,\C^n/\langle f_{n}, \ldots, f_{n-k+1} \rangle) = 
\Gr(d,\langle f_{n-k}, \ldots , f_{1} \rangle)$. Furthermore, 
the restriction of $\cQ_{n-k-d}$ to this Grassmannian 
is the corresponding quotient bundle.
Then the claim follows from \Cref{thm:lyQveeGr}. 
\end{proof}

\begin{remark}\label{rmk:groth-classes} Using the level-rank isomorphism 
$\Theta':\Gr(k;\C^n) \simeq \Gr(n-k; \C^n)$ from \eqref{ss:level-rank},
one can express the Schubert classes 
in terms of the exponential Chern roots $X_1, \ldots , X_k$ of $\cS_k$. 
{Applying the level rank duality to \eqref{E:deglocus} (with the effect of 
reversing the order of the equivariant parameters, 
and taking inverses of equivariant parameters) 
one obtains 
\[ {\cO_{1^r}= G_{1^r}(1-X_1, \ldots, 1-X_{n-k}|1-\ve_1^{-1}, \ldots, 1-\ve_{n-k-r+1}^{-1})\/,} \]
in $\K_{\T}(\Gr(n-k;\C^n))$.}
{More generally, using the factorial
Grothendieck polynomial $G_\lambda(x|t)$ defined by 
McNamara \cite{mcnamara:factorial}, and 
its relation to Schubert classes as obtained by combining Lemma 2.11, Corollary 4.12 (especially eq. (4.40)) in \cite{gorbounov2017quantum}, one obtains that in any $\K_{\T}(\Gr(k;n))$,
\[ \cO_\lambda =  G_\lambda(1-X_1, \ldots, 1-X_k|1-\ve_1^{-1}, \ldots, 1- \ve_n^{-1}) \/. \]
The same result may also be obtained from the geometry of double
Grothendieck polynomials studied in \cite{buch:groth-classes}, see, e.g. \S 5 in {\em loc.cit.}
}\end{remark}
\section{Quantum K-theory and `quantum=classical'}\label{sec:QK}
\subsection{Definitions and notation} The quantum K-ring was defined by Givental and Lee 
\cite{givental:onwdvv, lee:QK}. We recall 
the definition below, following \cite{givental:onwdvv}.~The (small) {\bf quantum K-pairing} 
is a deformation of the usual K-theory pairing; we recall the definition of this pairing for $X= \Gr(k;n)$. 
For any $\kappa_1, \kappa_2 \in \K_{\T}(\Gr(k,n))$, 
\begin{equation}\label{E:qKmetric} (\kappa_1 , \kappa_2)_{\QK} 
= \sum_{d \ge 0} q^d \langle \kappa_1, \kappa_2 \rangle_d \quad \in \K_{\T}(\pt)[[q]]\/;\end{equation} 
then extend this by $\K_{\T}(\pt)[\![q]\!]$-bilinearity. The elements 
$\langle \kappa, \zeta \rangle_d \in \K_T(pt)$ denote the
$2$-point (genus $0$, equivariant) K-theoretic Gromov-Witten (KGW) invariants 
if $d>0$. If $d=0$ then this is the usual pairing in $\K_{\T}(\Gr(k;n))$.
The ($n$-point, genus $0$) KGW invariants 
$\langle \kappa_1, \ldots, \kappa_n \rangle_d \in \K_T(\pt)$
are defined by pulling
back via the evaluation maps, then integrating, over the Kontsevich moduli space of stable maps 
$\overline{\mathcal{M}}_{0,n}(\Gr(k;n),d)$. Instead of recalling the precise definition of the KGW invariants, 
in \Cref{thm:q=cl} below
we give a `quantum=classical' statement calculating $2$ and $3$-point KGW invariants of Grassmannians.  
Explicit combinatorial formulae for the $2$-point KGW invariants for any homogeneous space 
may be found in \cite{buch.m:nbhds,BCLM:euler}.

\begin{thm}[Givental \cite{givental:onwdvv}] Consider the $\Kpt[\![q]\!]$-module 
\[ \QK_{\T}(\Gr(k;n)):= \K_{\T}(\Gr(k;n)) \otimes \Kpt[\![q]\!] \/. \]
Define the (small) equivariant quantum K-product $\star$ on $\QK_{\T}(\Gr(k;n))$ by the equality
\begin{equation}\label{E:QKdefn}(\kappa_1 \star \kappa_2, \kappa_3)_{\QK} = \sum_{d \ge 0} q^d \langle \kappa_1,\kappa_2,\kappa_3 \rangle_d \quad \in \K_{\T}(\pt)[\![q]\!] \/,\end{equation}
for any $\kappa_1, \kappa_2,\kappa_3 \in \K_{\T}(\Gr(k;n))$. Then $(\QK_T(\Gr(k;n)), +,\star)$ is a commutative, associative $\Kpt[\![q]\!]$-algebra. Furthermore, the small quantum K-metric
gives it a structure of 
a Frobenius algebra, i.e. $(\kappa_1 \star \kappa_2, \kappa_3)_{\QK} = (\kappa_1,\kappa_2 \star \kappa_3)_{\QK} $.
\end{thm}
\begin{remark} It was proved in \cite{BCMP:qkfin} that the submodule 
$\K_{\T}(\Gr(k;n)) \otimes \Kpt[q]$ is stable under the QK product $\star$. This means 
that the product of two Schubert classes has structure constants which are polynomials in 
$q$. Similar statements hold for any flag manifold \cite{ACTI:finiteness,kato:loop}. 
However, working over the ring of formal power series in $q$
has the advantage of having inverses of elements such as $\det \cS$
in the quantum K-ring.\end{remark}

As proved in \cite[\S 8]{MNS:left}, the ring endomorphisms on $\K_{\T}(\Gr(k;n))$ 
given by the Weyl group elements extend by $q$-linearity to endomorphisms 
of the equivariant quantum K-theory ring 
$\QK_{\T}(\Gr(k;n))= \K_{\T}(\Gr(k;n)) \otimes_{\Kpt} \Kpt[\![q]\!]$ 
(and in fact the equivariant quantum K-ring of any $G/P$). Using this, one can 
define by the same formula Demazure operators 
\[ \bs{\delta}_w \in \End_{\Q[\![q]\!]}(\QK_{\T}(\Gr(k;n))) \/.\] 
These satisfy the same properties as the classical ones, in particular the linearity with respect 
to classes in $\K_{\GL_n}(\Gr(k;n))$ of the Demazure operators, and the same formulae as in 
\eqref{E:left-schub} for the actions on Schubert classes. Then $\QK_{\T}(\Gr(k;n))$ is again 
a cyclic module over the appropriate degenerate Hecke algebra. 
(The main difference is that the Leibniz rule satisfied by the Demazure operators 
uses the quantum K-product, see \cite[Prop. 8.1]{MNS:left}.) We note that the QK pairing is $W$-invariant, in the sense that for any $w \in W$, and any $\kappa_1,\kappa_2 \in \QK_{\T}(\Gr(k;n))$,
\[ (\bs{w}.\kappa_1, \bs{w}.\kappa_2)_{\QK}= w(\kappa_1, \kappa_2)_{\QK} \/.\]

\subsection{Quantum=classical} We recall next the `quantum = classical' statement,
which relates the ($3$-point, genus $0$) equivariant KGW invariants on Grassmannians to a `classical' calculation in the equivariant K-theory of a two-step flag manifold. This statement was proved in \cite{buch.m:qk}, and it generalized 
results of Buch, Kresch and Tamvakis \cite{buch.kresch.ea:gromov-witten} from quantum cohomology. The proofs
rely on the `kernel-span' technique introduced by Buch \cite{buch:qcohgr}. A Lie-theoretic approach, 
for large degrees $d$, and for the larger family of cominuscule Grassmannians, was obtained in 
\cite{chaput.perrin:rationality}; see also \cite{BCMP:projgw} for a `quantum=classical' statement utilizing projected 
Richardson varieties. 

We recall next the `quantum = classical' result, proved in \cite{buch.m:qk}, and 
which will be used later in this paper. To start, form the 
following incidence diagram:
\begin{eqnarray}\label{E:qcl-all}
\xymatrix{Z_d:=\Fl(k-d,k,k+d;n)\ar[r]^{{\quad \quad \bar{p}}}\ar[d]^{{q}} & \Fl(k-d,k;n) \ar[r]^{\quad \bar{p}_1} \ar[d]^{\bar{p}_2} &\Gr(k;n) \\
Y_d:=\Fl(k-d,k+d;n) \ar[r]^{\quad \quad {pr}} & \Gr(k-d;n) &}
\end{eqnarray}
Here all maps are the natural projections. Denote by $p: \Fl(k-d,k,k+d;n) \to \Gr(k;n)$
the composition $p:=\bar{p}_1 \circ \bar{p}$. If $d\ge k$ then we set $Y_d:= \Fl(k+d;n)$ and
if $k+d \ge n$ then we set $Y_d:= \Gr(k-d;n)$. In particular, if $d \ge \min \{ k, n-k \}$, then
$Y_d$ is a single point.

\begin{thm}\label{thm:q=cl} Let $a,b,c \in \K_{\T}(\Gr(k;n))$ and $d \ge 0$ a degree.

(a) The following equality holds in $\K_{\T}(pt)$:
\[ \langle a, b, c \rangle_d = \int_{Y_d} q_*(p^*(a))\cdot q_*(p^*(b))\cdot q_*(p^*(c)) \/. \]

(b) Assume that $q_*(p^*(a)) = pr^*(a')$ for some $a' \in \K_{\T}(\Gr(k-d;n)$. Then 
\[ \langle a, b, c \rangle_d = \int_{\Gr(k-d;n)} a' \cdot (\bar{p}_2)_*(\bar{p}_1^*(b))\cdot (\bar{p}_2)_*(\bar{p}_1^*(c)) \/. \]
A similar statement holds when one  considers the analogous
diagram \eqref{E:qcl-all} with $\Gr(k-d;n)$ replaced by $\Gr(k+d;n)$.
\end{thm}
Observe that part (b) follows from (a) and the fact that the left 
diagram is a fibre square; for details
see \cite{buch.m:qk}. Recall
the tautological sequence on $\Gr(k;n)$: 
\[ 0 \to \cS=\cS_k \to  \mathbb{C}^n \to  \cQ:=\mathbb{C}^n/\cS_k \to 0 \/. \] 
To lighten notation, we will denote by the same letters the bundles on various flag manifolds
from the diagram \eqref{E:qcl-all}, but we will indicate in the subscript the rank of the bundle in question.

An important consequence of the `quantum=classical' statement is a formula to calculate 
quantum K-products; see \cite{buch.m:qk} and further \cite{BCMP:qkfin,BCMP:rcfin} for more general formulae. Consider any class 
$\kappa \in \K_{\T}(\Gr(k;n))$, a Schubert class
$\cO_\lambda$, and consider the multiplication
\[ \kappa \star \cO_\lambda = \sum_{d, \nu} N_{\kappa, \lambda}^{\nu,d} q^d \cO_\nu \/. \]  
Then for $d>0$ the coefficient $N_{\kappa, \lambda}^{\nu, d}$ may be calculated in terms of $3$-point
KGW invariants as follows:   
\begin{equation}\label{E:QKstr} N_{\kappa, \lambda}^{\nu,d} = 
\langle \kappa, \cO_\lambda, \cO_\nu^\vee \rangle_d - 
\sum_{\mu} \langle \kappa, \cO_\lambda, \cO_\mu^\vee \rangle_{d-1} \cdot \langle \cO_\mu, \cO_\nu^\vee \rangle_{1} \/. \end{equation}
Here where $\cO_\nu^\vee$ is the dual of $\cO_\nu$ with respect to the classical pairing
in $\K_{\T}(\Gr(k;n))$.

\section{Quantum K-products by $\lambda_y$ classes}\label{sec:QK-by-ly}
The goal of this 
section is to prove a formula for the quantum K-multiplication of the form
$\lambda_y(\cQ^\vee) \star [\pt]_{k,n}$,  which will be utilized 
later to interpret the diagonal operators on the monodromy matrix $T$ in the Yang-Baxter algebra.
Here we used the notation $[\pt]_{k,n}$ for the Schubert point $X_{(n-k)^k}=\langle f_{n-k+1}, \ldots, f_n \rangle$.

Consider the commutative diagram
\begin{eqnarray}
\label{diag:q=cl2}
\xymatrix{\Fl(k-d,k,k+d;n)\ar[r]^{{\quad p_1'}}\ar[d]^{{q}} & \Fl(k,k+d;n) \ar[r]^{\quad p_1} \ar[d]^{p_2} &\Gr(k;n) \\
\Fl(k-d,k+d;n) \ar[r]^{\quad \pi} & \Gr(k+d;n) &}
\end{eqnarray}
where all maps are the natural projections. Denote by $p= p_1 \circ p_1'$ and we abuse notation to denote by 
$\cS_i$ the tautological subbundle of rank $i$ on the two or three step manifold.
It should be clear from the context what is the base of this bundle. 

Fix a partition $\lambda$ included in the $k \times (n-k)$ rectangle
and an integer $d \ge 0$. Define the following combinatorial operations on 
$\lambda$: 
\begin{enumerate} \item[(A)] The partition $\lambda^{r[-d]}$ 
is obtained from $\lambda$
by removing the top $d$ rows. If $d=1$ we will simply denote it by 
$\lambda^r$.

\item[(B)] The partition $\lambda^{c[-d]}$ 
is obtained from $\lambda$ by removing the leftmost $d$ columns.
Again, if $d=1$, we denote this operation by $\lambda^c$.

\item[(C)] The partition $\lambda[-d]=(\lambda^{r[-d]})^{c[-d]}$ 
denotes the removal of the top 
row and the leftmost column. In other words, this is 
the composition of two operations above. 
\end{enumerate}
Extend all the operations to any class $\sum a_\lambda \cO_\lambda$ in 
$\K_{\T}(\Gr(k;n))$ by linearity. 
We refer to $X_{\lambda[-d]}$ as the 
{\em $d$-th curve neighborhood of $X_\lambda$}; for more genral constructions, 
see  \cite{BCMP:qkpos}. 
The curve neighborhoods may be used to calculate any two-point
KGW invariants (cf.~\cite{buch.m:nbhds}): for any degree $d \ge 0$ and any partitions $\lambda, \mu$,
\begin{equation}\label{E:2KGW} \langle \cO_\lambda, \cO_\mu^\vee \rangle_d = 
\delta_{\lambda[-d],\mu} \/. \end{equation}

Next we recall some formulae for convolutions. 
\begin{lemma}\label{lemma:push-pulls} The following hold:

(a) $(p_2)_*(p_1)^*(\cO_\lambda) = \cO_{\lambda^{c[-d]}}$;

(b) $(p_1)_*(p_2)^*(\cO_\lambda) = \cO_{\lambda^{r[-d]}}$;

(c) $(p_1)_*(p_2)^* (p_2)_*(p_1)^*(\cO_\lambda) = \cO_{\lambda[-d]} $.
\end{lemma}
\begin{proof} This follows from \Cref{lemma:proj}, then tracing the (pre)images of the 
Schubert varieties under the maps $p_1, p_2$.\end{proof}

The lemma allows us to interpret the operations $\lambda^{r[-d]}, \lambda^{c[-d]}$ and
$\lambda[-d]$ geometrically in terms of push-pull operators. For any $\kappa \in \K_{\T}(\Gr(k;n))$,
\[ \kappa^{r[-d]} = (p_1)_*(p_2)^*(\kappa)\/; \quad \kappa^{c[-d]} = (p_2)_*(p_1)^*(\kappa)\/; 
\quad \kappa[-d] = (p_1)_*(p_2)^* (p_2)_*(p_1)^*(\kappa) \/. \]

The next lemma upgrades these statements to push forwards and pull backs of $\lambda_y$ classes.
\begin{lemma}\label{lemma:pqQvee} Let $d\ge 1$. The following hold:

(a) $q_* (p^* \lambda_y(\mathcal{Q}^\vee_{n-k})) = \lambda_y(\cQ_{n-k-d}^\vee)$. 

(b) $q_* (p^*\lambda_y(\cQ_{n-k}^\vee)) = \pi^*(\lambda_y(\cQ_{n-k-d}^\vee))$.

(c) $(p_2)_* (p_1^* \lambda_y(\mathcal{Q}^\vee_{n-k})) = \lambda_y(\mathcal{Q}_{n-k-d}^\vee)$.

(d) $(p_1)_* (p_2^* \lambda_y(\mathcal{Q}^\vee_{n-k-d})) = 1$.

In all statements $\cQ_{n-k-d}=0$ if $n-k-d<0$.

Analogous statements hold when one replaces $\cQ^\vee$ by the tautological subbundle $\cS$. 
\end{lemma}
\begin{proof} This essentially follows from the formulae proved in 
\cite[\S 5]{GMSZ:QK}; for the convenience of the reader, we include
a proof. From definition $p^*(\lambda_y(\cS_k)) = \lambda_y(\cS_k)$. To calculate $q_* (\lambda_y((\mathbb{C}^n/\cS_k)^\vee))$ we observe that we have the short exact sequence on $\Fl(k-d,k,k+d;n)$
\begin{equation}\label{E:sesZ1} 0 \to (\mathbb{C}^n/\cS_{k+d})^\vee \to p^*(\cQ_{n-k}^\vee) = (\mathbb{C}^n/\cS_{k})^\vee \to 
(\cS_{k+d}/\cS_k)^\vee \to 0 \/. \end{equation}
From the Whitney formula it follows that 
\[ p^* \lambda_y(\mathcal{Q}_{n-k}^\vee)= \lambda_y(p^*(\cQ_{n-k}^\vee) ) = \lambda_y((\mathbb{C}^n/\cS_{k+d})^\vee) \cdot
\lambda_y((\cS_{k+d}/\cS_k)^\vee) \]
and then by the projection formula we deduce that 
\begin{equation}\label{E:W1} q_* (p^* \lambda_y(\mathcal{Q}_{n-k}^\vee)) = \lambda_y((\mathbb{C}^n/\cS_{k+d})^\vee) \cdot q_*(\lambda_y((\cS_{k+d}/\cS_k)^\vee)) \/. \end{equation}
Observe that 
the projection $q$ realizes the three-step flag manifold $\Fl(k-d,k,k+d;n)$ 
as the Grassmann bundle
$\mathbb{G}(\cS_{k+d}/\cS_{k-d})$ over $\Fl(k-d,k+d;n)$, with tautological sequence
\[ 0 \to \cS_{k}/\cS_{k-d} \to \cS_{k+d}/\cS_{k-d} \to \cS_{k+d}/\cS_{k} \to 0 \/. \]
\Cref{cor:relativeBWB}(b) implies that  
\[ q_*(\lambda_y(\cS_{k+d}/\cS_k)^\vee) = q_* [\cO_{\Fl(k-d,k,k+d;n)}] = [\cO_{\Fl(k-d,k+d;n)}] \/. \]
Then the claim in (a) follows from this and \eqref{E:W1}, and (b) follows because 
$\pi^*(\lambda_y(\cQ_{n-k-d})^\vee)= \lambda_y((\mathbb{C}^n/\cS_{k+d})^\vee)$. 

The claim in (c) follows along the same lines. More precisely, one uses again the short exact sequence in 
\eqref{E:sesZ1} and the projection formula to show that 
\begin{equation}\label{E:W2} (p_2)_* (p_1^* \lambda_y(\mathcal{Q}_{n-k}^\vee)) = 
\lambda_y(\cQ_{n-k-d}^\vee) \cdot (p_2)_*(\lambda_y((\cS_{k+d}/\cS_k)^\vee)) \/. \end{equation}
The projection $p_2$ realizes the two-step flag manifold $\Fl(k,k+1;n)$ as the Grassmann bundle
$\mathbb{G}(k,\cS_{k+d})$ with tautological sequence
$0 \to \cS_{k} \to \cS_{k+d} \to \cS_{k+d}/\cS_{k} \to 0$.
Then one applies \Cref{cor:relativeBWB}(b) to show that 
$(p_2)_*(\lambda_y((\cS_{k+d}/\cS_k)^\vee)) = [\cO_{\Gr(k+d;n)}]$. 
This finishes the proof of (c). 

Finally, to prove (d) one observes that $p_1:\Fl(k,k+d;n) \to \Gr(k;n)$ realizes the two-step flag manifold
as the Grassmann bundle $\mathbb{G}(\C^n/\cS_k)$ with tautological sequence 
$0 \to \cS_{k+d}/\cS_k \to \C^n/\cS_k \to \C^n/\cS_{k+d} \to 0$. Since
$\C^n/\cS_{k+d}=p_2^* \lambda_y(\mathcal{Q}^\vee_{n-k-d})$, the claim
follows again from \Cref{cor:relativeBWB}(b).
\end{proof}

\begin{cor}\label{cor:q=clQvee} Let $a,b \in \K_{\T}(\Gr(k;n))$ and $d \ge 1$. Then there is an equality: 
\[ \langle a, b, \lambda_y(\mathcal{Q}^\vee_{n-k}) \rangle_d = \int_{\Gr(k+d;n)} 
(p_2)_*(p_1^*(a)) \cdot (p_2)_*(p_1^*(b)) \cdot \lambda_y(\cQ_{n-k-d}^\vee) \/. \]
\end{cor}
\begin{proof} This follows from  \Cref{thm:q=cl} and \Cref{lemma:pqQvee}.\end{proof}
\begin{thm}\label{prop:vanishingQvee} Consider the multiplication
\[ \lambda_y(\cQ^\vee_{n-k}) \star [\pt]_{k,n} = \sum_{d, \nu} N_{\kappa, \lambda}^{\nu,d} q^d \cO_\nu \/. \]  
Then $N_{\kappa, \lambda}^{\nu, d}=0$ for any $d \ge 2 $.
\end{thm}
\begin{proof} From \eqref{E:QKstr} it suffices to show that for any $\cO_\lambda \in \K_{\T}(\Gr(k;n))$, and
\begin{equation}\label{E:vanish} \langle \lambda_y(\mathcal{Q}^\vee_{n-k}), [\pt]_{k;n}, \cO_\lambda^\vee \rangle_d = 
\sum_{\nu} \langle \lambda_y(\mathcal{Q}^\vee_{n-k}), [\pt]_{k;n}, \cO_\nu^\vee \rangle_{d-1} \cdot
\langle \cO_\nu, \cO_\lambda^\vee \rangle_{1} \/. \end{equation}
We calculate both sides of this equation. By \Cref{cor:q=clQvee} and \Cref{lemma:push-pulls}(a), 
the left hand side is equal to
\[ \chi_{\Gr(k+d;n)}\Bigl(\lambda_y(\cQ_{n-k-d}^\vee) \cdot \cO_{(n-k-d)^k} \cdot (p_2)_*(p_1^*)(\cO_\lambda^\vee)\Bigr) \/. \]
Recall from \Cref{cor:special-mult} that 
\[ \lambda_y(\cQ^\vee_{n-k-d}) \cdot \cO_{(n-k-d)^k}= 
\prod_{i=d+1}^{n-k} (1+\frac{y}{\ve_i}) \cO_{(n-k-d)^k}
 -\sum_{r=1}^{n-k-d} \frac{y}{\ve_{d+r}} \Bigl(\prod_{i=d+r+1}^{n-k} (1+\frac{y}{\ve_i}) \Bigr)
\cO_{((n-k-d)^k,r)} \/. \]
Taking each Schubert class in this sum and using projection formula and \Cref{lemma:push-pulls}(b), 
one obtains that the left hand side of \eqref{E:vanish} is equal to:
\[ \begin{split} \chi_{\Gr(k+d;n)}(\cO_{((n-k-d)^k,r)} \cdot (p_2)_*(p_1^*)(\cO_\lambda^\vee)) & =
\chi_{\Gr(k;n)} (\cO_{((n-k-d)^{k-d},r)} \cdot \cO_\lambda^\vee ) \\
& = \delta_{((n-k-d)^{k-d},r), \lambda}
\end{split} \]
Therefore the left hand side of \eqref{E:vanish} is equal to
\begin{equation}\label{E:vanish-lhs} \begin{cases} \prod_{i=d+1}^{n-k} (1+\frac{y}{\ve_i}) & \lambda = (n-k-d)^{k-d} \/;\\
- \frac{y}{\ve_{d+r}} \Bigl(\prod_{i=d+r+1}^{n-k} (1+\frac{y}{\ve_i}) \Bigr) & \lambda = ((n-k-d)^{k-d},r) \/; \\
0 & \textrm{otherwise}\/. 
\end{cases}
\end{equation}
A similar calculation, using now that $\langle \cO_\nu, \cO_\lambda^\vee \rangle_{1} = \delta_{\nu[-1],\lambda}$, 
yields that the right hand side of \eqref{E:vanish} is equal to:
\begin{equation}\label{E:vanish-rhs} \begin{cases} \prod_{i=d}^{n-k} (1+\frac{y}{\ve_i}) \delta_{\nu[-1],\lambda} & \nu = (n-k-d+1)^{k-d+1} \/;\\
- \frac{y}{\ve_{d-1+s}} \Bigl(\prod_{i=d+s}^{n-k} (1+\frac{y}{\ve_i}) \Bigr) \delta_{\nu[-1],\lambda} & \nu = ((n-k-d+1)^{k-d+1},s) \/; \\
0 & \textrm{otherwise}\/. 
\end{cases}
\end{equation}
Here $1 \le s \le n-k-d+1$. The hypothesis $d\ge 2$ implies that the set of $\nu$'s of the 
form $((n-k-d+1)^{k-d+1},s)$ is non-empty. 

If $\lambda = (n-k-d)^{k-d}$ then there are two $\nu$'s such that $\nu[-1]=\lambda$ and 
$\nu = ((n-k-d+1)^{k-d+1},p)$. These arise for $p=0,1$, and their contribution to the right hand side is
\[ \prod_{i=d}^{n-k} (1+\frac{y}{\ve_i}) - \frac{y}{\ve_{d}} \Bigl(\prod_{i=d+1}^{n-k} (1+\frac{y}{\ve_i}) \Bigr)
= \prod_{i=d+1}^{n-k} (1+\frac{y}{\ve_i}) \/, \]
which confirms the two sides are equal in this case. If $\lambda = ((n-k-d)^{k-d},r)$ for some 
$1 \le r \le n-k-d$, then there is exactly one $\nu$ with the required constraints, of the form
$\nu = ((n-k-d+1)^{k-d+1},r+1)$; one easily checks the contributions from \eqref{E:vanish-lhs} and 
\eqref{E:vanish-rhs} are equal. The other possibilities for $\lambda$ contribute with $0$ in both sides. 
\end{proof}

For later use we denote the Euler class at $[\pt]_{k,n}$ by  
\begin{equation}\label{E:euler-pt} e_{k;n}=  \prod_{j=1}^{n-k} (1+
\frac{y}{\ve_j}) \/.\end{equation}

\begin{thm}\label{thm:lyQveept} 
The multiplication $\lambda_y(\cQ_{n-k}^\vee) \star [\pt]_{k,n} $ in $\mathrm{QK}_{\T}(\Gr(k;n))$ is given by 
\[ \begin{split} \lambda_y(\cQ_{n-k}^\vee) \star [\pt]_{k,n} & = 
e_{k;n} [\pt]_{k,n} - q \sum_{i=0}^{n-k-1} \frac{y}{\ve_{i+1}} \prod_{s=i+2}^{n-k} (1+ \frac{y}{\ve_s}) \cO_{(n-k-1)^{k-1},i)} 
\\ & = \prod_{j=1}^{n-k} (1+ \frac{y}{\ve_j}) \cO_{(n-k)^k} - q \sum_{i=0}^{n-k-1} \frac{y}{\ve_{i+1}} \prod_{s=i+2}^{n-k} (1+ \frac{y}{\ve_s}) \cO_{((n-k-1)^{k-1},i)} \/. \end{split}
\]
\end{thm}
\begin{proof} \Cref{prop:vanishingQvee} implies that we only need to consider
terms corresponding to the $q$-powers $q^0$ and $q^1$. The classical term
is clearly the one claimed. From \eqref{E:QKstr} the coefficient of $q\cO_\lambda$ is equal to 
\begin{equation}\label{E:q1} \langle \lambda_y(\mathcal{Q}^\vee_{n-k}), [\pt]_{k;n}, (\cO_\lambda)^\vee \rangle_1 - 
\sum_{\nu} \langle \lambda_y(\mathcal{Q}^\vee_{n-k}), [\pt]_{k;n}, (\cO_\nu)^\vee \rangle_{0} \cdot
\langle \cO_\nu, (\cO_\lambda)^\vee \rangle_{1} \/. \end{equation}
The terms in this expression have been calculated in \eqref{E:vanish-lhs} respectively \eqref{E:vanish-rhs}.
If $\lambda= (n-k-1)^{k-1}$ then the contribution is
\[ \prod_{i=2}^{n-k} (1+\frac{y}{\ve_i}) -  \prod_{i=1}^{n-k} (1+\frac{y}{\ve_i}) =
 -\frac{y}{\ve_{1}} \prod_{i=2}^{n-k} (1+\frac{y}{\ve_i}) \/, \]
 as claimed. The only other possible $\lambda$'s must be of the form
$((n-k-1)^{k-1},i)$ with $1 \le i \le n-k-1$. Their contribution arises from \eqref{E:vanish-lhs}
(with $d=1$), and it is equal to 
\[ - \frac{y}{\ve_{1+i}} \prod_{s=i+2}^{n-k} (1+\frac{y}{\ve_s}) \]
as claimed. This finishes the proof.
\end{proof}

\begin{cor}\label{cor:deg1-lambda} Let $\lambda$ be any partition in the $k \times (n-k)$ rectangle. Then the maximal power of $q$ in the multiplication $\lambda_y(\cQ_{n-k}^\vee) \star \cO_\lambda$ is at most $1$.\end{cor}
\begin{proof} By \Cref{thm:lyQveept}, the claim holds in the case $\cO_\lambda = [\pt]_{k;n}$.
For the general case, one may pick a left divided difference operator $\partial_{w_\lambda}$ such that
$\partial_{w_\lambda}[\pt]_{k;n} = \cO_\lambda$. Using that $\partial_{w_\lambda}$ commutes with
$\lambda_y(\cQ_{n-k}^\vee)$ we obtain:
\[ \begin{split} \lambda_y(\cQ_{n-k}^\vee) \star \cO_\lambda & = \lambda_y(\cQ_{n-k}^\vee) \star \partial_{w_\lambda}[\pt]_{k;n} \\
& = \partial_{w_\lambda} ( \lambda_y(\cQ_{n-k}^\vee) \star [\pt]_{k;n}) \/. \end{split} \]
The claim follows from this. 
\end{proof}
We will see below that the coefficients of $q^0$ and $q^1$ in the multiplication 
$\lambda_y(\cQ_{n-k}^\vee) \star \cO_\lambda$ recovers the expressions for the diagonal operators
$\tg_{00}$ and $\tg_{11}$ acting on $\cO_\lambda$. 


\section{On and off-diagonal operators from convolutions}\label{sec:convo-ops}
The goal of this section is to define the convolutions $\tau_{ij}$, which give the geometric counterparts of the entries in the monodromy matrix $T=(t_{ij})$ from \eqref{T}.
To make the connection between geometry and the integrable system, 
the key calculation is the action of  
these operators on the class of the Schubert point. Since 
$\QK_{\T}(\Gr(k;n))$ is a cyclic (degenerate) Hecke algebra generated by 
the class of the Schubert point, this determines the action on all other Schubert classes.

As advertised, the geometric operators arise as convolutions, and the attentive reader will observe that 
these are precisely the convolutions calculating the quantum K-product by $\lambda_y(\cQ^\vee)$
using the `quantum=classical' diagrams.

Recall the diagram: 
\begin{eqnarray}
\label{diag:q=cl}
\xymatrix{\Fl(k,k+1;n)\ar[d]_{p_1}\ar[r]^{{p_2}} & \Gr(k+1;n)\\
\Gr(k;n) & } \/, 
\end{eqnarray}
where all maps are the projection maps. All these maps are $\GL_n$-equivariant.
Next is our geometric definition of the coefficients $\tau_{ij}$ in the monodromy matrix.  

For each $0 \le k \le n$, define the `diagonal' operators $\tau_{00}, \tau_{11}$ as the endomorphisms in $\End(\K_{\T}(\Gr(k;n)))$ given by:
\[ \tau_{00} (\kappa) = \lambda_y(\cQ_{n-k}^\vee) \cdot \kappa \/. \]
The endomorphism $\tau_{11}$ is defined as the difference  
\[ \tau_{11} = \tau_{11}^{(1)} - \tau_{11}^{(2)} \] 
where 
\[ \tau_{11}^{(1)} (\kappa) = (p_1)_* p_2^*((p_2)_* p_1^*(\kappa) \cdot \lambda_{y}(\mathcal{Q}^\vee_{n-k-1})) \/; \] 
\[ \tau_{11}^{(2)} (\kappa) = (p_1)_* p_2^*(p_2)_* p_1^*(\kappa \cdot \lambda_y(\mathcal{Q}^\vee_{n-k})) \/. \] 
Define the `off-diagonal' operators $\tau_{10}: \K_{\T}(\Gr(k;n)) \to \K_{\T}(\Gr(k+1;n))$ and
 $\tau_{01}: \K_{\T}(\Gr(k+1;n)) \to \K_{\T}(\Gr(k;n))$ by:
\[ \tau_{10}(\kappa) = \lambda_y (\mathcal{Q}_{n-k-1}^\vee) \cdot (p_2)_*(p_1)^*(\kappa) - (p_2)_* p_1^*(\lambda_y (\mathcal{Q}_{n-k}^\vee) \cdot \kappa) \/. \]
\[ \tau_{01}(\kappa) = (p_1)_* (p_2)^*(\lambda_y (\mathcal{Q}_{n-k-1}^\vee) \cdot \kappa) \/.\]
(We will not use this, but note that $\tau_{11} = (p_1)_*p_2^*(\tau_{10})$.)
Since $\cQ_{n-k},\cQ_{n-k-1}$ are homogeneous bundles, their $\lambda_y$ classes commute 
with the left Weyl group action. In turn, this implies that the operators $\tau_{ij}$ also 
commute with the left Weyl group action, in the sense that for any $w \in W$, and any class $\kappa$,
\[ \tau_{ij}(\bs{w}.\kappa)= \bs{w}. \tau_{ij}(\kappa) \/. \]
The next result gives the main calculation needed to identify the operators $\tau_{00},\tau_{11}$ with the diagonal operators from the integrable system. 
\begin{thm}\label{thm:diag-ops-geom} The endomorphism
$\tau := \tau_{00} + q \tau_{11}$
is equal to the operator of quantum K-multiplication by $\lambda_y(\cQ_{n-k}^\vee)$. That is, for any 
$\kappa \in \K_{\T}(\Gr(k;n))$,  
\[ \tau(\kappa) = \lambda_y(\cQ_{n-k}^\vee) \star \kappa \/. \]
In particular, 
\[\tau_{00}[\pt]_{k,n} = \prod_{i=1}^{n-k} (1+ y \ve_i^{-1})[\pt]_{k,n} \/, \]
and
\[ \tau_{11}[\pt]_{k,n} = -\sum_{i=0}^{n-k-1} y\ve_{i+1}^{-1} \prod_{s=i+2}^{n-k} (1+ y\ve_s^{-1}) \cO_{(n-k-1)^{k-1},i)} \/. \]
\end{thm}
\begin{proof} The statement follows from the calculation of
$\lambda_y(\cQ_{n-k}^\vee) \star \kappa$ based on the 
`quantum=classical' statement. More precisely, by \Cref{cor:deg1-lambda}, 
the multiplication in question has only $q$-powers at most equal to $1$.
The classical term
is clearly the one claimed. Take an arbitrary partition $\lambda$ in the $k \times (n-k)$ rectangle. 
From \eqref{E:QKstr} the coefficient of $q\cO_\lambda$ in
$\lambda_y(\cQ_{n-k}^\vee) \star \kappa$ is equal to 
\begin{equation}\label{E:q1-lambda} \langle \lambda_y(\mathcal{Q}^\vee_{n-k}), \kappa, (\cO_\lambda)^\vee \rangle_1 - 
\sum_{\nu} \langle \lambda_y(\mathcal{Q}^\vee_{n-k}), \kappa, (\cO_\nu)^\vee \rangle_{0} \cdot
\langle \cO_\nu, (\cO_\lambda)^\vee \rangle_{1} \/. \end{equation}
From the quantum=classical statement in \Cref{cor:q=clQvee}, and by applying repeatedly the projection formula, and the cohomological triviality of the projection maps from \Cref{lemma:proj}, we obtain that
\[ \langle \lambda_y(\mathcal{Q}^\vee_{n-k}), \kappa, \cO_\lambda^\vee \rangle_1 = \chi( t_{11}^{(1)}(\kappa), \cO_\lambda^\vee) \/. \]
To calculate the sum in the second part, observe that by \eqref{E:2KGW} the two-point KGW invariants
may be calculated using the curve neighbourhoods: 
$\langle \cO_\nu, \cO_\lambda^\vee \rangle_{1}= \delta_{\nu[-1],\lambda}$. Then 
\[ \begin{split} \sum_{\nu} \langle \lambda_y(\mathcal{Q}^\vee_{n-k}), \kappa, \cO_\nu^\vee \rangle_{0} \cdot
\langle \cO_\nu, \cO_\lambda^\vee \rangle_{1} & = \sum_{\nu} \chi(\lambda_y(\mathcal{Q}^\vee_{n-k}) \cdot \kappa \cdot \cO_\nu^\vee) \delta_{\nu[-1],\lambda} \\
& = \chi \left( (\lambda_y(\mathcal{Q}^\vee_{n-k}) \cdot \kappa)[-1], \cO_\lambda^\vee \right)
\/.\end{split} \]
In other words, this is the coefficient of $\cO_\lambda$ in the curve neighborhood of
the multiplication $\lambda_y(\mathcal{Q}^\vee_{n-k}) \cdot \kappa$ in $\K_{\T}(\Gr(k;n))$. 
By \Cref{lemma:push-pulls},
this curve neighborhood is precisely $\tau_{11}^{(2)}(\kappa)$. Combining everything, we proved that 
the coefficient of $q \cO_\lambda$ in $\lambda_y(\mathcal{Q}^\vee_{n-k}) \star \kappa$ is equal
$\chi( (\tau_{11}^{(1)} - \tau_{11}^{(2)})(\kappa), \cO_\lambda^\vee)$, which proves the claim about $\tau$.

Finally, the formulae for $\tau_{00}[\pt]_{k;n}$ and $\tau_{11}[\pt]_{k;n}$ follow from the multiplication
$\lambda_y(\cQ_{n-k}^\vee) \star [\pt]_{k;n}$ from \Cref{thm:lyQveept}.
\end{proof}

The next theorem gives the actions of the `off-diagonal' operators $\tau_{01},\tau_{10}$ on the class of the point. This key result will be used to identify these operators with those
from \Cref{thm:int-system-ops}.

\begin{thm}\label{thm:geom-offdiag} The following hold:
\begin{enumerate} \item[(a)] $\tau_{10}[\pt]_{k;n} = -\sum_{r=1}^{n-k}y\ve_{r}^{-1}\prod_{i=r+1}^{n-k}(1+y\ve_i^{-1})\cO_{((n-k-1)^k,r-1)}$;

\item[(b)] $\tau_{01}[\pt]_{k+1;n}= \prod_{i=1}^{n-k-1}(1+y\ve_i^{-1})\cO_{(n-k-1)^{k}}$.

\end{enumerate}
\end{thm} 
In particular, part (1) of the theorem implies that if $y=-\ve_{n-k}$, then 
\[ (\tau_{10}[\pt]_{k;n})_{{y=-\ve_{n-k}}} = [\pt]_{k+1;n} \/. \]
\begin{proof} To prove (a), we calculate  
\[ \begin{split} \tau_{10}[\pt]_{k,n} 
& = \lambda_y (\mathcal{Q}_{n-k-1}^\vee) \cdot (p_2)_*(p_1)^*[\pt]_{k;n} - (p_2)_* p_1^*(\lambda_y (\mathcal{Q}_{n-k}^\vee) \cdot [\pt]_{k;n}) \\ 
& = 
\lambda_y (\mathcal{Q}_{n-k-1}^\vee) \cdot \cO_{(n-k-1)^k} - \prod_{i=1}^{n-k}(1+y\ve_i^{-1})
\cO_{(n-k-1)^{k}} \\
& = 
\prod_{i=2}^{n-k} (1+\frac{y}{\ve_i}) \cO_{(n-k-1)^k}\\
 &\hspace{1cm} -\sum_{r=1}^{n-k-1} \frac{y}{\ve_{r+1}} \Bigl(\prod_{i=r+2}^{n-k} (1+\frac{y}{\ve_i}) \Bigr)
\cO_{((n-k-1)^k,r)} - \prod_{i=1}^{n-k}(1+\frac{y}{\ve_i})
\cO_{(n-k-1)^{k}} \\
& =
- y\ve_1^{-1} \prod_{i=2}^{n-k} (1+y\ve_i^{-1}) \cO_{(n-k-1)^k} 
 -\sum_{r=1}^{n-k-1} y\ve_{r+1}^{-1} \Bigl(\prod_{i=r+2}^{n-k} (1+y\ve_i^{-1}) \Bigr)
\cO_{((n-k-1)^k,r)} 
\/. \end{split} \]
Here the second equality follows from \Cref{lemma:push-pulls} and the third from \Cref{cor:special-mult}.
The last expression is the same as that in part (a).
Finally, part (b) follows from definition and \Cref{lemma:push-pulls}.
\end{proof}

\subsection{Dual operators}\label{ss:dual-ops} The operators $\tg_{ij}$ involve $\cQ^\vee$, the (vector bundle) 
dual of the tautological quotient bundle. There is also a dual theory which involves 
the tautological subbundles $\cS$, and it is defined in terms of 
operators $\tilde{\tg}_{ij}$ acting on $\bbV_n$. The dual theory 
also arises in the context of integrable systems, see \Cref{sec:int-level-rank} below. 
The definition of the operators $\tilde{\tg}_{ij}$, and the proof of their properties (in analogy to 
$\tg_{ij}$) is an exercise in judiciously applying the level-rank duality from 
\Cref{ss:level-rank}. For the convenience of the reader, we give the precise 
definitions, but we leave out most of the details of proof.

The definitions of $\tilde{\tau}_{ij}$'s are similar to those for $\tau_{ij}$, except that the order of arrows is reversed. More precisely, define endomorphisms $\tilde{\tg}_{00},\tilde{\tg}_{11}$ of $\K_{\T}(\Gr(k;n))[y]$ by the condition that:
\[ \tilde{\tau}_{00}+q \tilde{\tau}_{11}:\QK_{\T}(\Gr(k;n))[y] \to \QK_{\T}(\Gr(k;n))[y]\/; \quad \cO_\lambda \mapsto \cO_\lambda \star \lambda_y(\cS_k) \/. \]
For the `off-diagonal' entries, recall the diagram \eqref{diag:q=cl}, and define the convolution operators:
\[ \tilde{\tg}_{01}=\tilde{\tg}_{01}(y): \K_T(\Gr(k+1,n))[y] \to \K_T(\Gr(k,n))[y] \] 
by 
\[ \tilde{\tg}_{01}(\kappa) = \lambda_y (\cS_{k}) \cdot (p_1)_*(p_2)^*(\kappa) - (p_1)_* p_2^*(\lambda_y (\cS_{k+1}) \cdot \kappa) \/, \]
and
\[ \tilde{\tg}_{10}=\tilde{\tg}_{10}(y): \K_T(\Gr(k,n))[y] \to \K_T(\Gr(k+1,n))[y] \]
by
\[ \tilde{\tg}_{10}(\kappa) = (p_2)_* (p_1)^*(\lambda_y (\cS_{k}) \cdot \kappa) \/. \]

\begin{thm}\label{thm:geom-dual-ops} The operators $\tilde{\tg}_{ij}[\pt_{k;n}]$
satisfy the same formulas as those satisfied by $\tilde{t}_{ij}(v_{(n-k)^k})$ in \Cref{thm:int-system-ops} below.\end{thm}
\begin{proof} One proof follows from the observation that the `geometric' level rank duality from \eqref{E:theta-schub} fits with the equality \eqref{Gamma} developed in integrable systems.~In particular, the identity \eqref{tGamma} also holds, which proves the claim.

Alternatively, one can do this using a geometric argument. We start with the diagonal operators
$\tilde{\tg}_{00}$ and $\tilde{\tg}_{11}$. In this case,
the claim follows from applying the level-rank duality to the formulae calculating  ${\tg}_{00}[\pt_{k;n}]$ and ${\tg}_{11}[\pt_{k;n}]$: indeed, the level rank duality extends without changes to quantum K-theory.

For the off diagonal operators, using the previous formulae for 
$\lambda_y(\cQ^\vee)$ acting on the class of the (Schubert) point, 
along with the level-rank duality, one obtains similar formulae for the action of 
$\lambda_y(\cS)$ on the point. Once multiplications
by $\lambda_y(\cS)$ are obtained, the `dual' push forward formulae from \Cref{lemma:push-pulls}
and \Cref{lemma:pqQvee} may be applied, and similar arguments as those in \Cref{thm:geom-offdiag}
yield the claim.
\end{proof}

\section{The Yang-Baxter algebra}\label{sec:YB} 
In this section we introduce a noncommutative and non-cocommutative 
Hopf algebra $\YB=\YB(R)$ in terms of a solution of the quantum 
Yang-Baxter equation, called the $R$-matrix, which is intimately 
related to the left Weyl group action in quantum K-theory. The 
Hopf algebra $\YB$ can be seen as a degenerate version of the 
dual Hopf algebra of the Drinfeld-Jimbo quantum group $U_q(\mathfrak{gl}_2[z^{\pm 1}])$ which is a quasi-triangular Hopf algebra,
see~e.g.~\cite[Ch.~7 and~12]{chari1995guide}.

Historically, quantum groups originated from the work of Faddeev, 
Reshetikhin, Takhtajan on the {\em quantum inverse scattering method} and {\em quantum integrable systems}, such as quantum spin chains \cite{faddeev1990lectures}. 
In their approach, one starts from the $R$-matrix to define an associative 
unital algebra in terms of quadratic relations which are written in matrix form;
the so-called $RTT$-relations. In the special case of $q$-deformed universal 
enveloping algebras such as $U_q(\mathfrak{gl}_2)$ 
as introduced by Drinfel'd and Jimbo, this approach can be understood as a $q$-deformation 
of the algebra of functions on $\GL_2$. Viewed as a Hopf algebra, the latter 
is dual to the universal enveloping algebra 
$U(\mathfrak{gl}_2)$. 
The case we consider here is a degenerate version of the deformation of the loop algebra 
$\mathfrak{gl}_2[z^{\pm 1}]$.

\subsection{Definitions}

Consider $\C^2$ and denote by $\{v_0,v_1\}$ its standard basis. Given some indeterminate $u$ define the following function $R(z)$ with values in $\End(\C^2\otimes\C^2)$,
\begin{equation}\label{R}
R(z)=\left(\begin{smallmatrix}1&0\\0&0\end{smallmatrix}\right)\otimes
    \left(\begin{smallmatrix} 1&0\\0&1-z\end{smallmatrix}\right)+
    \left(\begin{smallmatrix}0&1\\0&0\end{smallmatrix}\right)\otimes
    \left(\begin{smallmatrix}0&0\\1&0\end{smallmatrix}\right)+
    \left(\begin{smallmatrix}0&0\\1&0\end{smallmatrix}\right)\otimes
    \left(\begin{smallmatrix}0&z\\0&0\end{smallmatrix}\right)+
    \left(\begin{smallmatrix}0&0\\0&1\end{smallmatrix}\right)\otimes
    \left(\begin{smallmatrix}0&0\\0&1\end{smallmatrix}\right)\;.
\end{equation}
Note that $R(1)=P$, the flip-operator, defined by $P(v_a\otimes v_b)=v_b\otimes v_a$. Fixing in $\C^2\otimes\C^2$ the ordered basis $\{v_0\otimes v_0,v_0\otimes v_1,v_1\otimes v_0,v_1\otimes v_1\}$, the same $R$-matrix is often written as the $4\times 4$ matrix
\[
R(z)=\begin{pmatrix}
    1 & 0 &0 &0\\
    0& 1-z & z&0\\
    0& 1 & 0&0\\
    0&0&0&1
\end{pmatrix}\;.
\]
This form of the $R$-matrix is particularly convenient when deriving the graphical calculus where we identify each non-zero matrix element with one of the following five vertex configurations,
\begin{equation}\label{5vR}
\avoiding{0}{0}\qquad
\connecting{0}{1}\qquad
\avoiding{1}{0}\qquad
\avoiding{0}{1}\qquad
\avoiding{1}{1}
\;.
\end{equation}
Here the labels $\alpha$ and $\beta$ on the West and North edge determine which vector $v_\alpha\otimes v_\beta$ the $R$-matrix acts on, while the values $\gamma$ and $\delta$ of the East and South edge prescribe the term $v_\gamma\otimes v_\delta$ in the expansion of the image.
\begin{lemma}
The $R$-matrix \eqref{R} solves the quantum Yang-Baxter equation (QYBE):
\begin{equation}\label{YBE}
R_{12}(z/w)R_{13}(z)R_{23}(w)=R_{23}(w)R_{13}(z)R_{12}(z/w)\;.
\end{equation}
Furthermore, the additional identities hold for the inverse $R$-matrix:
\begin{equation}\label{invR}
R(z)^{-1}=P\circ R(z^{-1})\circ P=(\sigmax\otimes\sigmax)R(z^{-1})(\sigmax\otimes\sigmax) \/;
\end{equation} 
here $\sigmax=\left(\begin{smallmatrix}0&1\\1&0\end{smallmatrix}\right)$.
\end{lemma}

Employing the $R$-matrix \eqref{R} we follow the same steps as in \cite{gorbounov2017quantum} and
define a Hopf algebra $\YB:=\YB(R)$ in several steps. 

The generators of $\YB$ will be denoted by $\{t_{ij}[r]~:~r\in\Z,\;i,j=0,1\}$ and are collectively written in terms of the currents 
\begin{equation}\label{t}
t_{ij}(z)=\sum_{r\in\Z}t_{ij}[r]z^{r}\in\YB[\![z^{\pm 1}]\!],\qquad\YB[\![z^{\pm 1}]\!]=\YB\otimes_\C\C[\![z^{\pm 1}]\!]\;
\end{equation}
and the following {\em monodromy matrix},
\begin{multline}\label{T}
T(z)=
\left(\begin{smallmatrix}
    1&0\\0&0
\end{smallmatrix}\right)\otimes t_{00}(z)+
\left(\begin{smallmatrix}
    0&1\\0&0
\end{smallmatrix}\right)\otimes t_{01}(z)+
\left(\begin{smallmatrix}
    0&0\\1&0
\end{smallmatrix}\right)\otimes t_{10}(z)+
\left(\begin{smallmatrix}
    0&0\\0&1
\end{smallmatrix}\right)\otimes t_{11}(u)\\
=\begin{pmatrix}
t_{00}(z) & t_{01}(z)\\
t_{01}(z)& t_{11}(z)
\end{pmatrix}\in \End (\C^2) \otimes\YB[[z^{\pm 1}]]\;.
\end{multline}
The (quadratic) defining relations of $\YB$ are encoded in the matrix identity
\begin{equation}\label{RTT}
R_{12}(z/w)T_1(z)T_2(w)=T_2(w)T_1(z)R_{12}(z/w)\;
\end{equation}
in $\End(\C^2\otimes\C^2)\otimes\YB[\![z^{\pm 1},w^{\pm 1}]\!]$, where we have set
\[
T_1(z)=\sum_{i,j=0,1}E_{ij}\otimes 1\otimes t_{ij}(z)\quad\text{and}\quad T_2(w)=\sum_{i,j=0,1}1\otimes E_{ij}\otimes t_{ij}(w)\;
\]
with $E_{ij}=(\delta_{ia}\delta_{jb})_{0\le a,b\le 1}$. The following lemma 
explicitly states some of the commutation relations of $\YB$ in terms 
of the generators \eqref{t}. The latter are used in the computation of the Bethe vectors, which we will identify with the quantum idempotents in $\QK_T(\Gr(k,n))$; see also \cite{gorbounov2017quantum}, but we warn the reader that we use a change of variable in this paper. 
\begin{lemma}\label{lem:RTTrelations}
The matrix equation \eqref{RTT} implies (among others) the identities:
\begin{gather}\label{YBcomm}
t_{ij}(z)t_{ij}(w)=t_{ij}(w)t_{ij}(z)\;,\qquad i,j=0,1\\
t_{11}(w)t_{00}(z)-t_{00}(z)t_{11}(w)=\frac{1}{1-z/w}\,t_{10}(z)t_{01}(w)-\frac{z/w}{1-z/w}\,t_{10}(w)t_{01}(z)\\
t_{00}(z)t_{10}(w)=\frac{1}{1-z/w}\,t_{10}(w)t_{00}(z)-\frac{1}{1-z/w}\,t_{10}(z)t_{00}(w)\\
t_{11}(z)t_{10}(w)=\frac{1}{1-w/z}\,t_{10}(w)t_{11}(z)-\frac{w/z}{1-w/z}\,t_{10}(z)t_{11}(w)\\
t_{00}(z)t_{01}(w)=(1-z/w)t_{01}(w)t_{00}(z)+\frac{z}{w}\,t_{01}(z)t_{00}(w)\\
t_{11}(z)t_{01}(w)=(1-w/z)t_{01}(w)t_{11}(z)+t_{11}(w)t_{01}(z)
\end{gather}
In particular, we have that for any central element $q$ the elements $t(z)=t_{00}(z)+qt_{11}(z)$ generate an abelian subalgebra in $\YB$, i.e. we have that $t(z)t(w)=t(w)t(z)$.
\end{lemma}
\begin{proof}
A straightforward computation where one inserts the definitions \eqref{R} and \eqref{T} into the identity \eqref{RTT} and then compares coefficients of the basis elements $E_{ij}\otimes E_{kl}$ in $\End(\C^2\otimes\C^2)$ on both sides of the equation.
\end{proof}
We require $\YB$ to be unital, denoting the identity by $1$. The associativity of $\YB$ follows 
from the following standard argument: if one considers the triple product 
$T_1(u)T_2(z)T_3(w)$, then one can successively exchange the 
$T$-matrices using \eqref{RTT}. Because $R$ satisfies \eqref{YBE}, 
the two possible choices of exchanging the three $T$-matrices 
in a different order must coincide and, hence, the algebra $\YB$ is associative.
\begin{defn}
We shall call the associative unital algebra $\YB$ defined in terms of the 
$R$-matrix \eqref{R} via \eqref{T}, \eqref{t}, \eqref{RTT}  
the {\em Yang-Baxter algebra} associated with $R$. 
\end{defn}

One advantage of presenting the algebra relations in the $RTT=TTR$ form \eqref{RTT} 
is that one can easily read off the algebra automorphisms from the definition 
\eqref{R} and the properties \eqref{invR} of the $R$-matrix. In particular, 
recall from \eqref{invR} the vector space isomorphism $\sigmax:\C^2\to\C^2$, 
which simply swaps the basis vectors $v_0$ and $v_1$. We have the following statement:
\begin{lemma}\label{lem:auto}
The map $T(z)\mapsto T(a\cdot z)$ with $a\in\C^*$ constitutes an algebra automorphism, while the map %
$
T(z)\mapsto \sigmax\cdot T(z)\cdot\sigmax$ gives rise to an algebra anti-automorphism.
\end{lemma}
\begin{proof}
From the defining relations \eqref{RTT} and \eqref{R} we infer that only the difference of the variables $u,v$ in the RTT-equation \eqref{RTT} enters the defining relations of the algebra and, thus, the first assertion now readily follows. Similarly, multiplying by $\sigmax\otimes\sigmax$ from both sides in \eqref{RTT} we obtain (using \eqref{invR})
\begin{multline*}
\sigmax_1\sigmax_2 R_{12}(z/w)T_1(z)T_2(w)\sigmax_1\sigmax_2=R_{21}(z/w)\sigmax_1 T_1(z)\sigmax_1\sigmax_2 T_2(w)\sigmax_2=\\
\sigmax_2 T_2(w)\sigmax_2 \sigmax_1 T_1(z)\sigmax_1 R_{21}(z/w)=\sigmax_1\sigmax_2 T_2(w)T_1(z)R_{12}(z/w)\sigmax_1\sigmax_2\;.
\end{multline*}
After switching factors in the tensor product by multiplying from both sides with the flip operator $P_{12}$ the anti-automorphism now follows. 
\end{proof}

We now introduce the structure of an Hopf algebra on $\YB$. 
Define the following algebra homomorphisms
\begin{equation}\label{Delta}
\Delta:\YB\to\YB\otimes\YB,\qquad\Delta(t_{ij}(z))=\sum_{k=0,1}t_{kj}(z)\otimes t_{ik}(z)
\end{equation}
and
\begin{equation}\label{counit}
\epsilon:\YB\to\C,\qquad\epsilon(t_{ij}(z))=\delta_{ij}
\end{equation}
where $\delta_{ij}$ is the Kronecker delta. For the representation we consider below in 
this paper, the formal inverse $T(z)^{-1}$ exists; see Lemma \ref{lem:Tinv}. This equips $\YB$ with a 
structure of a Hopf algebra with the antipode $S:\YB\to\YB$ defined by
\begin{equation}\label{S}
S(t_{ij}(z))=(T(z)^{-1})_{ij}\;.
\end{equation}
That this map constitutes an anti-automorphism can directly be seen from the RTT-relation by multiplying both sides with the inverse $T$-matrices,
\[
T_1(z)^{-1}T_2(w)^{-1}R_{12}(z/w)=R_{12}(z/w) T_2(w)^{-1}T_1(z)^{-1}\;.
\]
\begin{prop}
$(\YB,\Delta,\epsilon,S)$ is a non-commutative and a non-cocommutative Hopf algebra.
\end{prop}
\begin{proof} 
A straightforward albeit somewhat tedious checking of the axioms of a Hopf algebra which we omit.
\end{proof}

Note that we have made a choice in the definition of the coproduct \eqref{Delta}. The `opposite coproduct' defined by
\begin{equation}\label{Delta_op}
\Delta^{\rm op}:\YB\to\YB\otimes\YB,\qquad\Delta^{\rm op}(t_{ij}(z))=\sum_{k=0,1} t_{ik}(z)\otimes t_{kj}(z)
\end{equation}
leads to an alternative Hopf algebra structure $(\YB,\Delta^{\rm op},\epsilon,S)$. Obviously, the two coproducts $\Delta$ and $\Delta^{\rm op}$ are related by taking the matrix transpose of the monodromy matrix \eqref{T}, $T(z)^\sfT=(t_{ji}(z))_{i,j=0,1}$. This corresponds to considering the dual space $(\C^2)^*=\C v^0\oplus\C v^1$, where $\langle v^i,v_j\rangle=\delta_{ij}$ under the canonical pairing, and, hence considering the transpose of the $R$-matrix \eqref{R}. In fact, when computing the transpose in each factor, denoted by $\sfT\otimes\sfT$, we notice that the $R$-matrix \eqref{R} is not invariant, 
\begin{equation}\label{Ropp}
    R^\opp(z)=R(z)^{\sfT\otimes \sfT}=
    \left(\begin{smallmatrix}1&0\\0&0\end{smallmatrix}\right)\otimes
    \left(\begin{smallmatrix} 1&0\\0&1-z\end{smallmatrix}\right)+
    \left(\begin{smallmatrix}0&1\\0&0\end{smallmatrix}\right)\otimes
    \left(\begin{smallmatrix}0&0\\z&0\end{smallmatrix}\right)+
    \left(\begin{smallmatrix}0&0\\1&0\end{smallmatrix}\right)\otimes
    \left(\begin{smallmatrix}0&1\\0&0\end{smallmatrix}\right)+
    \left(\begin{smallmatrix}0&0\\0&1\end{smallmatrix}\right)\otimes
    \left(\begin{smallmatrix}0&0\\0&1\end{smallmatrix}\right)\;.
\end{equation}
Therefore, we arrive at a different, albeit closely related, Hopf algebra which we shall call the `opposite Yang-Baxter algebra' $\YB^\opp$. Namely,  noting that $R^\opp$ also solves the quantum Yang-Baxter equation \eqref{YBE}, we can as before define a Hopf algebra $\YB^\opp=\YB(R^\opp)$ with generators $t^\opp_{ij}(z)$ and monodromy matrix $T^\opp(z)=(t^\opp_{i,j}(z))_{i,j=0,1}$.

\begin{lemma}\label{lem:YB2YBopp}
    The map $T(z)\mapsto T^\opp(z)=T(z)^\sfT$ defines a Hopf algebra anti-isomorphism $\YB\to\YB^\opp $.
\end{lemma}
\begin{proof}
    Suppose that $T(z)$ satisfies the RTT-relation \eqref{RTT}. Taking the transpose in both factors on both sides of the equation we obtain
    \begin{equation*}
        T_1(z)^{\sfT}T_2(w)^{\sfT}R^\opp_{12}(z/w)=R^\opp_{12}(z/w) T_2(w)^{\sfT}T_1(z)^{\sfT}\;.
    \end{equation*}
    Thus, sending $t_{ij}(z)\mapsto t^\opp_{ij}(z)=t_{ji}(z)$ is an algebra anti-isomorphism. As for the co-product we note that
    \[
    \Delta(t_{ij}(z))^\opp=\sum_k t^\opp_{kj}(z)\otimes t^\opp_{ik}(z)=\Delta^{\rm op}(t^\opp_{ij}(z))\;.
    \]
    Checking the relations for the antipode and co-unit are similarly straightforward.
\end{proof}
 
 \begin{remark}\label{rmk:YBopp}\rm
 For the sake of brevity and in order not to overburden the reader with too many technical details, we shall limit our discussion mostly to the algebra $\YB$ and often omit the parallel discussion of the algebra $\YB^\opp$ where the results follow along analogous lines using Lemma \ref{lem:YB2YBopp}. 
 
 Geometrically, the algebra $\YB^\opp$ describes the ring structure of $\QK$ in the basis of quantum ideal sheaves, while $\YB$ is used to describe the ring structure in the basis of structure sheaves. This technical complication is absent from the case of quantum cohomology which is reflected in the fact that the corresponding $R$-matrix for quantum cohomology is symmetric, i.e. $R(z)^{\sfT\otimes \sfT}=R(z)$.
\end{remark}

\subsection{Evaluation modules and their tensor products}
It is well known that the representations of quasi-triangular Hopf algebras give rise to braided monoidal categories; see e.g. \cite[Ch. 4-5]{chari1995guide} and references therein. 
We now consider a particular class of finite-dimensional modules 
of the Yang-Baxter algebra $\YB$ and discuss the braiding of their tensor 
products in terms of the $R$-matrix \eqref{R}; see Lemma \ref{lem:Rcheck} and Corollary \ref{cor:YBbraid}. These modules will be later identified 
with the sum of (quantum) equivariant K-theory modules of Grassmannians and their braiding will be used to define a left action of the symmetric group which we will show to (1) coincide with the geometrically defined action in $K$-theory and (2) to commute with the action of $\YB$; see Proposition \ref{prop:leftWaction}.

\begin{lemma}\label{lem:evhom}
Let $\ve$ be some indeterminate. The map $\YB[\![z^{\pm 1}]\!]\to\End\C^2[\ve^{\pm 1}]\otimes\C[z^{\pm 1}]$ given via
\begin{equation}\label{L}
T(z)\mapsto  R(-z/\ve)\;
\end{equation} 
defines an algebra homomorphism $\YB\to\End\C^2[\ve^{\pm 1}]$. Explicitly, we have in terms of the generators:
\begin{gather}
    t_{00}(z)\mapsto\begin{pmatrix}1&0\\0&1\end{pmatrix}+
    \begin{pmatrix} 0&0\\0&z /\ve\end{pmatrix}\;,\qquad
    t_{01}(z)\mapsto\begin{pmatrix}0&0\\1&0\end{pmatrix}\\
    t_{10}(z)\mapsto \begin{pmatrix}0&-z/\ve\\0&0\end{pmatrix}\;,\qquad
    t_{11}(z)\mapsto\begin{pmatrix}0&0\\0&1\end{pmatrix}\;.
\end{gather}
\end{lemma}
\begin{proof}
In order to prove the assertion one needs to verify the relations \eqref{RTT}, but the latter are trivially satisfied because the $R$-matrix solves \eqref{YBE} and $T(z)\mapsto T(-z)$ is an algebra automorphism. 
\end{proof}
\begin{remark}\rm
   The sign change $z\mapsto y=-z$ is introduced for later convenience when we discuss the geometric interpretation of the Yang-Baxter algebra.
\end{remark}
We call the map \eqref{L} the {\em evaluation homomorphism} and 
$\C^2[\ve^{\pm 1}]$ a left {\em evaluation module}, because of the 
close analogy with the definition of such modules for the affine quantum group 
$U_q(\widehat{\mf{gl}}_2)$. We shall refer to $\ve$ as `evaluation parameter' 
as one usually evaluates it in the base field of the algebra, although here we will 
(eventually) identify it with one of the equivariant parameters.

Fix $n\in\N$ and $w\in W\simeq S_n$. Exploiting the coproduct structure of $\YB$ define 
the tensor product of evaluation modules:
\begin{equation}\label{left_tensor}
\bbV_w=\C^2[\ve^{\pm 1}_{w(n)}]\otimes\cdots\otimes\C^2[\ve^{\pm 1}_{w(2)}]\otimes\C^2[\ve^{\pm 1}_{w(1)}]\;.
\end{equation}
Namely, sending
\begin{equation}\label{evT}
T(z)\mapsto R_{0n}(-z/\ve_{w(1)})\cdots R_{02}(-z/\ve_{w(n-1)})R_{01}(-z/ \ve_{w(n)})\in\End(\C^2\otimes\bbV_w)\otimes\C[z]
\end{equation}
in \eqref{RTT} defines an algebra homomorphism $\op{YB}\to\End \bbV_w$ via the decomposition \eqref{T}. By abuse of notation we will keep using the same symbol for the matrix elements $t_{ij}(z)$ in $\YB_m[\![z^{\pm 1}]\!]$ and their images in $\End\bbV_w \otimes\C[z]$. 
\begin{example}\rm
    Let us demonstrate on the simplest example, $n=2$, how the map \eqref{evT} defines a representation $\YB\to \End\bbV_{s_1}$. Comparing with the definition \eqref{T} of the monodromy matrix we need to decompose the product $R_{02}(-z/\ve_2)R_{01}(-z/\ve_1)$ into a sum of the form
    \[
    R_{02}(-z/\ve_2)R_{01}(-z/\ve_1)=\sum_{i,j=0,1}E_{ij}\otimes t_{ij}(z) \/,
    \]
    with
    \[
    t_{ij}(z)\in\End(\C^2\otimes\C^2(\ve_2^{\pm 1})\otimes\C^2(\ve_1^{\pm 1}))\otimes\C[z]\;.
    \]
   Here the index 0 of the R-matrices refers to the first factor in the tensor product $\C^2\otimes\C^2(\ve_2^{\pm 1})\otimes\C^2(\ve_1^{\pm 1})$ and the indices 1 and 2 to the second and third factor respectively, indicating where each $R$-matrix acts non-trivially according to the expansion in \eqref{R}. Using the latter expansion and multiplying matrices in the first factor labelled 0 we find,
   \begin{multline*}
   R_{02}(-z/\ve_2)R_{01}(-z/\ve_1)=\\
   \left(\begin{smallmatrix}1&0\\0&0\end{smallmatrix}\right)\otimes
    \left(\begin{smallmatrix} 1&0\\0&1+z/\ve_1\end{smallmatrix}\right)\otimes
    \left(\begin{smallmatrix} 1&0\\0&1+z/\ve_2\end{smallmatrix}\right)
    +\left(\begin{smallmatrix}1&0\\0&0\end{smallmatrix}\right)\otimes
    \left(\begin{smallmatrix}0&-z/\ve_1\\0&0\end{smallmatrix}\right)\otimes
    \left(\begin{smallmatrix}0&0\\1&0\end{smallmatrix}\right)\\
    +\left(\begin{smallmatrix}0&1\\0&0\end{smallmatrix}\right)\otimes
    \left(\begin{smallmatrix}0&0\\1&0\end{smallmatrix}\right)\otimes
    \left(\begin{smallmatrix} 1&0\\0&1+z/\ve_2\end{smallmatrix}\right)
    +\left(\begin{smallmatrix}0&1\\0&0\end{smallmatrix}\right)\otimes
    \left(\begin{smallmatrix}0&0\\0&1\end{smallmatrix}\right)\otimes
    \left(\begin{smallmatrix} 0&0\\1&0\end{smallmatrix}\right)\\    +\left(\begin{smallmatrix}0&0\\1&0\end{smallmatrix}\right)\otimes
    \left(\begin{smallmatrix} 1&0\\0&1+z/\ve_1\end{smallmatrix}\right)\otimes
    \left(\begin{smallmatrix}0&-z/\ve_2\\0&0\end{smallmatrix}\right)
    +\left(\begin{smallmatrix}0&0\\1&0\end{smallmatrix}\right)\otimes
\left(\begin{smallmatrix}0&-z/\ve_1\\0&0\end{smallmatrix}\right)\otimes
\left(\begin{smallmatrix}0&0\\0&1\end{smallmatrix}\right)\\
    +\left(\begin{smallmatrix}0&0\\0&1\end{smallmatrix}\right)\otimes
    \left(\begin{smallmatrix}0&0\\1&0\end{smallmatrix}\right)\otimes
    \left(\begin{smallmatrix}0&-z/\ve_2\\0&0\end{smallmatrix}\right)\otimes
    +\left(\begin{smallmatrix}0&0\\0&1\end{smallmatrix}\right)\otimes
    \left(\begin{smallmatrix}0&0\\0&1\end{smallmatrix}\right)\otimes
    \left(\begin{smallmatrix}0&0\\0&1\end{smallmatrix}\right)\;.
   \end{multline*}
   Thus, when comparing with the expansion of $T(z)$ in \eqref{T} we find from the first line that
   \[
   t_{00}(z)\mapsto 
   \left(\begin{smallmatrix} 1&0\\0&1+z/\ve_1\end{smallmatrix}\right)\otimes
    \left(\begin{smallmatrix} 1&0\\0&1+z/\ve_2\end{smallmatrix}\right)
    +\left(\begin{smallmatrix}0&-z/\ve_1\\0&0\end{smallmatrix}\right)\otimes
    \left(\begin{smallmatrix}0&0\\1&0\end{smallmatrix}\right)
   \]
   which matches the action prescribed by the coproduct \eqref{Delta},
   \[
   \Delta\,t_{00}(z)=t_{00}(z)\otimes t_{00}(z)+t_{10}(z)\otimes t_{01}(z)
   \]
   and the evaluation modules when inserting the expressions from Lemma \eqref{lem:evhom} in each factor. Similarly, one checks the result for the remaining generators $t_{ij}(z)$.
\end{example}
\begin{lemma}\label{lem:Tinv}
    The inverse $T(z)^{-1}$ of the monodromy matrix exists in $\End\bbV_w$ and is given by
    \[
    T(z)^{-1}\mapsto R_{10}(-\ve_{w(n)}/z)R_{20}(-\ve_{w(n-1)}/z)\cdots R_{01}(-\ve_{w(n)}/z)\;.
    \]
\end{lemma}
\begin{proof}
This is immediate from the existence of the inverse of the $R$-matrix \eqref{invR} and \eqref{evT}.
\end{proof}

\begin{remark}\rm
In general it is difficult to write explicit formulae for the generators $t_{ij}(z)$ as in the case $n=1$. However, we will be able to describe the action of the matrix entries $t_{ij}(z)$ on the evaluation modules $\bbV_w$ in terms of the graphical calculus described in \S \ref{sec:graphical} below. This action can be identified with certain multiplication and convolution operators in the quantum K-rings; see \Cref{thm:int-system-ops} and \Cref{cor:main-cor}.
\end{remark}

\begin{remark}\label{rmk:Vopp}\rm
It follows from Lemma \ref{lem:YB2YBopp} that any left $\YB$-module is a right $\YB^\opp$-module and vice versa. Our definition of evaluation modules extends to $\YB^\opp$ along the same lines as discussed previously. However, one then considers instead the dual modules
 \begin{equation}\label{Vopp}
 \bbV_w^\opp:=(\C^2)^*[\ve^{\pm 1}_{w(1)}]\otimes\cdots\otimes(\C^2)^*[\ve^{\pm 1}_{w(n)}]
 \end{equation}
where we have reversed the ordering of the factors (as usual in the context of dual modules for Hopf algebras). Both sets of modules are related via the following isomorphism: given a $01$-word $J=j_1\ldots j_n$ and $f_i\in\C[\ve_i^{\pm 1]}]$, $i=1,\ldots,n$ define a $\C$-linear map $\opp:\bbV_w\to\bbV^\opp_{w}$ via
\begin{equation}\label{PDdef}
    f_n(\ve_n) v_{j_n}\otimes\cdots\otimes f_1(\ve_1)v_{j_1}\mapsto
    f_n(\ve_1) v^{j_1}\otimes\cdots\otimes f_1(\ve_n) v^{j_n}\;,
    \end{equation}
    where $\{v^{j}:j=0,1\}\subset(\C^2)^\ast$ denotes the dual basis to $\{v_j:j=0,1\}\subset\C^2$. 
    Then we have the following intertwining relation between the (left) actions of the Yang-Baxter algebras $\YB$ and $\YB^\opp$,
    \begin{equation}\label{Poincare}
        \opp\circ t_{ij}(z)=(t^\opp_{ji}(z))^\sfT\circ\opp\;,
    \end{equation}
    where $\sfT$ denotes the matrix transpose in $\End\bbV^\opp_{w}$. In the case of quantum cohomology this isomorphism is simply Poincar\'e duality and one has $\YB\cong\YB^\opp$ because the corresponding $R$-matrix in that case is symmetric, $R^\opp(z)=R(z)^{\sfT\otimes\sfT}=R(z)$. This ceases to be true in the case of quantum $K$-theory and, hence, two Yang-Baxter algebras and two types of modules are involved.
\end{remark}

We now discuss the braiding of the modules \eqref{left_tensor} which we will then identify with the geometric left Weyl group action below.

\begin{cor}\label{cor:YBbraid}
Let $u,w\in W$. Then the $\YB$-modules $\bbV_u$ and $\bbV_w$ are isomorphic and the isomorphism is given explicitly in terms of the $R$-matrix \eqref{R}.
\end{cor}
The proof of this corollary is immediate from the following lemma.
\begin{lemma}\label{lem:Rcheck}
Set $\check R(z)=P\circ R(z)$. Then 
\begin{equation}
\check R_{12}(\ve_1/\ve_2) R_{02}(-z/\ve_2)R_{01}(-z/\ve_1)=
R_{02}(-z/\ve_2)R_{01}(-z/\ve_2)\check R_{12}(\ve_1/\ve_2) \;.
\end{equation}
In particular, set $n=2$ and consider the tensor products $\bbV_{s_1}$ and $\bbV_{id}$. Then the matrix $\check R(\ve_1/\ve_2)$ gives an isomorphism $\bbV_{id}\to\bbV_{s_1}$ of $\YB$-modules.
\end{lemma}
\begin{proof}
From \eqref{YBE} it follows (after a relabelling of the spaces) that 
\[
R_{01}(z/\ve_1)R_{02}(z/\ve_2)R_{12}(\ve_1/\ve_2)=
R_{12}(\ve_1/\ve_2)R_{02}(z/\ve_2)R_{01}(z/\ve_1)\;.
\]
Multiplying with the flip operator $P_{12}$ from the left on both sides and replacing $z$ with $-z$ the first assertion now follows. 

To prove the second claim recall that exploiting the coproduct structure of $\YB$ the map $T(z)\mapsto R_{02}(-z/\ve_2)R_{01}(-z/\ve_{1})$ gives $\bbV_{s_1}$ and  $T(z)\mapsto R_{02}(-z/\ve_1)R_{01}(-z/\ve_{2})$ gives the representation $\bbV_{id}$ for $n=2$. The second claim now follows from the first noting that the $\check R$-matrix is invertible, according to \eqref{invR}.
\end{proof}
\begin{defn} A left $\YB$-module $M$ is called 
highest weight if there exists a nonzero vector $v_o\in M$ such that the following properties hold:
\begin{equation}\label{lefthw}
\text{(i)}\quad t_{ii}(z).v_o=\mu_i(z)v\qquad\text{ and }\qquad\text{(ii)}\quad t_{ij}(z).v_o=0,\quad i<j
\end{equation}
for some $\mu_i\in\C[\![z^{\pm 1}]\!]$ and $i,j=0,1$. 
\end{defn}

\begin{prop}\label{prop:pseudo}
(i) All evaluation modules and their tensor products $\bbV_w$ are highest weight $\YB$-modules. Specifically, we have for the (left) $\YB$-module $\bbV_w$ defined in \eqref{evT} that
\begin{equation}\label{pseudom}
v_o=v_1\otimes\cdots\otimes v_1\in(\C^2)^{\otimes n}\;
\end{equation}
and
\begin{equation}\label{leftOm}
t_{00}(z).v_o=\prod_{k=1}^{n}(1+z/\ve_k)v_o,\quad
t_{01}(z).v_o=0,\quad
t_{11}(z).v_o=v_o\;.
\end{equation}
(ii) The images $t_{ij}(z)$ in $\End\bbV_w\otimes\C[z]$ have at most degree $n$ in $z$, i.e. for the coefficients defined in \eqref{t} we have that
\[
t_{ij}[r]\mapsto 0,\qquad\forall r>n\;.
\]
\noindent (iii) Denote by $V_{k,n}\subset\bbV_w$ the subspace spanned by vectors of the form
\[
f_n(\ve_{w(n)})v_{i_n}\otimes\cdots\otimes f_1(\ve_{w(1)})v_{i_1},\qquad\sum_{j=1}^ni_j=n-k,
\]
where $f_j\in\Kpt$. Then
\[
t_{00}(z),t_{11}(z):V_{k,n}\to V_{k,n}\otimes\C[z],\quad t_{01}(z):V_{k,n}\to V_{k-1,n}\otimes\C[z],
\quad t_{10}(z):V_{k,n}\to V_{k+1,n}\otimes\C[z]\;.
\]
That is, the operators $t_{00}(z),t_{11}(z)$ preserve the number of 0 and 1-letters in the word $I=i_1\ldots i_n$ labeling a vector in $V_{k,n}$, while $t_{01}(z)$ adds a 1-letter and removes a zero letter and $t_{10}(z)$ does the opposite.
\end{prop}

\begin{proof}
All claims are a straightforward computation employing the explicit form of the $R$-matrix \eqref{R} and \eqref{evT}.
\end{proof}

Recall from Lemma \ref{lem:RTTrelations} that the elements $t(z)=t_{00}(z)+q t_{11}(z)$ with $q$ central generate an abelian subalgebra of $\YB$. We now construct a new basis in the evaluation module $\bbV_w$ where these elements act diagonally. The elements of this basis are called `Bethe vectors' and have been previously constructed in \cite[Section 4]{gorbounov2017quantum} and we refer the readers to this work for proofs.

We start by defining so-called `off-shell' Bethe vectors: fix some $0\le k\le n$ and let $x_1,\ldots,x_k$ be pairwise commuting indeterminates. Then we set
\begin{equation}\label{Bethe_def}
\be(x_1,\ldots,x_k)=t_{10}(-x_1)\cdots t_{10}(-x_k)v_o,
\end{equation}
where $v_o$ is the highest weight vector \eqref{pseudom}. (For $k=0$ we shall simply take $v_o$ instead.) The following is a restatement of \cite[Prop 4.3]{gorbounov2017quantum} which gives the expansion of the off-shell Bethe vectors in the spin basis $v_\lambda=v_{i_n}\otimes\cdots \otimes v_{i_1}$ where $I=i_1\ldots i_n$ is the unique 01-word corresponding to $\lambda$:
\begin{prop}
    We have the expansion
    \begin{equation}\label{Bethe2spin}
        \be(x_1,\ldots,x_k)=x_1\cdots x_k\sum_{\lambda}\ve_\lambda^{-1}G_{\lambda^\vee}(1-x_1,\ldots,1-x_k|1-\ve^{-1}_{w(n)},\ldots,1-\ve^{-1}_{w(1)})v_\lambda,
    \end{equation}
    where the sum runs over all $\lambda\subset (n-k)^k$, $\lambda^\vee$ denotes the partition whose Young diagram is the complement of the one of $\lambda$ in the $k\times(n-k)$ bounding box and $\ve_\lambda=\prod_{j\in J_\lambda}\ve_{w(j)}$ with $J_\lambda\subset\{1,\ldots,n\}$ being the positions of $0$-letters in the 01-word corresponding to $\lambda$.
\end{prop}
Using the commutation relations from Lemma \ref{lem:RTTrelations} one now shows by a standard computation in the literature on quantum integrable systems (see e.g. \cite{faddeev1990lectures}) that the Bethe vectors are eigenvectors of the images of the elements $t(z)$ in the evaluation module $\bbV_w$ provided the indeterminates $x_i$ satisfy the so-called Bethe ansatz equations,
\begin{equation}\label{BAE}
\prod_{j=1}^n(1-x_i/\ve_j)\prod_{j\neq i}^k(x_j/x_i)+(-1)^kq=0,\qquad i=1,\ldots,k\;.
\end{equation}
The solutions of these equations are called `Bethe roots' and the Bethe vectors evaluated on the Bethe roots are called `on-shell'. (Physically, the Bethe roots are the momenta of quasi-particles in the associated quantum spin chain.) The statement which in is general difficult to prove is that all eigenvectors are obtained this way and that they from an eigenbasis in each subspace $V_{k,n}=\mathrm{span}\{v_\lambda~:~\lambda\subset(n-k)^k\}\subset\bbV_w$. Expanding the Bethe roots as power series in $q$ both statements have been proven in \cite[Lemma 4.6 and Theorem 4.8]{gorbounov2017quantum}. In particular, each solution $x^\lambda=(x^\lambda_1,\ldots,x_k^\lambda)$ with $\lambda\subset (n-k)^k$ is uniquely identified by its formal limit $q\to0$ where it coincides with $\ve_\lambda$.
\begin{thm}\label{thm:bethe-vectors}
    The `on-shell' Bethe vectors $\{\be_\lambda=\be(x^\lambda_1,\ldots,x_k^\lambda)~:~\lambda\subset (n-k)^k\}$ provide an eigenbasis in each subspace $V_{k,n}$ and we have the eigenvalue equations
    \begin{equation}\label{eigenvalues}
        t(z)\be_\lambda=\left(\frac{\prod_{j=1}^n(1+z/\ve_j)}{\prod_{i=1}^k(1+z/x^\lambda_i)}+\frac{q}{\prod_{i=1}^k(1+x^\lambda_i/z)}\right)\be_\lambda\;.
    \end{equation}
\end{thm} 
N.B. the eigenvalues are polynomial in $z$ because of the Bethe ansatz equations. In fact, one verifies that the condition that the residues at $z=-x_i$ vanish for $i=1,\ldots,k$ is equivalent to the Bethe ansatz equations \eqref{BAE}. The polynomial form of the eigenvalues can be made explicit using  the so-called level-rank duality which we discuss next.

\subsection{Right modules and Level-Rank duality}\label{sec:int-level-rank}
Recall that there is a natural ring isomorphism $\QK_{\T}(\Gr(k;n))\to\QK_{\T}(\Gr(n-k;n))$, often referred to as `level-rank duality' in the literature (because of the close connection between these rings in the simpler case of quantum cohomology and fusion rings in conformal field theory when setting $q=1$). We briefly discuss this isomorphism in terms of the Yang-Baxter algebra by considering a specific linear isomorphism $\bbV_w\to\bbV_{ww_0}$ and comment on its algebraic origin at the end.

Define the following transformation of the $R$-matrix \eqref{R},
\begin{equation}\label{L'}
\tilde R(-z\ve)=(1\otimes \sigmax)R(-z\ve)^{\sfT\otimes 1}(1\otimes \sigmax)
\end{equation}
where $\sfT\otimes 1$ denotes taking the transpose in the first factor of $\C^2\otimes\C^2$. Using the latter we define a corresponding $\tilde T$-matrix via
\begin{equation}\label{evT'}
\tilde T(z)= 
\tilde R_{0n}(-z\ve_{w(n)})\cdots \tilde R_{02}(-z\ve_{w(2)})\tilde R_{01}(-z\ve_{w(1)})\in\End(\C^2\otimes\bbV_w)\otimes\C[z]
\end{equation}
and denote the corresponding matrix elements by $\tilde t_{ij}(z)$ with $i,j=0,1$. We then have the following relation with the (left) action of the Yang-Baxter algebra $\YB$ on $\bbV_w$:
\begin{lemma}\label{lemma:intsys-level-rank} 
Consider the $\C$-linear extension of the involutive map $\Gamma:\bbV_w\to\bbV_{w\cdot w_0}$ defined by
\begin{equation}\label{Gamma}
\Gamma:f_1(\ve_{w(1)})v_{i_1}\otimes\cdots\otimes f_n(\ve_{w(n)})v_{i_n}\mapsto f_1(\ve^{-1}_{w(n)})v_{1-i_n}\otimes\cdots \otimes f_n(\ve^{-1}_{w(1)})v_{1-i_1}\;.
\end{equation}
Then we have the following relation between the entries of the monodromy matrix \eqref{evT} and the ones of \eqref{evT'},
\begin{equation}\label{tGamma}
\tilde t_{ij}(z)\circ\Gamma=\Gamma\circ t_{ji}(z)\;.
\end{equation}
\end{lemma}
\begin{proof}
It suffices to consider the case where $w=w_0$, since all the other tensor products are isomorphic. Consider the tensor product $\C^2\otimes\bbV_{w_0}$, where we label the first factor with `0' and the factors in $\bbV_{w_0}$ consecutively from left to right with $1,2,\ldots,n$. Then by taking the trace over the $0$th factor, we arrive at the identity
\begin{multline*}
t_{ij}(z)=\Tr_0(E_{ji}\otimes1)R_{0n}(-z/\ve_n)\cdots R_{01}(-z/\ve_1)\\
=\Tr_0(E_{ij}\otimes1)R_{01}(-z/\ve_1)^{\sfT\otimes 1}\cdots R_{0n}(-z/\ve_n)^{\sfT\otimes 1}\\
=\bigl(\bigotimes_{i=1}^n\sigmax_i\bigr) \Tr_{0}(E_{ij}\otimes1)\tilde R_{01}(-z/\ve_1)\cdots \tilde R_{0n}(-z/\ve_n)\bigl(\bigotimes_{i=1}^n\sigmax_i\bigr), 
\end{multline*}
where the indices indicate on which factors in $\C^2\otimes\bbV_{w_0}$ the operators act non-trivially. Comparing the last line with \eqref{evT'} the assertion follows.
\end{proof}

\begin{remark}\rm
    The algebraic context of the last lemma becomes clear upon noting that the $\tilde T$-matrix defined in \eqref{evT'} satisfies the following quadratic relation,
    \[
\tilde T_1(z)\tilde T_2(w)R^{\sfT\otimes\sfT}_{12}(z/w)=R^{\sfT\otimes\sfT}_{12}(z/w)\tilde T_2(w)\tilde T_1(z)\;.
    \]
    That is, the map $T^\opp(z)\mapsto\tilde T(z)$ is an algebra anti-morphism $\YB^\opp\to\End\bbV_{ww_0}$, where $\YB^\opp$ is the Yang-Baxter algebra from Remark \ref{rmk:YBopp}. Due to Lemma \ref{lem:YB2YBopp} we can regard each left $\YB$-module as a right $\YB^\opp$-module and the matrix \eqref{evT'} just describes this right action of $\YB^\opp$ on the (left) $\YB$-module $\bbV_w$. On the level of algebras this amounts to the statement that the map $\gamma: T(z)\mapsto \sigmax\cdot (T(z^{-1})^{-1})^{\sfT}\cdot\sigmax$ defines an Hopf algebra anti-isomorphism $\YB\to\YB^\opp $. The latter is the algebraic analogue for the geometrically defined level-rank duality.
\end{remark}

Recall that we have shown that all $\{\bbV_w\}_{w\in W}$ viewed as left $\YB$-modules are isomorphic. Below we will show that this statement implies that the action of $\YB$ commutes with the natural (geometric) left Weyl-group action. It is therefore important to also discuss that all $\{\bbV_w\}_{w\in W}$ viewed as right $\YB^\opp$-modules are isomorphic and, importantly, that the associated braiding for the right $\YB^\opp$-modules is the same as for as for the left $\YB$-modules. 
\begin{lemma}\label{lem:R'check}
Let $\check R(z)=P\circ R(z)$ be the same matrix as defined in Lemma \ref{lem:Rcheck}. Then
\begin{equation}
\check R_{12}(\ve_1/\ve_2)\tilde R_{02}(-z\ve_2)\tilde R_{01}(-z\ve_1)=
\tilde R_{02}(-z\ve_1)\tilde R_{01}(-z\ve_2)\check R_{12}(\ve_1/\ve_2)\;.
\end{equation}
\end{lemma}
\begin{proof}
Employing once more \eqref{YBE}, we obtain 
\begin{eqnarray*}
R_{13}(-z\ve_2)^{\sfT\otimes 1}R_{12}(-z \ve_3)^{\sfT\otimes 1} R_{23}(\ve_2/\ve_3)&=&
R_{23}(\ve_2/\ve_3)R_{12}(-z\ve_3)^{\sfT\otimes 1}R_{13}(-z\ve_2)^{\sfT\otimes 1}\\
\tilde R_{13}(-z\ve_2)\tilde R_{12}(-z\ve_3)R_{32}(\ve_2/\ve_3) &=&
R_{32}(\ve_2/\ve_3)\tilde R_{12}(-z\ve_3)\tilde R_{13}(-z\ve_2)
\end{eqnarray*}
where we have used the identity $(\sigmax \otimes\sigmax)R_{23}(z)(\sigmax\otimes\sigmax)=R_{32}(z)$ from \eqref{invR} in the second line. The assertion then follows upon noting that $R_{32}(z)=\check R_{23}(z)\circ P_{23}$.
\end{proof}

\begin{cor}
The vector spaces $\{\bbV\}_{w\in W}$ interpreted as right $\YB^\opp$-modules via \eqref{evT'} are all isomorphic, with the isomorphisms given by the repeated action with the $\check R$-matrix.
\end{cor}
Having established that the modules $\{\bbV_w\}_{w\in W}$ form isomorphism classes we specialize henceforth $w=id$ and to ease the notation simply write $\bbV_n$ instead. Note that, as vector spaces, we have 
\begin{equation}\label{tensor2V}
\bbV_n\cong (\C^2)^{\otimes n}\otimes\Kpt\;,
\end{equation} 
where the latter isomorphism is canonical,
\begin{equation}
f_n(\ve_n)v_{j_n}\otimes\cdots\otimes f_1(\ve_1)v_{j_1}\mapsto f_n(\ve_n)\cdots f_1(\ve_1) v_{j_n}\otimes\cdots\otimes v_{j_1}\;.
\end{equation}

\begin{remark}\rm
Having identified the vector space $\bbV_n$ simultaneously as a left $\YB$ and as a right $\YB^\opp$-module, we can ask what the commutation relations are between the corresponding monodromy matrices \eqref{evT} and \eqref{evT'}. One finds that the following identity holds true,
\begin{equation}\label{LTT}
\tilde R_{12}(zw)\tilde T_1(z)T_2(w)=T_2(w)\tilde T_1(z)\tilde R_{12}(zw)\;,
\end{equation}
from which the commutation relations of the corresponding matrix elements $t_{ij}(z)$ and $\tilde t_{ij}(z)$ can be derived.
\end{remark}

Using level-rank duality there is an alternative construction of the Bethe vectors \eqref{Bethe_def} using instead the operators $\tilde t_{01}(z)$. Namely, we now fix the lowest weight vector $\tilde v_o=\Gamma(v_o)=v_0\otimes\cdots\otimes v_0\in V_{n,n}$ satisfying
\[
\tilde t_{00}(z)\tilde v_o=\prod_{i=1}^n(1+z\ve_i)\;\tilde v_o,\qquad\tilde t_{11}(z)\tilde v_o=\tilde v_o,\qquad
\tilde t_{10}(z)\tilde v_o=0
\]
and set
\begin{equation}
    \tilde\be(\tilde x_1,\ldots,\tilde x_{n-k})=\tilde t_{01}(-\tilde x_1)\cdots\tilde t_{01}(-\tilde x_{n-k})\tilde v_o\;.
\end{equation}
Using \eqref{tGamma} one easily verifies that $\tilde\be(\tilde x_1,\ldots,\tilde x_{n-k})\in V_{k,n}[\tilde x_1,\ldots,\tilde x_{n-k}]$ is the image of $\be(\tilde x_1,\ldots,\tilde x_{n-k})\in V_{n-k,n}[\tilde x_1,\ldots,\tilde x_{n-k}]$ under level-rank duality. The following is a restatement of \cite[Prop. 4.9 and Cor. 4.10]{gorbounov2017quantum}.
\begin{thm}
Suppose that the $\tilde x_i$ satisfy the dual Bethe ansatz equations,
\begin{equation}\label{dualBAE}
    \prod_{j=1}^n(1-\tilde x_i\ve_j)\prod_{j\neq i}^{n-k}(\tilde x_j/\tilde x_i)+(-1)^{n-k}q=0,\qquad i=1,\ldots,n-k\;.
\end{equation}
Then $\tilde\be_{\lambda'}=\tilde\be(\tilde x^{\lambda'}_1,\ldots,\tilde x^{\lambda'}_{n-k})=\be(x^{\lambda^\vee}_1,\ldots,x^{\lambda^\vee}_k)=\be_{\lambda^\vee}$ and we have the eigenvalue equations
\begin{gather}
    t(z)\be_{\lambda^\vee}=\prod_{j=1}^{n-k}(1+z\,\tilde x^{\lambda'}_{j})\be_{\lambda^\vee}\quad\text{and}\quad
    \tilde t(z)\be_{\lambda^\vee}=\prod_{i=1}^{k}(1+z\, x^{\lambda^\vee}_{i})\be_{\lambda^\vee}\;,
\end{gather}
where $\tilde t(z)=\tilde t_{00}(z)+q\tilde t_{11}(z)$.
\end{thm}
In the latter form the eigenvalue equations for both operators, $t(z)$ and $\tilde t(z)$, are manifestly polynomial in $z$, but they do require one to consider simultaneously solutions of both Bethe ansatz equations, \eqref{BAE} and \eqref{dualBAE}. In the classical limit $q\to 0$ these two sets of solutions can be identified with the Chern roots of the quotient and the tautological bundle, respectively. Both sets of roots are of course algebraically dependent and this algebraic dependence is encoded in the level-rank duality \eqref{Gamma}.


\section{The Weyl group action and the R-matrix; proof of the main theorem}\label{sec:left-weyl} 
Using the $\check R$-matrix from Lemmata \ref{lem:Rcheck} and \ref{lem:R'check}, which provides the braiding of the left $\YB$ and the right $\YB^\opp $-modules, we now define a non-trivial left Weyl group action on the tensor product $\bbV_n\cong\Kpt\otimes (\C^{2})^{\otimes n}$. Denote by $w\otimes 1$ with $w \in {W}$ 
the natural left ${W}$-action 
on the factor $\Kpt$ in $\Kpt\otimes(\C^2)^{\otimes n}$ and for ease of notation set 
$\alpha_i=\ve_i/\ve_{i+1}$ for $i=1,\ldots,n-1$ which we can identify with the 
simple roots of $A_n$. Recall that a cocycle $C:W\to\End\bbV_n$ is a map satisfying
\begin{equation}\label{cocycle}
C_{ww'}=C_w\;w(C_{w'}),\qquad \forall w,w'\in {fW} \/.
\end{equation}
\begin{lemma}
Setting for simple reflections $w=s_i$ with $i=1,\ldots,n-1$
\begin{equation}\label{C2R}
C_{s_i}=\check R_{n-i,n-i+1}(\alpha_i),
\end{equation}
where the indices of the $\check R$-matrix indicating on which factors of the tensor product $(\C^2)^{\otimes n}$ of $\bbV_n$ it acts non-trivially, gives a well-defined cocycle $C:W\to\End\bbV_n$. The other values $C_w\in\End\bbV_n$ are determined by the cocycle condition \eqref{cocycle}.
\end{lemma}
\begin{proof}
It suffices to show that $C_{s_i}$ satisfies the braid relations and that $s_i(C_{s_i})=\check R(\ve_{i+1}/\ve_i)=\check R(\ve_1/\ve_{i+1})^{-1}=1$. The braid relations are a rewriting of the relation \eqref{YBE} for the $\check R$-matrix. Namely, denote by $P_{ij}$ the flip operator in the $i$th and $j$th factor of the tensor product $(\C^{2})^{\otimes 3}$. The latter trivially satisfies the braid relation, $P_{23}P_{12}P_{23}=P_{12}P_{23}P_{12}$. Multiplying in \eqref{YBE} with $P_{23}P_{12}P_{23}$ on the left hand side and with $P_{12}P_{23}P_{12}$ on the right hand side, we find
\[
\check R_{23}(z/w)\check R_{12}(z)\check R_{23}(w)=\check R_{12}(w)\check R_{23}(z)\check R_{12}(z/w)\;.
\]
Setting $z=\ve_3/\ve_2$ and $w=\ve_1/\ve_3$ the braid relation now follows. The second identity is immediate from \eqref{invR}.
\end{proof}
We now define the $W$-action on $\bbV_n$ by $W\ni w\mapsto\bs{w}:= C_w (w\otimes 1)\in\End\bbV_n$ with $C_w$ the cocycle just introduced.
 \begin{prop}\label{prop:leftWaction}
     (i) The left $W$-action on $\bbV_n$ defined by the cocycle $C$ reads explicitly as follows: let $\chi\in \Kpt$ and $J=j_1\ldots j_n$ be a $01$-word, then
   \begin{equation}\label{leftWaction}
\bs{s}_i(\chi\otimes v_J)=\left\{
\begin{array}{ll}
\alpha_i\,\chi^{s_i}\otimes v_J+(1-\alpha_i) \chi^{s_i}\otimes v_{s_i.J},& j_i>j_{i+1}\\
\chi^{s_i}\otimes v_J,&\text{else}
\end{array}\right.,
\end{equation}
where $\alpha_i=\ve_i/\ve_{i+1}$, $\chi^{s_i}(\ve)=\chi(\ldots,\ve_{i+1},\ve_i,\ldots)$ and $s_i.J=j_1\cdots j_{i+1}j_{i}\cdots j_n$ for $i=1,\ldots,n-1$. \medskip

\noindent (ii) Moreover, the subspace $V_{k,n}\subset\bbV_n$ spanned by all $v_J$ with words $J$ that have $k$ zero-letters is left invariant under this action.\medskip

\noindent (iii) The $W$-action commutes with the left $\YB$ as well as the right $\YB^\opp $-action on $\bbV_n$, that is, the images of $t_{ij}(z)$ and $\tilde t_{ij}(z)$ are in $\End_{W}\bbV_n$. 
\end{prop}

\begin{proof}
The first assertion \eqref{leftWaction} is a straightforward computation using the explicit definition \eqref{R}. From this action one sees that that the $W$-action either swaps two letters in the word $J$ or leaves it unchanged, hence, $V_{k,n}$ is invariant. That the $W$-action commutes with the left and right action of the Yang-Baxter algebra $\YB$ is an easy consequence of Lemmata \ref{lem:Rcheck} and \ref{lem:R'check}.
\end{proof}

\subsection{Graphical calculus for the Yang-Baxter algebra}\label{sec:graphical}
We now introduce a convenient graphical calculus to compute matrix elements of the Yang-Baxter algebra generators $t_{ij}(z)$ in the module $\bbV_n$. This graphical calculus will be used to match the action of the $\YB$ elements with the geometrically defined convolution operators.

We first state the combinatorial description of how to compute matrix elements for the generators $t_{ij}(z)$ in the module \eqref{left_tensor} and then explain below how it is obtained from the coproduct structure \eqref{Delta} of the Yang-Baxter algebra. Given two 01-words $J=j_1\ldots j_n, J'=j'_1\ldots j'_n\in\mb{J}^k_n$ we depict the matrix element $\langle v^{J'},t_{ij}(z).v_J\rangle\in\C[z]\otimes \Kpt$, where $\{v^J\}$ denotes the dual basis of $\{v_J\}\subset \bbV_n$ with the canonical pairing $\langle v^I,v_J\rangle=\delta_{IJ}$, as a horizontal line of $n$ labelled vertices as follows, 
\begin{center}
\begin{tikzpicture}

\node[anchor=east] at (0.,0.) {$j$};

\node[anchor=south] at (0.25,0.25) {$j_n$};
\node[anchor=north] at (0.25,-0.25) {$j'_n$};


\node[anchor=south] at (1.75,0.3) {$\cdots$};
\node[anchor=north] at (1.75,-0.3) {$\cdots$};

\node[anchor=south] at (2.75,0.25) {$j_2$};
\node[anchor=north] at (2.75,-0.25) {$j'_2$};

\node[anchor=south] at (3.25,0.25) {$j_1$};
\node[anchor=north] at (3.25,-0.25) {$j'_1$};

\node[anchor=west] at (3.5,0.) {$i$};

\filldraw[black] (0.25, 0) circle (2pt);
\filldraw[black] (0.75, 0) circle (2pt);
\filldraw[black] (1.25, 0) circle (2pt);
\filldraw[black] (1.75, 0) circle (2pt);
\filldraw[black] (2.25, 0) circle (2pt);
\filldraw[black] (2.75, 0) circle (2pt);
\filldraw[black] (3.25, 0) circle (2pt);

\draw[black] (0, 0) -- (3.5, 0);
\draw[black] (0.25, 0.25) -- (0.25, -0.25);
\draw[black] (0.75, 0.25) -- (0.75, -0.25);
\draw[black] (1.25, 0.25) -- (1.25, -0.25);
\draw[black] (1.75, 0.25) -- (1.75, -0.25);
\draw[black] (2.25, 0.25) -- (2.25, -0.25);
\draw[black] (2.75, 0.25) -- (2.75, -0.25);
\draw[black] (3.25, 0.25) -- (3.25, -0.25);
\end{tikzpicture}\;,
\end{center}
where each bullet at a single vertex is a placeholder for one of the following five possible types of vertices,
\begin{equation}\label{5v}
\text{(I)}\quad\connecting{0}{1}\qquad\text{(II)}\quad
\avoiding{0}{1}\qquad
\text{(III)}\quad\avoiding{1}{0}\qquad
\text{(IV)}\quad\avoiding{0}{0}\qquad
\text{(V)}\quad\avoiding{1}{1}
\;.
\end{equation}
In order to obtain the matrix element $\langle v^{J'},t_{ij}(z).v_J\rangle\in\C[z]\otimes \Kpt$ one needs to sum over all possible vertex choices (with the outer left and right as well as top and bottom labels fixed as shown above), where each vertex of type (I) in column $i$ (labelled from left to right) contributes a factor $1+z/\ve_{n+1-i}$, each vertex of type (II) a factor $-z/\ve_{n+1-i}$ and (III-V) each contribute a factor 1. If there is no possible row configuration consisting only of these three types of vertices, then the matrix element is zero.

Let us now explain how this graphical calculus is derived from the coproduct \eqref{Delta} and the definitions \eqref{L}, \eqref{left_tensor}. The vertices in \eqref{5v} correspond to the matrix elements of the generators $t_{ij}(z)$ in the evaluation module $\C^2[\ve^{\pm 1}]$ from Lemma \ref{lem:evhom}. For example, we have according to Lemma \ref{lem:evhom} that $t_{00}(z)v_1=(1+z/\ve)v_1$ which corresponds to vertex type (I) in \eqref{5v} and the factor is the `weight' attached to this particular vertex. Similarly, one finds that the vertex of type (II) encodes the action $t_{10}(z)v_0=(-z/\ve)\,v_1$. Note that we are reading the diagrams from right to left, so the order of the indices of the generators $t_{ij}(z)$ is reversed to the order they appear on the horizontal edges of the vertex. To obtain the matrix element $\langle v^{J'},t_{ij}(z).v_J\rangle$ in the tensor product \eqref{left_tensor} one simply applies repeatedly the following recurrence formula which corresponds to applying the coproduct \eqref{Delta} to the last factor,
\begin{center}
\begin{tikzpicture}
\node[anchor=east] at (0.,0.) {$j$};
\node[anchor=east] at (-0.5,0.) {$\underset{k=0,1}{\mathlarger{\sum}}$};

\node[anchor=south] at (0.25,0.25) {$j_n$};
\node[anchor=north] at (0.25,-0.25) {$j'_n$};


\node[anchor=south] at (1.75,0.3) {$\cdots$};
\node[anchor=north] at (1.75,-0.3) {$\cdots$};

\node[anchor=south] at (3.25,0.25) {$j_2$};
\node[anchor=north] at (3.25,-0.25) {$j'_2$};

\node[anchor=south] at (4.25,0.25) {$j_1$};
\node[anchor=north] at (4.25,-0.25) {$j'_1$};

\node[anchor=west] at (4.5,0.) {$i$};
\node[anchor=east] at (4,0.) {$k$};

\filldraw[black] (0.25, 0) circle (2pt);
\filldraw[black] (0.75, 0) circle (2pt);
\filldraw[black] (2.25, 0) circle (2pt);
\filldraw[black] (2.75, 0) circle (2pt);
\filldraw[black] (3.25, 0) circle (2pt);
\filldraw[black] (4.25, 0) circle (2pt);

\draw[black] (0, 0) -- (3.5, 0);
\draw[black] (0.25, 0.25) -- (0.25, -0.25);
\draw[black] (0.75, 0.25) -- (0.75, -0.25);
\draw[black] (2.25, 0.25) -- (2.25, -0.25);
\draw[black] (2.75, 0.25) -- (2.75, -0.25);
\draw[black] (3.25, 0.25) -- (3.25, -0.25);
\draw[black] (4, 0) -- (4.5, 0);
\draw[black] (4.25, 0.25) -- (4.25, -0.25);
\end{tikzpicture}
\end{center}

The following example explains the graphical calculus:
\begin{example}\rm 
Set $k=2$, $n=7$ and $J=1101011$, $J'=1001101$. That is, we have $\lambda(J)=(1,0)$ and $\lambda(J')=(3,1,1)$. Setting $j=0$ and $i=1$, we compute one term of the matrix element $\langle v^J,t_{10}(z).v_I\rangle$ corresponding to the allowed vertex configuration
\begin{center}
\begin{tikzpicture}
\node[anchor=east] at (0.,0.) {$0$};

\node[anchor=south] at (0.25,0.25) {$1$};
\node[anchor=north] at (0.25,-0.25) {$1$};

\node[anchor=south] at (0.75,0.25) {$1$};
\node[anchor=north] at (0.75,-0.25) {$0$};

\node[anchor=south] at (1.25,0.25) {$0$};
\node[anchor=north] at (1.25,-0.25) {$1$};

\node[anchor=south] at (1.75,0.25) {$1$};
\node[anchor=north] at (1.75,-0.25) {$1$};

\node[anchor=south] at (2.25,0.25) {$0$};
\node[anchor=north] at (2.25,-0.25) {$0$};

\node[anchor=south] at (2.75,0.25) {$1$};
\node[anchor=north] at (2.75,-0.25) {$0$};

\node[anchor=south] at (3.25,0.25) {$1$};
\node[anchor=north] at (3.25,-0.25) {$1$};

\node[anchor=west] at (3.5,0.) {$1$};
\draw[black] (0, 0) -- (0.2, 0);
\draw[black] (0.25, 0.25) -- (0.25, -0.25);
\draw[black] (0.3, 0) .. controls (0.75, 0) .. (0.75, -0.25);

\draw[black] (0.75, 0.25) .. controls (0.75, 0) .. (1, 0);
\draw[black] (1.25, -0.25) .. controls (1.25, 0) .. (1, 0);
\draw[black] (1.25, 0.25) .. controls (1.25, 0) .. (1.7, 0);

\draw[black] (1.75, 0.25) -- (1.75, -0.25);
\draw[black] (2.25, -0.25) .. controls (2.25, 0) .. (1.8, 0);
\draw[black] (2.25, 0.25) .. controls (2.25, 0) .. (2.5, 0);
\draw[black] (2.75, -0.25) .. controls (2.75, 0) .. (2.5, 0);
\draw[black] (2.75, 0.25) .. controls (2.75, 0) .. (3, 0);
\draw[black] (3.15, 0) -- (3, 0);

\draw[black] (3.25, 0.25) .. controls (3.25, 0) .. (3.5, 0);
\draw[black] (3.25, -0.25) .. controls (3.25, 0) .. (3, 0);
\end{tikzpicture}
\end{center}
which produces the term 
\[
(1+z/\ve_{7})(1+z/\ve_{4})(-z/\ve_{2})(-z/\ve_6)
\]
as there is a vertex of type (I) in columns 4 and 7 and a vertex of type (II) in columns 2 and 6 (numbered from right to left). There are many other possible configurations and summing up all their contributions then gives the matrix element.
\end{example}

Another graphical calculus can also be introduced for the right $\YB^\opp$-action defined via \eqref{evT'}, that is, for the computation of the matrix elements $\langle v^{J'},\tilde t_{ij}(z).v_J\rangle$. Now, however, each bullet at a single vertex is a placeholder for one of the following five types of vertices,
\begin{equation}
\text{(I')}\quad\avoiding{0}{0}\qquad
\text{(II')}\quad\avoiding{1}{0}\qquad
\text{(III')}\quad\connecting{0}{1}\qquad
\text{(IV)'}\quad\passing{1}{0}\qquad
\text{(V')}\quad\avoiding{0}{1}
\;.
\end{equation}
The vertex of type (I') now contributes a factor $1+z\ve_{n+1-i}$ if placed in column $i$ and the vertex of type (II') gives a factor $-z\ve_{n+1-i}$ while the ones of type (III'-V') give a factor of 1. As before, if there is no possible row configuration which only consists of these five types of vertices then the matrix element is zero.\medskip

The next theorem employs the graphical calculus to describe action of the operators  $t_{ij}(z),\tilde t_{ij}(z):V_{k,n}\to V_{k\pm 1,n}[z]$ with $i, j=0,1$ on the special vector $v_\lambda$ with $\lambda=(n-k)^k$ which corresponds to the class of a point. Using that the action of $\YB$ commutes with the left Weyl group action this will allow us to identify the $t_{ij}(z),\tilde t_{ij}(z)$ with their geometric analogues. 
\begin{thm}\label{thm:int-system-ops}
Fix two integers $k\le n$ and set $\lambda=(n-k)^k$. 

\noindent(1) The diagonal elements $t_{00}(z)$ and $\tilde t_{00}(z)$ act as follows:
\begin{gather}
t_{00}(z).v_{(n-k)^k}=\prod_{i=1}^{n-k}(1+z/\ve_i)\;v_{(n-k)^k}\\
\tilde t_{00}(z).v_{(n-k)^k}=\prod_{i=n+1-k}^{n}(1+z\ve_i)\;v_{(n-k)^k}\;.
\end{gather}

\noindent(2) The off-diagonal elements act in $V_{k,n}$ according to the formulae
\begin{gather}
t_{01}(z).v_{(n-k)^k}=\prod_{i=1}^{n-k}(1+z/\ve_i)\; v_{(n-k)^{k-1}}\\
t_{10}(z).v_{(n-k)^k}=-\sum_{r=1}^{n-k}(z/\ve_{r})\prod_{i=r+1}^{n-k}(1+z/\ve_i)\;v_{((n-k-1)^k,r-1)}\;.
\end{gather}

and
\begin{gather}
\tilde t_{10}(z).v_{(n-k)^k}=\prod_{i=n+1-k}^{n}(1+z\ve_i)\; v_{(n-k-1)^{k}}\\
\tilde t_{01}(z).v_{(n-k)^k}=-\sum_{r=n+1-k}^{n}(z\ve_{r})\prod_{i=n+1-k}^{r-1}(1+z\ve_i)\;v_{((n+1-k-1)^{r-1},(n-k)^{k-r})}\;.
\end{gather}

(3) We have the following action of the operators $t_{11}(z)$ and $\tilde t_{11}(z)$,
\begin{gather}
t_{11}(z).v_{(n-k)^k}=\sum_{r=1}^{n-k}(-z/\ve_r)\prod_{i=r+1}^{n-k}(1+z/\ve_i)\;v_{((n-k-1)^{k-1},r-1)}\\
\tilde t_{11}(z).v_{(n-k)^k}=\sum_{r=n+1-k}^n(-z\ve_r)\prod_{i=n+1-k}^{r-1}(1+z\ve_i)\;v_{((n-k)^{r-1},(n-k-1)^{k-r},0)}\;.
\end{gather}
\end{thm}
We note that the formulae for $\tilde t_{ij}(z)$ can be obtained from the ones for $t_{ij}(z)$ using \eqref{tGamma} but for convenience we have stated them both explicitly.
\begin{proof}
Recall that $\lambda=(n-k)^k$ corresponds to the $01$-word $J=j_1j_2\ldots j_n=1\ldots 10\ldots 0$ with $k$ zero-letters and $(n-k)$ one-letters. We start with the proof of the first claim. Using once again the graphical calculus we find that in both cases there is only a single allowed row configuration of vertices,
\begin{gather*}
t_{00}(z).v_{1\ldots 10\ldots 0}=\Avoiding{0}{0}{}{0}\hspace{-.55cm}\Avoiding{}{0}{}{0}\cdots
\Avoiding{}{0}{}{0}\hspace{-.55cm}\Connecting{}{1}{}{1}\cdots\Connecting{}{1}{}{1}\hspace{-.55cm}\Connecting{}{1}{0}{1}
\;v_{1\ldots 10\ldots 0}\\
\tilde t_{00}(z).v_{1\ldots 10\ldots 0}=\Avoiding{0}{0}{}{0}\hspace{-.55cm}\Avoiding{}{0}{}{0}\cdots
\Avoiding{}{0}{}{0}\hspace{-.55cm}\Connecting{}{1}{}{1}\cdots\Connecting{}{1}{}{1}\hspace{-.55cm}\Connecting{}{1}{0}{1}
\;v_{1\ldots 10\ldots 0}
\end{gather*}
Note that in the first and second line different weights are assigned to the vertices resulting in different matrix elements although the row configurations look the same. Multiplying up these weights then gives the claimed coefficients in (1).

For the second claim, we observe that the allowed row configurations of vertices for the off-diagonal elements are
\begin{gather*}
t_{01}(z).v_{1\ldots 10\ldots 0}=\Avoiding{1}{0}{}{1}\hspace{-.55cm}\Avoiding{}{0}{}{0}\cdots
\Avoiding{}{0}{}{0}\hspace{-.55cm}\Connecting{}{1}{}{1}\cdots\Connecting{}{1}{}{1}\hspace{-.55cm}\Connecting{}{1}{0}{1}
\;v_{1\ldots 10\ldots 01}\\
t_{10}(z).v_{1\ldots 10\ldots 0}=\sum_{r=1}^{n-k}\Avoiding{0}{0}{}{0}\hspace{-.55cm}\Avoiding{}{0}{}{0}\cdots
\Avoiding{}{0}{}{0}\hspace{-.55cm}\Connecting{}{1}{}{1}\cdots
\Connecting{}{1}{}{1}\hspace{-.55cm}\underset{r}{\Avoiding{}{1}{}{0}}\hspace{-.55cm}\Avoiding{}{1}{}{1}
\cdots\Avoiding{}{1}{}{1}\hspace{-.55cm}\Avoiding{}{1}{1}{1}
\;v_{1\ldots 1\underset{r}{0}1\ldots 10\ldots 0}
\end{gather*}
where the sum in the second line runs over the column numbers (numbered from right to left) in which the additional $0$-letter is placed. Again multiplying up the vertex weights for each row configuration then gives the claimed coefficients. The formulae for $\tilde t_{01}(z)$ and $\tilde t_{10}(z)$ also follow from \eqref{tGamma}, and are left as an exercise to the reader.

Finally, for the third claim we have that
\begin{gather*}
t_{11}(z).v_{1\ldots 10\ldots 0}=\sum_{r=1}^{n-k}\Avoiding{1}{0}{}{1}\hspace{-.55cm}\Avoiding{}{0}{}{0}\cdots
\Avoiding{}{0}{}{0}\hspace{-.55cm}\Connecting{}{1}{}{1}\cdots
\Connecting{}{1}{}{1}\hspace{-.55cm}\underset{r}{\Avoiding{}{1}{}{0}}\hspace{-.55cm}\Avoiding{}{1}{}{1}
\cdots\Avoiding{}{1}{}{1}\hspace{-.55cm}\Avoiding{}{1}{1}{1}
\;v_{1\ldots 1\underset{r}{0}1\ldots 10\ldots 01}\\
\tilde t_{11}(z).v_{1\ldots 10\ldots 0}=\sum_{r=n+1-k}^n\Passing{1}{0}{}{0}\hspace{-.55cm}\Passing{}{0}{}{0}\cdots
\Passing{}{0}{}{0}\hspace{-0.55cm}\underset{r}{\Avoiding{}{0}{}{1}}\hspace{-0.55cm}\Avoiding{}{0}{}{0}\cdots
\Avoiding{}{0}{}{0}\hspace{-.55cm}\Connecting{}{1}{}{1}\cdots
\Connecting{}{1}{}{1}\hspace{-.55cm}\Avoiding{}{1}{1}{0}\;v_{01\ldots10\ldots0\underset{r}{1}0\ldots0}
\end{gather*}
where the sums in the first and second line again run over the column numbers (numbered from right to left). Again multiplying the weights of individual vertices then gives the desired coefficients.
\end{proof}
\begin{cor}\label{cor:main-cor} Parts (1) and (2) of \Cref{thm:QK=YB} hold. That is, under the
$\Kpt[y,q]$-module identification 
\[ \Phi: V_{k,n} \to \QK_{\T}(\Gr(k;n)) \/; \quad v_\lambda \mapsto \cO_\lambda \/, \]
the action of the operators $t_{ij}$ on $V_{k,n}$ is the same as the action of the analogous geometric 
operators $\tau_{ij}$ on $\QK_{\T}(\Gr(k;n))$ defined in section \ref{sec:convo-ops}.

Furthermore, the same result extends for the dual operators $\tilde{t}_{ij}$. 
\end{cor} 
\begin{proof} The identification of the diagonal operators follows from comparing
the statements in \Cref{thm:diag-ops-geom} and parts (1) and (3) of \Cref{thm:int-system-ops}. 
For the off-diagonal operators, one compares \Cref{thm:geom-offdiag} to part (2) of 
\Cref{thm:int-system-ops}. The statement on the dual operators follows from the level-rank duality 
from \cref{ss:level-rank}; details are in \cref{ss:dual-ops}.
\end{proof}



\section{The extended affine Weyl group action}\label{sec:extended}
The goal of this section is to introduce an action of the {\em extended
affine} Weyl group $\widetilde{W}$ on the space $\bbV_n^q :=\bbV_n \otimes \Z[q^{\pm 1}]$,
extending the action of the finite Weyl group from before. 
Depending on the chosen
presentation of $\widetilde{W}$, this amounts to defining actions by a
cyclic generator of the affine Dynkin  automorphism, or by translation elements. 

It is interesting to note that this  
action may be defined independently in the integrable systems and geometry.
In the context of quantum integrable systems it arises naturally 
when considering the {\em quantum Knizhnik-Zamolodchikov equations} 
associated with the $R$-matrix, cf.~\cite{frenkel.reshetikin:quantum}.

In geometry, the action of the cyclic element is called the Seidel action, and it 
has been studied for quantum cohomology and quantum K-theory 
in \cite{postnikov:affine,chaput.perrin:affine,buch.chaput.perrin:seidel,li.liu.song.yang:seidel}.
In future work, we plan to deduce this action from a 
similar action on the equivariant K-homology of
the affine Grassmannian. 

The extended affine Weyl group is the semidirect product 
$\widetilde{W}= W \ltimes \Z^n$, where $W \simeq S_n$ acts on 
$\Z^n$ by permuting coordinates.
We will use two realizations of $\widetilde{W}$,
and the relation between the two
will give non-trivial formulae. We start with the realization with generators 
\[
\widetilde W=\langle s_1,\ldots, s_{n-1},\rho\rangle 
\]
where $\rho$ is the outer (affine, type A) Dynkin diagram automorphism satisfying,
\begin{equation}\label{rho_def}
\rho s_i=s_{i-1}\rho,\quad i=2,\ldots, n-1\;.
\end{equation}
If we set $s_0=\rho s_1\rho^{-1}$ then the subgroup $\Waff=\langle s_0,s_1,\ldots,s_{n-1}\rangle$ is the {\em affine Weyl group} and 
$\Waff \cong W \ltimes Q^\vee$,
where $Q^\vee$ is the (co)root lattice.
The affine Weyl group $\Waff $ is a Coxeter group, unlike $\widetilde W$.

As we shall see below,
the element $\rho$ determines the Seidel action 
on the quantum cohomology and K-theory rings, 
and we refer to its image $\bs{\rho}$ in $\End\bbV^q_n$ as the {\em Seidel element}. 
We extend the $\widetilde{W}$
action on $\bbV_n$ from \eqref{leftWaction} by 
letting $\bs{\rho}$ act on $\bbV_n^q$ as follows. For 
$\chi\in\Kpt$ and {$v_I= v_{i_n} \otimes \ldots \otimes v_{i_1} \in(\C^2)^{\otimes n}$}, define:
\begin{equation}\label{E:rhoI}
\bs{\rho}(\chi\otimes v_I)=q^{i_1}\chi^\rho\otimes v_{i_1}\otimes v_{i_n}\otimes\cdots\otimes v_{i_2},\quad\forall I=i_1i_2\ldots i_n \/.
\end{equation}
Here $\rho$ acts on $\chi \in \Kpt$ by the Coxeter element $s_{n-1}\cdots s_1$, i.e.,
$\chi^\rho = s_{n-1}\cdots s_1(\chi)$. Explicitly, $\chi^\rho(\ve_1,\ldots,\ve_n)=\chi(\ve_n,\ve_1,\ldots,\ve_{n-1})$.
\footnote{
In the context of the quantum Knizhnik-Zamolodchikov equations the action of the affine Weyl group depends on an additional parameter $p$, which is identified with the loop parameter:  $\chi^\rho_p(\ve_1,\ldots,\ve_n)=\chi(p\ve_n,\ve_1,\ldots,\ve_{n-1})$. Here we have set $p=1$, which means that the action of affine Weyl group on $\Kpt[q^{\pm 1}]$ factors through the natural group epimorphism $\widehat W \twoheadrightarrow W$ that sends all translations to the identity.
}
\begin{remark}\rm If $k=0$ then $V^q_{0,n}=\Kpt[q^{\pm 1}]\otimes v_1\otimes\cdots\otimes v_1$, 
and if $k=n$ then $V^q_{n,n}=\Kpt[q^{\pm 1}]\otimes v_0\otimes\cdots\otimes v_0$. 
The action of $\bs{\rho}$ on the unique generators is:
    \[
    \bs{\rho}(v_1\otimes\cdots\otimes v_1)=q v_1\otimes\cdots\otimes v_1 \/; \quad \bs{\rho}(v_0\otimes\cdots\otimes v_0)= v_0\otimes\cdots\otimes v_0\;. \]
 Note the non-trivial action on $V_{0.n}^q$.\end{remark}   
In order to compare the action of $\rho$ with the known action of the Seidel element, we restate the action \eqref{E:rhoI} in terms of partitions $\lambda$:
\begin{equation}\label{E:rho}
        \bs{\rho}(\chi\otimes v_\lambda)=
 \begin{cases} q\,\chi^\rho\otimes v_{(\lambda_1-1,\ldots,\lambda_k-1)} &  \ell(\lambda)=k \/;\\
            \chi^\rho\otimes v_{(n-k,\lambda_1,\ldots,\lambda_{k-1})} & \text{else} \/.
            \end{cases}
    \end{equation}
We also record the inverse action: 
\begin{equation}\label{E:rhoinv}
        \bs{\rho}^{-1}(\chi\otimes v_\lambda)= \begin{cases} 
            q^{-1}\chi^{\rho^{-1}}\otimes v_{(\lambda_1+1,\ldots,\lambda_k+1)} &  \lambda_1<n-k \/; \\
            \chi^{\rho^{-1}}\otimes v_{(\lambda_2,\ldots,\lambda_{k},0)} & \text{else} \/.
        \end{cases}
    \end{equation}
These formulae are the same as those in \cite{postnikov:affine,chaput.perrin:affine,buch.chaput.perrin:seidel,li.liu.song.yang:seidel},
where the Seidel representation on quantum cohomology 
and quantum K-theory of Grassmannians is studied.
One may also calculate the action of the affine simple reflection $s_0 \in \Waff$. If $\lambda \in \Pi^k_n$ is a partition in the $k \times (n-k) $ rectangle, then:
\begin{equation}\label{E:s0act}
        \bs{s}_0(\chi\otimes v_\lambda)=\left\{\begin{array}{ll}
            q^{-1}(1-\alpha_0)\chi^{s_\theta}v_{s_\theta.\lambda}+\alpha_0\chi^{s_\theta}\otimes v_\lambda, &  s_\theta.\lambda\subset\Pi^k_{n} \/;\\
            \chi^{s_\theta}\otimes v_\lambda. & \text{else} \/.
        \end{array}\right.
    \end{equation}
Here $\theta=\ve_1/\ve_n=\alpha_0^{-1}$ is the highest root,  $s_\theta.\lambda$ is the partition $\lambda$ with an $(n-1)$-hook added and $\chi^{s_\theta}(\ve_1,\ldots,\ve_n)=\chi(\ve_n,\ve_2,\ldots,\ve_{n-1},\ve_1)$.
   
 One may check directly that the equations \eqref{E:rho} and \eqref{E:rhoinv} together with the left $W$-action
 from \eqref{leftWaction} define an action of the extended affine Weyl group $\widetilde{W}$
on $\bbV_n^q$. 
\begin{example}\rm 
    We compute the action of the cyclic element $\rho$ and the affine reflection $s_0$ for a couple of examples.     \[
    \bs{s}_0v_\emptyset=\bs{\rho}\bs{s}_1(q^{-1}v_{(1^k)})\\
    =q^{-1}\bs{\rho}(\alpha_1 v_{(1^k)}+(1-\alpha_1)v_{(1^{k-1})})\\
    =\alpha_0\,v_\emptyset+q^{-1}(1-\alpha_0)v_{(n-k,1^{k-1})} \/;
    \]
    \[
    \bs{s}_0v_{(n-k)^k}=\bs{\rho}\bs{s}_1(v_{(n-k)})=\bs{\rho}v_{(n-k)}=v_{
    (n-k)^k}\/.
    \]
    In particular, the identity $v_\emptyset$ is invariant under the action of the finite Weyl group
    $W$, and
the class of a point $v_{(n-k)^k}$ is invariant under the conjugate of $W$ given by $\langle s_{k-1},\ldots s_1, s_0, s_{n-1},\ldots,s_{k+1}\rangle\subset \Waff$.
\end{example}
We also record the following:
\begin{cor}
    The $n$th power of the Seidel element satisfies $\bs{\rho}^n=q^{n-k}$ on the subspace $V_{k,n}^q \subset \bbV_n^q$. In particular, $\bs{\rho}^n$ is a central element in $\End\bbV_n^q$.
\end{cor}
We now relate the action of the Yang-Baxter algebra $\YB$ with the action of the translations in the extended affine Weyl group. For this we use a second realization of the extended Weyl group $\widetilde{W}$:
\[ \widetilde W =\langle s_1,\ldots,s_{n-1},t_1,\ldots,t_n\rangle \] 
with relations given by the usual braid and commutation relations for the $s_i$'s and 
\[
s_it_i=t_{i+1}s_i,\qquad s_it_j=t_js_i,\quad j\neq i,i+1,\qquad t_it_j=t_jt_i\;.
\]
The connection with the previous realization is given by:
\begin{equation}\label{rho2t}
t_i=s_{i}\cdots s_{n-1}\rho s_1\cdots s_{i-1}\quad\text{ and }\quad
\rho=s_{n-1}\cdots s_1 t_1\;.
\end{equation}
The elements $t_i$ are translations in the weight lattice $\Z^n$. 
\begin{prop}\label{prop:trans} 
Let $t(z) = t_{00}(z)+ q t_{11}(z)$, be the (quantised) trace of the monodromy operator $T$ from \eqref{T}. Then the action of the $t_i$ from \eqref{rho2t} on $\bbV_n^q$ is given by 
\[
t(-\ve_i)=\bs{t}_i
\qquad
\text{and}
\qquad
\tilde t(-1/\ve_i)=q\;\bs{t}^{-1}_i,\quad i=1,\ldots n.
\]
In other words, via the equivalence from \Cref{cor:main-cor}, the action of the translations is given by:
\[ t_i.\cO_\lambda = \lambda_{-\ve_i}(\cQ^\vee)\star \cO_\lambda 
\qquad\text{and}\qquad t_i^{-1}.\cO_\lambda=q^{-1}\lambda_{-1/\ve_i}(\cS)\star\cO_\lambda
\/.  \]
\end{prop}
\begin{proof}
Recall the definition of the operator $t(z)$ in terms of the $R$-matrix,
\[
t(z)=t_{00}(z)+q t_{11}(z)=\Tr\begin{pmatrix} t_{00}(z) &t_{01}(z)\\
q t_{10}(z)& qt_{11}(z)
\end{pmatrix}=\Tr_0\left( \begin{smallmatrix}
    1&0\\
    0&q
\end{smallmatrix}\right)_0R_{0n}(-z/\ve_1)\cdots R_{01}(-z/\ve_n)\;,
\]
where we have used the definition of the monodromy matrix from \eqref{evT} in the last equality and the notation $\Tr_0$ just means that we take the trace over the first factor in the tensor product $\C^2[z]\otimes \bbV^q_n$. 
Using the cyclic property of the trace this can be rewritten as
\[
t(z)=\Tr_0 R_{0,i-1}(-z/\ve_{n+2-i})\cdots R_{01}(-z/\ve_n)\left(\begin{smallmatrix}
    1&0\\
    0&q
\end{smallmatrix}\right)_0R_{0n}(-z/\ve_1)\cdots R_{0i}(-z/\ve_{n+1-i})\;
\]
for any $1\le i\le n$. Setting $z=-\ve_{n+1-i}$ and using that $R(1)=P$ (the permutation, or flip operator), we find that:
\[ \begin{split} t(-\ve_{n+1-i})& =
\Tr_0R_{0,i-1}(\ve_{n+1-i}/\ve_{n+2-i})\cdots R_{01}(\ve_{n+1-i}/\ve_n)\left(\begin{smallmatrix}
    1&0\\
    0&q
\end{smallmatrix}\right)_0R_{0n}(\ve_{n+1-i}/\ve_1)\cdots R_{0i}(1) \\ & 
= R_{i,i-1}(\ve_{n+1-i}/\ve_{n+2-i})\cdots R_{i1}(\ve_{n+1-i}/\ve_n)\left(\begin{smallmatrix}
    1&0 \\
    0&q
\end{smallmatrix}\right)_0R_{in}(\ve_{n+1-i}/\ve_1)\cdots R_{i,i+1}(\ve_{n+1-i}/\ve_{n-i})\\ & =
\bs{s}_{n+1-i}\cdots \bs{s}_{n-1}\bs{\rho} \bs{s}_1\cdots \bs{s}_{n-i} \\ & =\bs{t}_{n+1-i}\/.\end{split} \]
Here we have used in the last step that 
$$\bs{s}_{n+1-i}=P_{i,i+1}\circ R_{i,i+1}(\ve_{n-i}/\ve_{n+1-i})(s_{n+1-i}\otimes 1) \/.$$ 
The proof for $\tilde t(z)$ is analogous.
\end{proof}

If one combines the equations \eqref{rho2t} and \eqref{E:rho} with \Cref{prop:trans} one obtains
the following formula for the Seidel representation, which is an equivariant generalization 
of the formula in the quantum K-ring - see \cite{chaput.perrin:affine} for quantum cohomology. 

\begin{cor}\label{cor:QK-seidel} Let $\chi \in \Kpt$. The Seidel element $\bs{\rho}$ acts on $\bbV_n^q$ by
\[ \bs{\rho} (\chi \otimes v_\lambda) = \chi^{\rho} \bs{t}_n \bs{s}_{n-1} \cdots \bs{s}_1 (v_\lambda) \/. \]
In particular, the following equality holds:
\[ \bs{t}_n \bs{s}_{n-1} \cdots \bs{s}_1 (v_\lambda) = \begin{cases} q\, v_{(\lambda_1-1,\ldots,\lambda_k-1)} &  \ell(\lambda)=k \/;\\
           v_{(n-k,\lambda_1,\ldots,\lambda_{k-1})} & \text{else} \/.
\end{cases} 
\]
\end{cor}
Using that $\lambda_{-\ve_n}(\cQ_{n-k}^\vee) = \cO_{n-k}$ from \Cref{thm:lyQveeGr},
one may rewrite this corollary in geometric terms:
\[ \bs{\rho} (\cO_\lambda)  = \cO_{n-k} \star (\bs{s}_{n-1} \cdots \bs{s}_1)(\cO_\lambda) 
=  \begin{cases} q\, \cO_{(\lambda_1-1,\ldots,\lambda_k-1)} &  \ell(\lambda)=k \/;\\
           \cO_{(n-k,\lambda_1,\ldots,\lambda_{k-1})} & \text{else} \/.
\end{cases} 
 \]
Observe that in the non-equivariant case the simple reflections $s_i$ act as identity for 
$1 \le i \le n-1$, and $\ve_i =1$ for all $i$. Therefore the actions of the translations $\bs{t}_i$ coincide, 
and they act by  
\[ \bs{\rho} = \bs{t}_i = t(-1) = \lambda_{-1}(\cQ^\vee) \star = \cO_{n-r} \star \/. \]
This is the form of the Seidel representation from \cite{buch.chaput.perrin:seidel,li.liu.song.yang:seidel}.


\section{The functional relations}\label{sec:functional}
In the integrable system context, the quantum K-theory is the ring with relations given by the Bethe Ansatz equations.
In turn, these are equivalent
to a functional equation satisfied by the monodropmy operator $t(z)$ and its dual $\tilde{t}(1/z)$. The precise statement may be found 
in \cite[Prop. 5.28]{gorbounov2017quantum}. After making the change of variables 
\begin{equation}\label{E:cov} -y=1+\beta z \/, \quad \ve_i = 1+\beta y_i \/, \quad \beta = -1 \end{equation} 
in {\em loc.~cit.} this functional equation becomes: 
\begin{equation}\label{E:functional} 
t(y) \tilde{t}(1/y) = \prod_{i=1}^k (1+\ve_i/y) \prod_{i=k+1}^n (1+ y/\ve_i) (1-\cO_1) +q \/. 
\end{equation}
(These are equivalent to Bethe ansatz equations when written in terms of eigenvalues and specializing $y$ to a Bethe root.)
Given our interpretation from \Cref{cor:main-cor}, \eqref{E:functional} translates into the following relation in $\QK_{\T}(\Gr(k;n))$:
\begin{equation}\label{E:functional-geom} 
\lambda_y(\cQ^\vee) \star \lambda_{1/y}(\cS) = \prod_{i=1}^k (1+\ve_i/y) \prod_{i=k+1}^n (1+ y/\ve_i) (1-\cO_1) +q \/. 
\end{equation}
We will show that if $0 < k <n$, this identity is equivalent to the quantum K Whitney relations proved in \cite[Thm. 1.1]{GMSZ:QK}:
\begin{equation}\label{E:QKW}
\lambda_y(\cS) \star \lambda_{y}(\cQ) = \lambda_y(\C^n) - \frac{q}{1-q} y^{n-k} (\lambda_y(\cS) -1) \star \det \cQ \/. 
\end{equation}
It was also proved that these relations generate the full ideal of relations in the quantum K-ring. 
In particular, this proves directly 
the isomorphism between the two rings. We will need the following result, cf.~\cite[Cor.~6.4]{GMSZ:QK}:
\begin{lemma}\label{conj2:qSvee} Let $i>0$. Then 
 $\wedge^{n-k-i} \mathcal{Q} \star \det \mathcal{S} = \wedge^{i} \mathcal{Q}^\vee  \cdot \det(\mathbb{C}^n)$ in $\mathrm{QK}_{\T}(\Gr(k;n))$.
\end{lemma}
\begin{prop}\label{prop:equiv} The functional equations \eqref{E:functional-geom} and the quantum K Whitney relations \eqref{E:QKW} are equivalent in the ring $\QK_{\T}(\Gr(k;n))$.
\end{prop}
\begin{proof} Observe that the Whitney relations imply that $\det \mathcal{Q} \star \det \mathcal{S} = (1-q) \det \mathbb{C}^n$. In particular, $\det \cS$ is invertible in the quantum K-ring. We multiply the Whitney relations \eqref{E:QKW} by $\det \cS$ to obtain:
\[ \lambda_y \mathcal{S} \star \lambda_y \mathcal{Q} \star \det \mathcal{S} = \lambda_y(\mathbb{C}^n) \det \mathcal{S} - q y^{n-k} (\lambda_y \mathcal{S} -1) \det \mathbb{C}^n \/. \]
We rearrange this into:
\begin{equation}\label{E:QKW-rearr} \lambda_y \cS \star \left( \lambda_y \cQ \star \det \cS + q y^{n-k} \det \mathbb{C}^n \right)  = \lambda_y(\C^n) \det \cS + q y^{n-k} \det \mathbb{C}^n \/.\end{equation}
We use \Cref{conj2:qSvee} to write
\[ \begin{split} \lambda_y \cQ \star \det \cS + q y^{n-k} \det \mathbb{C}^n & = 
\det \C^n  \sum_{i=0}^{n-k-1} y^i \wedge^{n-k-i} (\cQ^\vee)  + y^{n-k} \det \cQ \star \det \cS +  
q y^{n-k} \det \mathbb{C}^n \\
& = y^{n-k} \det \C^n \sum_{i=0}^{n-k-1} y^{i-n-k} \wedge^{n-k-i} (\cQ^\vee) + y^{n-k} \det \C^n \\
& = y^{n-k} \det \C^n \lambda_{1/y} (\cQ^\vee)  \/. \end{split}
\]
Combining with \eqref{E:QKW-rearr} and making $y \mapsto y^{-1}$ yields:
\[ \lambda_{1/y} \cS \star \lambda_{y} \cQ^\vee = y^{n-k} \frac{\lambda_{1/y}(\C^n)}{\det \C^n} \det \cS +q \/. \]
From \Cref{thm:lyQveeGr} we obtain that 
\[ \frac{\det \cS}{\det \C^n} = \det \cQ^\vee = \frac{1-\cO_1}{ \prod_{i=k+1}^n \ve_{i} } \/. \] 
Finally, since
\[ \frac{y^{n-k} \lambda_{1/y} (\C^n)}{\prod_{i=k+1}^n \ve_i} = \prod_{i=1}^k (1+\ve_i/y) \prod_{i=k+1}^n (1+y/\ve_i) \/, \]
it follows that the Whitney relations imply \eqref{E:functional-geom}. The process can obviously be 
reversed, proving the claim.
\end{proof}
Since the Whitney relations generate the ideal of relations of $\QK_{\T}(\Gr(k;n))$, we obtain the following corollary:
\begin{cor}\label{cor:QK-pres} There is an isomorphism
\[ \Psi: \QK_{\T}(\Gr(k;n)) \to \K_{\T}[\pt][X_1, \ldots, X_k;Y_1, \ldots, Y_k][\![q]\!]/ I \]
such that 
\[\Psi(\wedge^i \cS) = e_i(X) \textrm{ and } \Psi(\wedge^j \cQ)  = e_j(Y) \/. \]
Furthermore, for any partition $\lambda \subset k \times (n-k)$, 
\[ \Psi(\cO_\lambda) = G_\lambda(1-X_1, \ldots, 1-X_k|1-\ve_1^{-1}, \ldots, 1- \ve_n^{-1}) \/. \]
Here $I$ is the ideal obtained by equating the powers of $y^i$ in the Whitney relations
\eqref{E:QKW}. 
\end{cor}
\begin{proof} The first assertion is proved in \cite[Thm. 1.1]{GMSZ:QK}. The identification
of the Schubert classes follows from \cite[Cor. 4.12, eq. (4.40)]{gorbounov2017quantum},
together with the equivalence of the two presentations, proved in \Cref{prop:equiv}. \end{proof}



\section{Frobenius structures and the quantum localization map}
From the geometric definition, the quantum K-theory ring has a 
structure of a Frobenius algebra, denoted by $( \cdot , \cdot )_{\QK}$.
It was proved in \cite{BCLM:euler} that for any two opposite Schubert classes
$\mathcal{O}_\lambda$ and $\mathcal{O}^\mu := \bs{w}_0.\mathcal{O}_{\lambda^\vee}$, this pairing is equal to
\begin{equation}\label{E:geomfrob} (\mathcal{O}_\lambda, \mathcal{O}^\mu)_{\QK} = \frac{q^{d(\lambda,\mu)}}{1-q} \/,\end{equation}
where $d(\lambda,\mu)$ is the smallest power of $q$ which appears in the 
(equivariant) quantum {\em cohomology} product $\sigma_\lambda \star \sigma^\mu$ 
of the corresponding cohomological Schubert classes. For example, since $1$ is the identity element, 
\[ (\mathcal{O}_\lambda, 1)_{\QK} = \frac{1}{1-q} \/. \]

In the integrable systems context, the first two authors
used the eigenvectors of the quantum trace of the monodromy matrix to define a product 
on the Yang-Baxter module $V_{k,n}$ so that these eigenvectors become idempotents; see \cite[eq. (4.31)]{gorbounov2017quantum}.
The eigenvectors are only defined up to a multiple,
and in this paper we (re)normalize the idempotents so that they are equal to $\mathbf{b}_\lambda^q$, 
where $\Phi(\mathbf{b}_\lambda^q)=\bfe_\lambda^q$ and $\Phi$ is defined
in \Cref{cor:main-cor}. 
With this renormalization, the integrable system pairing is the unique pairing 
$\langle \cdot , \cdot \rangle$ which satisfies
\begin{equation}\label{E:intsysfrob} \langle \bfe_\lambda^q, \bfe_\mu^q\rangle = \delta_{\lambda,\mu} \mathrm{Eu}_q(\lambda) \end{equation}
where the `quantum Euler class' $\mathrm{Eu}_q(\lambda)$ is defined by 
$\bfe_\lambda^q \star \bfe_\mu^q = \delta_{\lambda,\mu} \mathrm{Eu}_q(\lambda) \bfe_\lambda^q$. 
\begin{footnote}{The elements $\bfe^q_\lambda$ quantize the fixed points, thus $\mathrm{Eu}_q(\lambda)$ quantizes the Euler 
class, justifying the terminology.}\end{footnote}
This definition implies that $\langle \cdot , \cdot \rangle$ is a Frobenius pairing, and it is $W$-equivariant, i.e., 
for any $w \in W$, and any $a,b \in \K_{\T}(\Gr(k;n))$, 
\[ w.\langle a,b \rangle = \langle \bs{w}.a, \bs{w}.b\rangle \/. \]
Furthermore, it is proved in \cite[eq. (4.37)]{gorbounov2017quantum}
that 
\begin{equation}\label{E:intsysfrob2} \langle \cO_\lambda, \bfe_\mu^q \rangle = G_\lambda(1-x^\mu | 1-\ve_1^{-1}, \ldots, 1-\ve_n^{-1} ) \end{equation}
where (recall) $G_\lambda$ is the double Grothendieck polynomial, and 
$x^\mu=(x^{\mu}_1, \ldots, x^{\mu}_k)$ is the solution of \eqref{BAE} which specializes to $\ve^\mu$ when $q=0$. (Here we adapted the notation from {\em loc.~cit.}
to be consistent with the current paper.)
Our goal is to prove that the two Frobenius structures are the same. To this aim,
we introduce the {\em quantum localization
map}. This is defined for each of the pairings, 
using the Frobenius structures, and the Bethe vectors $\bfe_\lambda^q$: 
\[ \iota: \QK_{\T}(\Gr(k;n)) \to \bigoplus_\lambda \Kpt[\![q]\!] \/; \quad \kappa \mapsto 
(\kappa, \bfe_\lambda^q)_{\QK} \/. \]
A similar definition can be given for $\langle \cdot , \cdot \rangle$, leading to a map $\iota'$. 
\begin{prop}\label{prop:qloc-inj} The quantum localization maps $\iota,\iota'$ are injective homomorphisms of $\Kpt[\![q]\!]$-algebras.\end{prop}
\begin{proof} We start with the map $\iota$. 
The injectivity follows because the classical equivariant 
localization map (i.e., when $q=0$) is injective. The main statement to 
prove is the ring homomorphism property. By definition, the vectors $\bfe_\lambda^q$ are 
eigenvectors of the (dual) operator $\tilde{t}_{00}+q \tilde{t}_{11}$, which, in geometry, is the operator
of multiplication by $\lambda_y(\cS)$; cf.~\Cref{cor:main-cor}.
The coefficient of $y^i$ in this operator is (the class of) $\wedge^i \cS$, 
and it follows that $\bfe_\lambda^q$ are also eigenvectors of $\wedge^i \cS$, for $1 \le i \le k$.
In \cite{GMSZ:QK} it was proved that
these classes generate $\QK_{\T}(\Gr(k;n))$ as an algebra over $\Kpt[\![q]\!]$. Thus
it suffices to check that 
\[ (\kappa_1 \star \kappa_2, \bfe_\lambda^q)_{\QK} = (\kappa_1, \bfe_\lambda^q)_{\QK} \cdot (\kappa_2, \bfe_\lambda^q)_{\QK} \]
for any partition $\lambda$ and any $\kappa_1, \kappa_2$ having $\bfe^q_\lambda$ 
as eigenvector. Let $c_1, c_2$ be the corresponding eigenvalues of $\kappa_1, \kappa_2$. 
Then, using the Frobenius property, 
\begin{equation}\label{E:frob-maineq} (\kappa_1 \star \kappa_2, \bfe_\lambda^q)_{\QK} = (\kappa_1, \kappa_2 \star \bfe_\lambda^q)_{\QK} = c_2 (\kappa_1, \bfe_\lambda^q)_{\QK} =c_1 c_2 (1,\bfe_\lambda^q)_{\QK} = (\kappa_1,\bfe_\lambda^q)_{\QK}\cdot (\kappa_2, \bfe_\lambda^q)_{\QK} \/.
\end{equation}
Now observe that the proof above only used that $\iota$ deforms the usual localization map, and that
it satisfies the Frobenius property. Then the same proof applies to $\iota'$.\end{proof}

For $0 \le k \le n$ and indeterminates $x_1, \ldots, x_k$, recall the off-shell Bethe vectors:
\begin{equation}\label{E:Bk} \Phi(\be_k(x))= \tau_{10}(-x_k)\tau_{10}(-x_{k-1}) \ldots \tau_{10}(-x_1)\Phi(v_o) \in \K_{\T}(\Gr(k;n))[x_1, \ldots, x_k] \end{equation}
For $\lambda \subset k \times (n-k)$, define the element
$\ve_\lambda = \prod_{i=1}^k \ve_{\lambda_{k-i+1}+i}$, where
the product is over the positions of $0$'s in the $01$ word 
$J_\lambda$ - see section \ref{sec:conventions}. 
From \cite[Prop.~4.3, eq.~(4.9)]{gorbounov2017quantum} it follows that 
\begin{equation}\label{E:Bk-GK} \Phi(\be_k(x)) = \prod_{i=1}^k x_i \sum_{\lambda \subset k \times (n-k)} \frac{1}{\ve_\lambda}G_{\lambda^\vee}(1-x_1, \ldots, 1-x_k| 1-\ve_n^{-1}, \ldots, 1-\ve_1^{-1})\cO_\lambda \/.\end{equation}
From \Cref{thm:bethe-vectors} it follows that
for a partition $\mu \subset k \times (n-k)$, the Bethe vector $\bfe_\mu^q$ satisfies
$\Phi(\be_k(x^\mu))=\bfe_\mu^q$.

\begin{thm}\label{thm:frob-pairings} The `geometric' and the `integrable systems' Frobenius pairings coincide, i.e., 
\[ (a, b)_{\QK} =\langle a, b \rangle \]
for any $a, b \in \QK_{\T}(\Gr(k;n))$. 
\end{thm}
\begin{proof} It is proved in \cite{GMSZ:QK} 
that there are two isomorphic presentations 
of the ring $\QK_{\T}(\Gr(k;n))$, both with generators 
$e_i(X_1, \ldots, X_k)$ for $1 \le i \le k$, and subject to relations. 
In the `Coulomb branch' presentation,
the relations are obtained by a symmetrization of the 
Bethe Ansatz equations \eqref{BAE}. In particular, 
the indeterminates $X_1, \ldots, X_k$ do satisfy
these equations. (The second, `Whitney' presentation, has relations equivalent to 
those given by the functional equation \eqref{E:functional-geom}.)
In both cases, under the isomorphism to $\QK_{\T}(\Gr(k;n))$,
\[ \wedge^i \cS = e_i(X_1, \ldots, X_k)  \/; \]
in other words, the generator $e_i(X_1, \ldots, X_k)$ represents $\wedge^i \cS$. 
Consider the (unique!) expansion of $\wedge^i \cS$ into Schubert classes:
\[ \wedge^i \cS = \sum a_{i,\lambda} \cO_\lambda \/,\]
with $a_{i,\lambda} \in \K_{\T}(\pt)$. 
Note that this expansion does not involve $q$'s. 
Inside the `Coulomb branch' presentation, this gives an expansion
\[ e_i(X_1, \ldots, X_k) = \sum a_{i,\lambda} G_\lambda(1-X_1, \ldots, 1-X_k|1-\ve_1^{-1}, \ldots, 1-\ve_n^{-1}) \quad \mod I \/, \]
where $I$ is the ideal of relations in the given presentation. 
{In fact, this holds as an algebraic identity without modding out by $I$, see \cite[eq. (2.26)]{gorbounov2017quantum}, specialized to the variables used in this paper.}
Since
any relation in $I$ localizes to $0$, 
using the definition of the pairing 
$\langle \cdot, \cdot \rangle$, this means that for
any partition $\mu$, 
\[ \langle \wedge^i \cS, \bfe_{\mu}^q \rangle = e_i(x^{\mu}_1, \ldots,x^{\mu}_k) = e_i(x^\mu) \/. \]
In other words, the localization is obtained by specializing 
$(X_1, \ldots, X_k)\mapsto x^\mu$, 
the solution of the 
Bethe Ansatz equations \eqref{BAE} corresponding to the partition $\mu$. 
We also used that $\cO_\lambda$ corresponds to
$G_\lambda(1-X_1, \ldots, 1-X_k|1-\ve_1^{-1}, \ldots, 1-\ve_n^{-1})$; cf.
\Cref{cor:QK-pres}.

Take $\mu$ to be an arbitrary partition in the $k \times (n-k)$ rectangle.
Specialize $x \mapsto x^\mu$ in \eqref{E:Bk-GK} to obtain:
\[ \bfe_\mu^q= \langle \det \cS, \bfe_{\mu}^q \rangle \sum_{\lambda \subset k \times (n-k)} \frac{1}{\ve_{\lambda}} \langle \cO^{\lambda}, 
\bfe_{\mu}^q \rangle \cO_\lambda \/, \]
where $\ve_{\lambda}$ denotes the denominator from \eqref{E:Bk-GK} and
$\cO^\lambda$ is the opposite Schubert class.
Using the ring homomorphism property from \Cref{prop:qloc-inj} it follows that the right hand side of this equality can be rewritten as 
\[  
\sum_{\lambda \subset k \times (n-k)} \bigl\langle \frac{1}{\ve_{\lambda}} \cO^{\lambda} \star \det \cS, 
\bfe_{\mu}^q \bigr\rangle \cO_\lambda \]
Recall now from \cite{summers:qideal} (see also \cite{buch.m:qk})
that $\frac{1}{\ve_{\lambda}} \cO^{\lambda} \star \det \cS = \mathcal{I}^{\lambda,q}$ is the quantum ideal sheaf,
i.e. the unique element which satisfies 
$(\mathcal{I}^{\lambda,q}, \cO_\mu )_{\QK} = \delta_{\lambda, \mu}$ for any
partition $\lambda$.
Pairing both sides with the quantum ideal sheaf $\mathcal{I}^{\nu,q}$ under the geometric pairing $(\cdot, \cdot)_{\QK}$, and combining everything, yields
\begin{equation*} 
\begin{split} (\mathcal{I}^{\nu,q}, \bfe_\mu^q)_{\QK} & =  \sum_{\lambda \subset k \times (n-k)} \langle \mathcal{I}^{\lambda,q}, \bfe_{\mu}^q \rangle (\mathcal{I}^{\nu,q}, \cO_\lambda)_{\QK} \\
& = \sum_{\lambda \subset k \times (n-k)} \langle \mathcal{I}^{\lambda,q}, \bfe_{\mu}^q \rangle \delta_{\nu, \lambda} \\
& = \langle \mathcal{I}^{\nu,q}, \bfe_{\mu}^q \rangle \/. 
\end{split}
\end{equation*}
Since $\mu,\nu$ were chosen arbitrarily, the claim follows.
\end{proof}

Since the quantum localization map
\[ \mathrm{QK}_T(\mathrm{Gr}(k;n)) \to \bigoplus_{\lambda \subset (n-k)^k} \K_T(pt) \/; \quad \kappa \mapsto (\kappa, \bfe_\lambda^q )_{\QK} \]
is a ring homomorphism, we have a `quantum Atiyah-Bott' theorem:
\begin{cor}\label{cor:qAB} For any class $\kappa \in \mathrm{QK}_T(\mathrm{Gr}(k;n))$, 
the quantum character (or the {\em Frobenius trace}) defined by 
\[ \mathrm{qch}(\kappa) := (\kappa, 1 )_{\QK} \]
satisfies
\[ \mathrm{qch}(\kappa) = \sum_\lambda \frac{(\kappa, \bfe_\lambda^q)_{\QK}}{(\bfe_\lambda^q , \bfe_\lambda^q)_{\QK}} =  \sum_\lambda \frac{(\kappa, \bfe_\lambda^q)_{\QK}}{\mathrm{Eu}_q(\lambda)}\/. \]
Furthermore, if $\mathrm{ch}(\kappa)$ denotes the classical character (or Euler characteristic)
of $\kappa$, then the quantum and classical characters are related by 
\[ \mathrm{qch}(\kappa) = \frac{\mathrm{ch}(\kappa)}{1-q} \/. \]
\end{cor}
\begin{proof} The first part follows from the expansion of $\kappa$ 
in terms of the Bethe vectors: if 
$\kappa = \sum a_\lambda \bfe_\lambda^q$, 
then from pairing with 
$\bfe_\lambda^q$ we obtain $a_\lambda = (\kappa, \bfe_\lambda^q)_{\QK}/ \mathrm{Eu}_q(\lambda)$. The second part follows because $( 1, \cO_\lambda )_{\QK}=1/(1-q)$. 
\end{proof}
One may interpret the second part of the above corollary in the following way.
Consider the expansions of $\kappa$ into Schubert classes and into Bethe vectors:
\[ \kappa = \sum a_\lambda \bfe_\lambda^q \/; \quad \kappa = \sum b_\lambda \cO_\lambda \/. \]
Then
\[ \mathrm{qch}(\kappa) = \sum a_\lambda \quad \textrm{ and } \quad \mathrm{ch}(\kappa) = \sum b_\lambda \/, \]
thus 
\[ (1-q) \sum a_\lambda = \sum b_\lambda \/. \]
An interesting particular case is when $\kappa=1$. The inverse quantum Euler classes are the coefficients in the expansion into the Bethe vectors, and:
\begin{equation}\label{E:AB-1} \frac{1}{1-q} = \sum_\lambda \frac{1}{\mathrm{Eu}_q(\lambda)} = \sum_\lambda w_\lambda \frac{1}{\mathrm{Eu}_q(\emptyset)} \quad \/, \end{equation}
where $w_\lambda \in W$ is the permutation giving the partition $\lambda$. 
An illustration of the calculation of the Euler class is given in the Appendix. 
\begin{remark}\label{rmk:QH-Frobenius} Similar results hold in $\QH^*_{\T}(\Gr(k,n))$, the (equivariant) quantum cohomology 
ring of $\Gr(k,n)$. The product is defined by the condition that 
\[ (a \circ b, c )_{\mathrm{H}} = \sum_d \langle a, b,c \rangle_d q^d \/,\]
where $(\cdot , \cdot )_{\mathrm{H}}$ is the {\em classical} Poincar{\'e} pairing extended by $q$ linearity, 
and $\langle a,b,c \rangle_d$ are the (equivariant) cohomological GW invariants.
The equivariant quantum cohomology ring is graded, with $\deg q = n$. 
In particular, for any $\kappa_1, \kappa_2 \in H^*_{\T}(\Gr(k,n))$, 
\[ \langle \kappa_1, \kappa_2 \rangle_{\mathrm{H}} = \textrm{ coefficient of } [\pt_{k,n}] \textrm{ in } \kappa_1 \cdot \kappa_2 \/. \]
In particular, since $\bfe_\lambda^q  = \bfe_\lambda +  q A$, where 
$\deg A < \deg [\pt_{k;n}]$, it follows that 
\[ \langle 1, \bfe_\lambda^q \rangle_{\mathrm{H}} =1 = \langle 1, \bfe_\lambda^q \rangle \/, \]
where the latter is the pairing defined in the integrable system case. Then, using the Frobenius property, 
\[ \langle \bfe_\lambda^q, \bfe_\mu^q \rangle_{\mathrm{H}} = \langle 1, \bfe_\lambda^q \circ  \bfe_\mu^q \rangle_{\mathrm{H}} = \mathrm{Eu}_\lambda \delta_{\lambda,\mu} = \langle \bfe_\lambda^q, \bfe_\mu^q \rangle \/. \] 
This shows that the geometric and the integrable systems pairings are the same. 
\end{remark} 
Examples of quantum equivariant localization may be found in \Cref{sec:app-examples}.

\appendix

\section{A guiding example, mostly on $\Gr(1,2)$.}\label{sec:app-examples}
We work out below the simplest non-trivial example when $\G=\GL_2$.
In this case, the Yang-Baxter module is  
\[\bbV_2 = \K_T(\Gr(0;2)) \oplus \K_T(\Gr(1;2)) \oplus \K_T(\Gr(2;2)) \/,\]
and as a $\T \simeq (\C^*)^2$-module, $\C^2$ has a weight decomposition
$\C^2 = \C_{\ve_1} \oplus \C_{\ve_2}$. 

In what follows we will illustrate both the geometry and the graphical calculus giving the 
entries $t_{ij}$ of the monodromy matrix $T(y)$, and  
we will work out the algorithm calculating the Bethe vectors on $\QK_{\T}(\Gr(1,2)) \simeq \QK_{\T}(\bP^1)$. We also work out some examples of the Frobenius pairing involving the 
Bethe vectors $\bfe_\lambda^q$.  

\subsection{The Schubert classes in $\bbV_2$} We will use superscripts to 
indicate which Grassmannian contains a
Schubert variety or a class, for example, $X_\lambda^i \subset \Gr(i,2)$.
The Grassmannians $\Gr(0;2)$ and $\Gr(2;2)$ each have a single Schubert 
variety: $X^0_\emptyset = 0$, 
respectively $X^2_{\emptyset} = \langle e_1, e_2 \rangle$. The Schubert varieties on $\Gr(1,2)$ are
\[ X^1_\Box= \langle e_2 \rangle \/; \quad X^1_\emptyset = \Gr(1;2)=\bP^1 \/. \]
The Yang-Baxter module $\bbV_2$ has 
the spin/Schubert  basis 
\[ v_\emptyset^0=v_1 \otimes v_1 = \cO^0_\emptyset, \quad v_\emptyset^1=v_0 \otimes v_1 =\cO^1_\emptyset, \quad v_\Box^1=v_1 \otimes v_0=\cO^1_\Box, \quad v_\emptyset^2=v_0 \otimes v_0=\cO^2_\emptyset \/. \]
The quotient bundle on $\Gr(0;2)$ is $\cQ_{2}= \C^2$, while
the quotient bundle $\cQ_0$ on $\Gr(2;2)$ is trivial of rank $0$. 
The classes 
$\lambda_y(\cQ_i^\vee)$ (for $i=0,1,2$) are given by:
\begin{equation}\label{E:lyQ02} \lambda_y(\cQ_2^\vee) = (1+y/\ve_1)(1+y/\ve_2) \cO^0_{\emptyset}\/;\quad \lambda_y(\cQ_0^\vee) = 1 \/;\end{equation}
\begin{equation}\label{E:lyQ1}  \lambda_y(\cQ_1^\vee) = (1+y/\ve_2)\cO^1_\emptyset - (y/\ve_2) \cO^1_\Box \/.\end{equation}

\subsection{The monodromy matrix}
We describe next the monodromy matrix 
$T(y)= \begin{pmatrix} t_{00}(y) & t_{01}(y) \\ t_{10}(y) & t_{11}(y) \end{pmatrix}$ using 
both \Cref{thm:QK=YB} and the graphical calculus. We keep using the notation
$\tau_{ij}$ for the convolution operators, and $t_{ij}$ for the entries of the monodromy matrix.

We start with the diagonal entries. Denote by $\tau(y) = \tau_{00}(y) + q \tau_{11}(y)$.
Since there are no quantum corrections in $\QK_{\T}(\Gr(0;2))$ and 
$\QK_{\T}(\Gr(2;2))$ it follows that
\[ \tau(y)|_{\QK_{\T}(\Gr(i;2))}= \lambda_y(\cQ_{2-i}^\vee) \/;\]
the formulae are given in \eqref{E:lyQ02}.
The formula for $\lambda_y(\cQ_1^\vee)\star \cO^1_\emptyset $ is given in \eqref{E:lyQ1}; using for example the `quantum=classical' statement one calculates that 
\begin{equation}\label{E:lyQ1Box} \lambda_y(\cQ_1^\vee) \star \cO^1_\Box = -(qy/\ve_1) \cO^1_\emptyset + (1+y/\ve_1) \cO^1_\Box \/. \end{equation}
On the other side, according to the graphical calculus we have that:
\begin{gather}
t(y).v_0 \otimes v_1 =\Connecting{0}{1}{}{1}\hspace{-.55cm}\Avoiding{}{0}{0}{0}v_0 \otimes v_1\;+\Avoiding{0}{1}{}{0}\hspace{-.55cm}\Avoiding{}{0}{0}{1} v_1 \otimes v_0 =(1+y/\ve_2)v_0 \otimes v_1-y/\ve_2\, v_1 \otimes v_0\\
t(y).v_1 \otimes v_0=\Avoiding{0}{0}{}{0}\hspace{-.55cm}\Connecting{}{1}{0}{1} v_1 \otimes v_0\;+\Avoiding{1}{0}{}{1}\hspace{-.55cm}\Avoiding{}{1}{1}{0}\!v_0 \otimes v_1=(1+y/\ve_1)v_1 \otimes v_0-q\,y/\ve_1\,v_0 \otimes v_1
\end{gather}
The first equality matches \eqref{E:lyQ1} and the second \eqref{E:lyQ1Box}. 
We leave it to the reader to match the graphical calculus for $t(y)$ 
restricted to $\QK_{\T}(\Gr(0,2))$ and 
$\QK_{\T}(\Gr(2,2))$ with the expressions from \eqref{E:lyQ02}. 

We now calculate the values of the off-diagonal operators on the spin basis. We start with $t_{01}(y)$: 
\[ \tau_{01}(y) \cO^0_\emptyset=0\/;\]
\[ \tau_{01}(y) \cO^1_\emptyset= (p_1)_* p_2^*(\lambda_y(\cQ_1^\vee) \cdot \cO^1_\emptyset) = \cO^0_\emptyset \/;\]
\[ \tau_{01}(y) \cO^1_\Box= (p_1)_* p_2^*(\lambda_y(\cQ_1^\vee) \cdot \cO^1_\Box) = (p_1)_* p_2^*((1+y/\ve_1) \cO^1_\Box)
= (1+y/\ve_1)\cO^0_\emptyset \/;\]
\[ \tau_{01}(y) \cO^2_\emptyset = (p_1)_* p_2^*(\lambda_y(\cQ_0^\vee) \cdot \cO^2_\emptyset) = \cO^1_\emptyset \/. \]
Using the graphical calculus, we find for $t_{01}(y)$:
\[t_{01}(y)v_1 \otimes v_1 =0 \/;\]
\[
t_{01}(y)v_0 \otimes v_1=\Avoiding{1}{1}{}{1}\hspace{-.55cm}\Avoiding{}{0}{0}{1}v_1\otimes v_1=v_1 \otimes v_1\/;\]
\[ t_{01}(y)v_1 \otimes v_0=\Avoiding{1}{0}{}{1}\hspace{-.55cm}\Connecting{}{1}{0}{1}v_1\otimes v_1=\left(1+y/\ve_1\right)v_1\otimes v_1 \/;\]
\[ t_{01}(y)v_0 \otimes v_0 =\Avoiding{1}{0}{}{1}\hspace{-.55cm}\Avoiding{}{0}{0}{0}v_0 \otimes v_1 = v_0 \otimes v_1 \/.\]

We continue with $\tau_{10}(y)$. 
\[ \begin{split} \tau_{10}(y) \cO^0_\emptyset & = \lambda_y (\mathcal{Q}_{1}^\vee) \cdot (p_2)_*(p_1)^*(\cO^0_\emptyset) - (p_2)_* p_1^*(\lambda_y (\mathcal{Q}_{2}^\vee) \cdot \cO^0_\emptyset) = - y/\ve_1 (1+y/\ve_2) \cO^1_\emptyset - (y/\ve_2) \cO^1_\Box \/. \end{split}\]
\[ \begin{split} t_{10}(y) \cO^1_\emptyset & = \lambda_y (\mathcal{Q}_{0}^\vee) \cdot (p_2)_*(p_1)^*(\cO_1^\emptyset) - (p_2)_* p_1^*(\lambda_y (\mathcal{Q}_{1}^\vee) \cdot \cO_1^\emptyset) =0 \/. \end{split} \]
\[ \begin{split} t_{10}(y) \cO^1_\Box & = \lambda_y (\mathcal{Q}_{0}^\vee) \cdot (p_2)_*(p_1)^*(\cO_1^\Box) - (p_2)_* p_1^*(\lambda_y (\mathcal{Q}_{1}^\vee) \cdot \cO_1^\Box) = -(y/\ve_1) \cO^2_\emptyset \/. \end{split} \] 
\[ t_{10}(y) \cO^2_\emptyset =0 \/. \]
The graphical calculus yields:
\[  
t_{10}(y)v_1\otimes v_1=\Connecting{0}{1}{}{1}\hspace{-.55cm}\Avoiding{}{1}{1}{0}v_0 \otimes v_1+
\Avoiding{0}{1}{}{0}\hspace{-.55cm}\Avoiding{}{1}{1}{1}v_1 \otimes v_0=
-y/\ve_1\left(1+y/\ve_2 \right)v_0 \otimes v_1-(y/\ve_2)\,v_1 \otimes v_0\;
\]
\[ t_{10}(y)v_0\otimes v_1=0 \/;\]
\[
t_{10}(y)v_1 \otimes v_0=\Avoiding{0}{0}{}{0}\hspace{-.55cm}\Avoiding{}{1}{1}{0}v_0\otimes v_0=-(y/\ve_1)\,v_0\otimes v_0\;.
\]
\[ t_{10}(y)v_0\otimes v_0=0 \/. \]

\subsection{Bethe vectors} The Bethe vectors for  $\QK_{\T}(\Gr(0,2))$ and  $\QK_{\T}(\Gr(2,2))$ are equal to the Schubert classes. The nontrivial calculation is that of the 
Bethe vectors for $\QK_{\T}(\Gr(1,2))$. To ease notation, we will remove the superscripts
from the notation of Schubert classes. 
For $\Gr(1,2)$ there is a single Bethe Ansatz equation:
\[ (1-x/\ve_1)(1-x/\ve_2) - q = 0 \/, \]
with roots 
\[
x_{\pm}:= \frac{\ve_1}{2} + \frac{\ve_2}{2} \pm  \frac{\sqrt{4q\ve_1\ve_2 + (\ve_1-\ve_2)^2}}{2}
\]
At $q=0$, the root $x_+=\ve_1$ corresponds to $\lambda=\emptyset$ and the root $x_{-}=\ve_2$ to $\lambda=(1)$. The `off-shell' Bethe vector is
\begin{equation*}
\tau_{10}(-x)\cO^0_\emptyset =\frac{x}{\ve_1}\left(1-\frac{x}{\ve_2}\right)\cO_\emptyset+\frac{x}{\ve_2}\cO_\Box \/.
\end{equation*}
The (`on-shell') Bethe vectors are obtained by specializing $x$ to be a root of the Bethe Ansatz equation:
\[ \bfe_\emptyset^q=\Phi(\be_{01})  = \tau_{10}(-x_+)\cO_\emptyset^0 =  \frac{x_+}{\ve_1} (1-\frac{x_+}{\ve_2}) \cO_\emptyset + \frac{x_+}{\ve_2} \cO_\Box \]
\[ \bfe_{(1)}^q=\Phi(\be_{10}) = \tau_{10}(-x_{-})\cO_\emptyset^0 =  \frac{x_{-}}{\ve_1} (1-\frac{x_{-}}{\ve_2}) \cO_\emptyset + \frac{x_{-}}{\ve_2} \cO_\Box \]
One may check directly that if $q=0$ the Bethe vectors are 
precisely the classes of the corresponding 
torus fixed points in $\Gr(1,2)$. 

As a reality check, using that in the quantum ring $\QK_T(\Gr(1,2))$, 
\begin{equation}\label{E:O11}  \cO^1_{(1)} \star \cO^1_{(1)} = q (\ve_2/\ve_1) \cO_\emptyset^1 - (\ve_2/\ve_1) \cO^1_{(1)}  + \cO^1_{(1)} \end{equation}
an algebra calculation gives that 
\[ t_{10}(-x_+) \star t_{10}(-x_{-}) = 0 \/, \]
verifying the orthogonality property of the Bethe vectors. 

\subsection{Quantum localization} By definition of the Frobenius pairing
\[ (\cO^\Box, \cO_\Box)_{\QK} = \frac{q}{1-q} \/; \quad (\cO_\emptyset, \cO_\lambda)_{\QK} = \frac{1}{1-q} \quad \forall \lambda \/. \]
We may also calculate the quantum localizations:
\[ (\cO_\emptyset, \bfe_\emptyset^q )_{\QK} = \frac{\frac{x_+}{\ve_1} (1-\frac{x_+}{\ve_2})}{1-q}+\frac{ \frac{x_+}{\ve_2}}{1-q} = \frac{1-q}{1-q}=1=G_0(1-x_+|1-\ve^{-1}) \/;\]
\[ (\cO_\emptyset, \bfe_\Box^q )_{\QK} = \frac{\frac{x_-}{\ve_1} (1-\frac{x_-}{\ve_2})}{1-q}+\frac{ \frac{x_-}{\ve_2}}{1-q} = \frac{1-q}{1-q}=1=G_0(1-x_{-}|1-\ve^{-1}) \/.\]
We now turn to the quantum localizations of $\cO_\Box$. We have:
\[ \begin{split} (\cO_\Box, \bfe_\emptyset^q )_{\QK} & = \frac{\frac{x_+}{\ve_1} (1-\frac{x_+}{\ve_2})}{1-q}+ \frac{x_+}{\ve_2}(\cO_\Box, \cO_\Box)_{\QK} \\ 
& =  \frac{\frac{x_+}{\ve_1} (1-\frac{x_+}{\ve_2})}{1-q}+\frac{x_+}{\ve_2} \cdot \frac{q \ve_2/\ve_1-\ve_2/\ve_1+1}{1-q} \\ &
= \frac{1-q- x_+/\ve_2\cdot \ve_2/\ve_1(1-q)}{1-q}=1-x_+/\ve_1 \\
& =G_1(1-x_+|1-\ve^{-1})
 \/; \end{split}\]
\[ \begin{split} (\cO_\Box, \bfe_\Box^q )_{\QK} & = \frac{\frac{x_-}{\ve_1} (1-\frac{x_-}{\ve_2})}{1-q}+ \frac{x_-}{\ve_2}(\cO_\Box, \cO_\Box)_{\QK} \\ 
& =  \frac{\frac{x_-}{\ve_1} (1-\frac{x_-}{\ve_2})}{1-q}+\frac{x_-}{\ve_2} \cdot \frac{q \ve_2/\ve_1-\ve_2/\ve_1+1}{1-q} \\ &
= \frac{1-q- x_-/\ve_2\cdot \ve_2/\ve_1(1-q)}{1-q}=1-x_-/\ve_1 \\ &
=G_1(1-x_{-}|1-\ve^{-1})\/; \end{split}\]
We now calculate the quantum Euler pairing 
$(\bfe_\emptyset^q,\bfe_\emptyset^q )_{\QK}$. 
From the expansion into Schubert classes we obtain 
\[ \begin{split} (\bfe_\emptyset^q,\bfe_\emptyset^q )_{\QK} & = \left( \frac{x_+}{\ve_1} (1-\frac{x_+}{\ve_2}) \cO_\emptyset + \frac{x_+}{\ve_2} \cO_{(1)}, \bfe_\emptyset^q \right)_{\QK} \\
& = \frac{x_+}{\ve_1} (1-\frac{x_+}{\ve_2}) + \frac{x_+}{\ve_2} (\cO_{(1)}, \bfe_\emptyset^q )_{\QK} \\ & = \frac{x_+}{\ve_1} (1-\frac{x_+}{\ve_2}) + \frac{x_+}{\ve_2} (1-\frac{x_+}{\ve_1}) \\
& =  \frac{x_+}{\ve_1} +  \frac{x_+}{\ve_2} - \frac{2 (x_+)^2}{\ve_1 \ve_2} \\
& = 1-q - \frac{(x_+)^2}{\ve_1 \ve_2} \/. 
\end{split}
\]
From the $W$-equivariance of the quantum pairing we obtain
\[ (\bfe_\Box^q,\bfe_\Box^q )_{\QK} = s^L.(\bfe_\emptyset^q,\bfe_\emptyset^q )_{\QK} = 1-q  - \frac{(x_{-})^2}{\ve_1 \ve_2} \/.\]
Note that if $q=0$, $x_+ = \ve_1$ and $x_-=\ve_2$, giving that 
\[ (\cO_\Box, \bfe_\emptyset )_{\K} =0 \/, \quad
(\cO_\Box, \bfe_\Box )_{\K} =1-\ve_2/\ve_1 , \quad  
(\bfe_\emptyset, \bfe_\emptyset )_{\K} =1-\ve_1/\ve_2 \/,\] 
consistent with the classical case. (We denoted the classical pairing by $( \cdot , \cdot)_{\K}$.)

We illustrate next the calculation of two quantum characters (cf.~\Cref{cor:qAB}).
First, we consider $(1, 1)_{\QK}$ in $\QK_{\T}(\bP^1)$; see \eqref{E:AB-1}. In this case, the equality states:
\[ \frac{1}{1-q} = \frac{1}{1-q - \frac{(x_+)^2}{\ve_1 \ve_2}}+ \frac{1}{1-q - \frac{(x_-)^2}{\ve_2 \ve_1}} \/. \]
Then a direct algebra check gives that the right hand side is indeed $1/(1-q)$, using that, from the Bethe ansatz equation,
\begin{equation}\label{E:BA-cons} \frac{x_{\pm}^2}{\ve_1 \ve_2} = q-1+\frac{x_{\pm}}{\ve_1}+\frac{x_{\pm}}{\ve_2} \/.\end{equation}

We now illustrate the calculation of the  quantum character of
$\det \mathcal{Q}_1$. To start, we have
\[ \lambda_y \mathcal{Q}_1 = (1+y \ve_2) \cO_1^\emptyset + y \ve_1 \cO_1^\Box \/,\]
thus
\[ \det \mathcal{Q}_1 = \ve_2 \cO_1^\emptyset + \ve_1 \cO_1^\Box \/.\]
The quantum localizations are:
\[ ( \det \mathcal{Q}_1, \bfe_\emptyset^q )_{|QK} = \ve_2 + \ve_1 (1-\frac{x_+}{\ve_1}) =  \ve_1+ \ve_2- x_+  \/;\] 
\[ ( \det \mathcal{Q}_1, \bfe_\Box^q )_{\QK} = \ve_1+\ve_2 - x_{-} \/. \]
Then the quantum character is
\[ \mathrm{qch}(\det \mathcal{Q})= \frac{\ve_1+ \ve_2- x_+ }{1-q  - \frac{(x_{+})^2}{\ve_1 \ve_2}}+
\frac{\ve_1+ \ve_2- x_{-} }{1-q  - \frac{(x_{-})^2}{\ve_1 \ve_2}} = \frac{\ve_1+\ve_2}{1-q} \/.\] 
The last equality follows again from \eqref{E:BA-cons}, and it confirms \Cref{cor:qAB} in this case.

\subsection{Action of the affine Weyl group} The Weyl group $W=S_2$
is generated by the reflection $s_1=s_\alpha$ corresponding to the root
$\alpha=\ve_1/\ve_2$. The extended affine Weyl group 
$\widetilde W=W\ltimes\Z^2$ has two presentations:
\[ \widetilde{W} = \langle s_1, \rho : s_1^2 =1\rangle = 
\langle s_1, t_{1}, t_{2} : s_1^2 = 1, t_2 = s_1 t_1 s_1 \rangle \/, \]
related by $t_1 = s_1\rho$ and $t_2=\rho s_1$.
The elements $t_i$ are the translations $t_{\ve_i}$. The affine Weyl group
$\Waff = \langle s_0 ,s_1 \rangle$ is a subgroup of $\widetilde{W}$,
with $s_0=\rho s_1\rho^{-1}$; then $\widetilde{W} = \Waff \ltimes \Z$, and $\rho$ is the cyclic generator of $\Z$.

We now describe the action of $\widetilde{W}$ on the $\YB$-module 
$\bbV_2=V_{0,2}\oplus V_{1,2}\oplus V_{2,2}$. Since the action preserves
each of the weight spaces $V_{i,2}$, we still use 
the notation of Schubert classes without superscripts.
 
The action on $V_{1,2}$ is given by 
\begin{gather*} 
s_1.\cO_\emptyset=v_{\emptyset}\,,\qquad s_1.\cO_\Box=\alpha \cO_\emptyset+(1-\alpha)\,\cO_\Box \/; \\
\rho.\cO_\emptyset=\cO_{\Box}\,,\qquad \rho.\cO_\Box=q \cO_\emptyset\,,
\end{gather*}
From this, one may derive the action of the translations. Alternatively, one may use the geometric interpretation and \eqref{E:lyQ1}:
\begin{gather*} t_1.\cO_\emptyset = \lambda_{-\ve_1}(\cQ_1^\vee)\star \cO_\emptyset = (1-\alpha) \cO_\emptyset + \alpha \cO_\Box \/; \qquad t_2.\cO_\emptyset = \lambda_{-\ve_2}(\cQ_1^\vee)\star \cO_\emptyset = \cO_\Box \/; \\
t_1.\cO_\Box = \lambda_{-\ve_1}(\cQ_1^\vee)\star \cO_\Box = q\cO_\emptyset,\qquad t_2.\cO_\Box=\lambda_{-\ve_2}(\cQ_1^\vee)\star \cO_\Box = \alpha^{-1}q\cO_\emptyset+(1-\alpha^{-1})\cO_\Box \/. \end{gather*}
The above formulae imply the following actions of the affine reflection $s_0=\rho s_1\rho^{-1}$:
\[s_0.\cO_\emptyset=\alpha^{-1}\cO_\emptyset+(1-\alpha^{-1})q^{-1}\cO_\Box\,,
\qquad\qquad s_0.\cO_\Box=\cO_\Box\/.\]
The action of $W$ on $V_{0,2}$ and $V_{2,2}$ is trivial while 
\[ \rho|V_{0,2}=q \textrm{ and }\rho|V_{2,2}=1 \/. \]

\section{$\beta$-calculus}\label{sec:beta-calc} 
In order to facilitate the comparison with the results from \cite{gorbounov2017quantum} we briefly summarize the change of conventions between the latter and the current work. 

In \cite{gorbounov2017quantum} the emphasis was on expressing the operators $t(z)$ and $\tilde t(z)$ of the integrable system in terms of Schubert classes and in this context the multiplicative formal group law
\[
z\oplus w=z+w+\beta\, zw,\qquad -1\leq \beta\leq 0
\]
is particularly convenient, where $z$ denotes the spectral parameter used in \cite{gorbounov2017quantum} and $\beta$ is chosen to be a real parameter with the values $\beta=0$ and $\beta=-1$ corresponding to the case of (quantum) cohomology and (quantum) K-theory, respectively. In the physical interpretation of the lattice models the parameter $\beta$ plays the role of a `coupling constant' or `interaction strength'.
In the current work the focus is instead on vector bundles and their classes. 

Let $E \to X$ be a vector bundle of rank $e$, with the Hirzebruch $\lambda_y$ class
\[ \lambda_y(E) = 1 +y E + y^2 \wedge^2 E + \ldots + y^e \wedge^e E \/.\]
Then the parameter $y$ and the equivariant parameters $\ve_i$ 
are related to the spectral parameter $z$ and the equivariant parameters $t_i$ in \cite{gorbounov2017quantum}\footnote{The variable $z$ was called $x$ and $t_i$ called $y_i$ in \cite{gorbounov2017quantum}} via
\begin{equation}\label{E:cov-eqparams} 
-y =1+ \beta z\;\quad\text{and} \quad \ve_i = 1+ \beta t_i \;.
\end{equation}
Inserting these variable transformations we have the following relationship between the tautological bundles and the transfer matrices\footnote{Our conventions of how to map 01-words to partitions in the current article differ also from the ones in \cite{gorbounov2017quantum}: swapping 0-letter and 1-letters one obtains via level-rank duality $t(y)$ from $\tilde t(y)$ and, thus, the geometric interpretations of $E(z)$ and $H(z)$ in \cite{gorbounov2017quantum} are opposite to the ones here.} considered in the current article and those in \cite{gorbounov2017quantum},
\begin{equation}
   E(z)=(-\beta)^{n-k} t^{(\beta)}(y) = \lambda_y^\beta(\mathcal{Q}^\vee)\quad\text{and}\quad
    H(z)=(-\beta)^{k} \tilde{t}^{(\beta)}(y) = \lambda_y^\beta(\mathcal{S})\;.
    \end{equation}
We also weight Schubert classes by
\begin{equation}\label{E:Schub-cov} \cO_\lambda \mapsto (-\beta)^{|\lambda|} \cO_\lambda \end{equation}    
Specializing $\beta=-1$ in these formulae we arrive at the conventions and operators used in the current work.

\bibliographystyle{halpha}
\bibliography{QKGrass.bib}
\end{document}